\documentclass[11pt,reqno]{amsart}
\usepackage{amsmath,amsfonts,amssymb,bbm}
\usepackage{a4wide}
\usepackage[active]{srcltx}
\usepackage{dsfont}
\usepackage{xfrac}
\usepackage{xcolor}
\usepackage{mathrsfs}
\usepackage{mathtools}
\usepackage{tikz}
\usepackage{cancel}
\usepackage{pgfplots}
\usetikzlibrary{intersections}
\usepackage{tikz-3dplot}
\usetikzlibrary{shadings,intersections,patterns}
\usepackage{pgfplots}
\usepgfplotslibrary{fillbetween}
\usetikzlibrary{patterns}
\usetikzlibrary{intersections, calc, angles}
\usepackage[outline]{contour}
\contourlength{0.05em}
\usepackage{makecell}
\usepackage{verbatim}
\usepackage[shortlabels]{enumitem}
\usepackage{float}
\usepackage[unicode,hypertexnames=false,colorlinks=true,linkcolor=blue,citecolor=blue]{hyperref}
\usepackage{geometry}

\usepackage[
backend=biber,
style=alphabetic,
sorting=nyt,
maxbibnames=50,
giveninits=true,
maxalphanames=10
]{biblatex}
\usepackage{csquotes}
\renewbibmacro{in:}{%
 \ifentrytype{article}{}{}}
\DeclareFieldFormat[article]{volume}{\textbf{#1}}
\DeclareFieldFormat[article]{title}{\textit{#1}}
\DeclareFieldFormat[article]{journaltitle}{#1}
\DeclareFieldFormat[article]{number}{\textbf{#1}}
\renewbibmacro*{volume+number+eid}{%
  \printfield{volume}%
\setunit*{\textbf{\addcolon}}
  \printfield{number}\setunit{\addcomma\space}%
  \printfield{eid}}

\usepackage{cfr-lm}
\usepackage[osf]{libertine}
\usepackage{hyperref}
\usepackage[capitalise, nameinlink, noabbrev]{cleveref} 
\hypersetup{colorlinks=true,
	linkcolor=blue,
	citecolor=red,
}

\hypersetup{colorlinks,breaklinks,
	linkcolor=[rgb]{0,0,0},
	citecolor=[rgb]{0,0,0},
	urlcolor=[rgb]{0,0,0}}

\newcommand{\addappendix}{%
  \section*{\appendixname}
  \counterwithin*{figure}{section}
  \stepcounter{section}
  \renewcommand{\thesection}{A}
  \renewcommand{\thefigure}{\thesection.\arabic{figure}}
}

\usepackage{esint}

\newcommand{\be}{\begin{equation}}
\newcommand{\ee}{\end{equation}}
\theoremstyle{plain}

\newtheorem{theorem}{Theorem}[section]
\newtheorem{proposition}[theorem]{Proposition}
\newtheorem{corollary}[theorem]{Corollary}
\newtheorem{lemma}[theorem]{Lemma}
\newtheorem{definition}[theorem]{Definition}

\newtheorem{remark}[theorem]{Remark}

\def\y{\lvert y\rvert^a}
\newcommand{\Gammab}{\mathbf{\Gamma}}
\newcommand{\R}{\mathbb{R}}

\newcommand{\N}{\mathbb{N}}
\newcommand{\vf}{\varphi}
\newcommand{\eps}{\varepsilon}

\newcommand{\HH}{\mathcal{H}}
\newcommand{\loc}{{{\text{loc}}}}

\newcommand{\bea}{\begin{equation*}\begin{aligned}}
		\newcommand{\eea}{\end{aligned}\end{equation*}}

\newcommand{\res}{\mathop{\hbox{\vrule height 7pt width .5pt depth 0pt
			\vrule height .5pt width 6pt depth 0pt}}\nolimits}
\newcommand{\ssubset}{\subset\joinrel\subset}

\DeclareMathOperator{\dist}{dist}

\DeclareMathOperator{\Tr}{Tr}
\DeclareMathOperator{\osc}{osc}
\DeclareMathOperator{\diver}{div}

\DeclareMathOperator{\sgn}{sgn}

\crefname{subsection}{subsection}{subsections}

\numberwithin{equation}{section}

\title[Generic regularity of free boundaries in the obstacle problem for the fractional Laplacian]{Generic regularity of free boundaries in the obstacle problem \\for the fractional Laplacian}
\keywords{Free boundaries, generic regularity, obstacle problems, fractional Laplacian, epiperimetric inequality}
\subjclass[2010]{35R35, 47G20}

\author[M. Carducci]{Matteo Carducci}\thanks{}
\address {Matteo Carducci \newline \indent
	Classe di Scienze, Scuola Normale Superiore \newline \indent
	Piazza dei Cavalieri 7, 56126 Pisa, Italy}
\email{\href{mailto:matteo.carducci@sns.it}{matteo.carducci@sns.it}}

\author[R. Colombo]{Roberto Colombo}\thanks{}
\address {Roberto Colombo \newline \indent
	EPFL SB MATH\newline \indent
	Institute of Mathematics, Station 8, CH-1015 Lausanne, Switzerland}
\email{\href{mailto:roberto.colombo@epfl.ch}{roberto.colombo@epfl.ch}}
\pgfplotsset{compat=1.18}
\begin{document}

\newgeometry{margin=2cm}
\begin{abstract} 
We establish generic regularity results of free boundaries for solutions of the obstacle problem for the fractional Laplacian $(-\Delta)^s$. 
We prove that, for almost every obstacle, the free boundary contains only regular points up to dimension $3$, for every $s\in(0,1)$. To do so, we extend some results on the fine structure of the free boundary to the case $s\in (0,1)$ and general non-zero obstacle, including a blow-up analysis at points with frequency $2m+2s$, and we prove new explicit uniform frequency gaps for solutions of the fractional obstacle problem.
\end{abstract}
\maketitle
\tableofcontents

\section{Introduction}\label{section-intro}
We consider solutions $v:\R^n\to \R$ of the fractional
obstacle problem 
 \be\label{def:soluzione-frac}
 \left\{
\begin{array}{rclll}
		(-\Delta)^s v&=&0 & \quad\mbox{in } \R^n\setminus \{v=\vf\},\\
		(-\Delta)^s v &\ge&0 & \quad\mbox{in } \R^n,\\
        v&\ge &\vf & \quad\mbox{in } \R^n,\\
            v(x)&\rightarrow& 0 & \quad\mbox{as } |x|\rightarrow \infty,
	\end{array}
 \right.
 \ee 
 where, for every $s\in(0,1)$, the operator $(-\Delta)^s$ is the fractional Laplacian, defined as $$ (-\Delta)^sv(x) :=c_{n,s}\,\mbox{P.V.}\int_{\R^n}\frac{v(x)-v(z)}{\lvert x-z\rvert^{n+2s}}\,dz
 $$ and $c_{n,s}>0$ 
 is a normalization constant. The obstacle $\vf:\R^{n}\to \R$ is such that \be\label{e:hypo-phi}\vf\in C^{k,\gamma}(\R^n)\quad \mbox{for some } k\ge 2,\, \gamma \in (0,1) \qquad\mbox{and}\qquad \{\varphi>0\}\ssubset \R^{n}.\ee

Problem \eqref{def:soluzione-frac} in the case $s=1/2$ is locally equivalent to the thin obstacle problem, first studied by Signorini in \cite{sig33, sig59} in connection with linear elasticity.
In the general case $s\in (0,1)$, it finds applications in several other contexts, such as the optimal stopping problem, interacting particle systems, and the study of semipermeable membranes; see \cite{dl76,m76,ct04,s15,cdm16}. 
The obstacle problem for the fractional Laplacian \eqref{def:soluzione-frac} was intensely studied in the past years by the mathematical community; see
\cite{caf79,ac04,sil07,acs08,css08,gp09,kps15,ds16,gps16,fs16,gpps17,fs18,csv20,fs22,sy22:7/2,sy22,sy23,fj21,fr21,fs24,car24,cv24}. The reader may also consult the surveys \cite{dasa18,survey} and the books \cite{psu12,frro24}.

The existence and uniqueness of a solution of problem \eqref{def:soluzione-frac} can be proved by minimizing
the fractional Dirichlet energy in the class of functions which live above the obstacle.
 Regarding the regularity of a solution $v$, 
 defining the contact set as $$\Lambda(v):=\{x\in\R^n:v(x)=\vf(x)\},$$ we see directly from \eqref{def:soluzione-frac} and \eqref{e:hypo-phi}, that $v$ is regular in both the interiors of $\Lambda(v)$ and of its complementary $\R^{n}\setminus \Lambda(v)$. 
 Less obvious is the behavior of the solution across the free boundary $$\Gamma(v):=\partial\Lambda(v).$$
 In \cite{ac04,acs08,css08} the authors proved the optimal $C^{1,s}_{\loc}$-regularity of the solution, as well as the following splitting of the free boundary into regular and degenerate points: 
$$\Gamma(v)=\text{Reg}(v)\cup \text{Deg}(v).$$ Here, the sets of regular and degenerate points are given by
\begin{gather*}
    \text{Reg}(v):=\left\{x\in \Gamma(v):0<cr^{1+s}\le \sup_{B_r(x)}(v-\vf)\le Cr^{1+s}\quad\mbox{for every 
}r\in(0,r_0)\right\},\\
    \text{Deg}(v):=\left\{x\in \Gamma(v):0\le \sup_{B_r(x)}(v-\vf)\le Cr^{2}\quad\mbox{for every } r\in(0,r_0)\right\}.
\end{gather*}
The set $\text{Reg}(v)$ is an open $C^{1,\alpha}$-manifold of dimension $(n-1)$ inside the free boundary. In addition, if the obstacle is $C^\infty$, the set $\text{Reg}(v)$ is a $C^\infty$-manifold (see \cite{kps15,ds16,jn17,krs19}). It is also known that $\text{Deg}(v)$ can be non-empty, and even of the same dimension as the whole free boundary (see \cite{gp09, fr21}).

In recent years, a growing interest has concerned the question of \textit{generic regularity}, which, roughly speaking, refers to the presence of singularities in the free boundary being somewhat rare. Investigations in this direction have also found very recently application in the study of measure minimizers of interaction energies with confining potentials (see \cite{rs24} and the forthcoming \cite{cf24}).
In the context of the classical obstacle problem ($s=1$), the generic regularity of the free boundary was proved in \cite{mon03} in dimension $2$ and then in the celebrated \cite{frs20} up to dimension $4$. The same result was obtained for problem \eqref{def:soluzione-frac} in \cite{fr21} 
in dimension $1$, for every $s\in (0,1)$. Later, in \cite{ft23}, in the case $s=1/2$ and zero obstacle, generic regularity was proved up to dimension $3$.

\subsection{The main result}
In this paper we improve the generic regularity result of \cite{fr21} and extend the one in \cite{ft23} to the case $s\in (0,1)$ and for a general obstacle. 
Similarly to the aforementioned works (\cite{mon03,frs20,fr21,ft23}), we use a measure theoretic notion of genericity (in the sense of \textit{prevalence}, see \cite{hsy92,yor05}),
that has also been used to investigate similar questions in other free boundary problems (see \cite{cms23a,cms23b,fy23}). More precisely, we say that a property holds for \lq\lq almost every'' solution to the fractional obstacle problem if, for any family of solutions 
$v=v(x,t):\R^{n}\times[0,1]\to\R$ with obstacle $\vf-t$,  for $t\in[0,1]$, 
 \be\label{def:soluzione-frac-famiglia}
     \left\{
\begin{array}{rclll}
    (-\Delta)^s v(\cdot,t)&=&0 &\quad \mbox{in } \R^n\setminus \{v(\cdot,t)=\vf-t\},\\
		(-\Delta)^s v(\cdot,t)&\ge&0 & \quad\mbox{in } \R^n,\\
        v(\cdot,t)&\ge& \vf-t & \quad\mbox{in } \R^n,\\
        v(x,t)&\rightarrow& 0& \quad\mbox{as } |x|\to+\infty,
 \end{array}
 \right.
 \ee 
then, the property holds for almost every $t\in [0,1]$.
In this context, our main result is the following.
\begin{theorem}\label{thm:mainthm} Let $v:\R^{n}\times[0,1]\to \R$ be a family of solutions of \eqref{def:soluzione-frac-famiglia}, with obstacle $\varphi$ satisfying \eqref{e:hypo-phi}, with $k=4$ and $\gamma\in(a^-,1]$, where $a:=1-2s$. Then, there is $\alpha>0$ depending only on $n$, $s$ and $\gamma$ such that, 
for almost every $t\in [0,1]$, we have:
\begin{itemize}
\smallskip
    \item if $n\le 3$, $\text{Deg}(v(\cdot,t))=\emptyset$;
\smallskip
    \item if $n\ge4$, $\text{dim}_{\HH}(\text{Deg}(v(\cdot,t)))\le n-3-\alpha$.
\end{itemize}

\end{theorem}
When the obstacle is $C^\infty$, by the results in \cite{jn17,krs19}, we get the following corollary (again, in the sense of prevalence).
\begin{corollary}\label{cor:main}
    Let $v:\R^{n}\times [0,1]\to \R$ be a family of solutions of \eqref{def:soluzione-frac-famiglia}, with obstacle $\vf \in C^{\infty}(\R^{n})$ satisfying $\{\vf >0\}\ssubset \R^{n}$. Then, if $n\le 3$, the free boundary $\Gamma(v(\cdot, t))$ is a $C^{\infty}$-manifold of dimension $n-1$, for almost every $t\in[0,1]$. 
\end{corollary}
\cref{thm:mainthm} and \cref{cor:main} provide an improvement of \cite[Theorem 1.5]{fr21} for every $s\in (0,1)$. 
\subsection{Generic regularity for the extended problem} In order to study the non-local problem \eqref{def:soluzione-frac}, we will reduce to a local problem in $\R^{n+1}$ via the celebrated Caffarelli-Silvestre extension procedure (see \cite{cs07}). Using the notation $X=(x,y)\in \R^{n}\times \R$ for a point in $\R^{n+1}$, and calling $a:=1-2s \in (-1,1),$ we consider the degenerate/singular elliptic operator
$$L_{a}u:=\diver_{x,y}(|y|^{a}\nabla_{x,y}u),\quad u:\R^{n+1}\to \R.$$
Given a bounded function $v:\R^{n}\rightarrow \mathbb{R}$, we consider its standard even $L_a$-harmonic extension $u:\R^{n+1}\to \R$, satisfying 
\begin{equation}\label{eq:La-harmonic-extension}
    \left\{
    \begin{array}{rclll}
        L_{a}u&=&0&\quad \mbox{in } \R^{n+1}\setminus \left(\R^{n}\times \left\{0\right\}\right),\\
        u(\cdot,0)&=& v&\quad \mbox{in } \R^{n},\\
        u(x,y)&=&u(x,-y)&\quad \mbox{for } (x,y)\in\R^{n+1},
    \end{array}
    \right.
\end{equation}
which is obtained by convolution as $u(\cdot,y)=P_{a}(\cdot,y)*v$ with the corresponding Poisson kernel (see \cite{cs07,frro24})
\begin{equation}\label{def:Poissonkernel}
    P_{a}(x,y):= C_{n,a}\frac{|y|^{1-2a}}{(|x|^{2}+|y|^{2})^{\frac{n+1-a}{2}}}.
\end{equation}
Using such an extension, one gets the following local formula (in one extra dimension) for the non-local operator $(-\Delta)^{s}$ in $\R^{n}$:
\begin{equation}\label{formula:representation_fractional_lap_extension}
    (-\Delta)^{s}v(x)= -\lim_{y\downarrow 0}y^{a}\partial_{y}u(x,y).
\end{equation}
In particular, if $B_{1}$ is the unit ball in $\R^{n+1}$ and $B'_{1}:= B_{1}\cap \{y=0\}$, given $v:\R^{n}\to \R$ a solution to \eqref{def:soluzione-frac} for some obstacle $\varphi$ satisfying \eqref{e:hypo-phi}, its extension $u$ is a solution of
\be \label{def:soluzione-estesa}
\left\{
\begin{array}{rclll}
            -L_a u&=&0 &\quad \mbox{in } B_{1}\setminus \left(\{u(\cdot,0)=\vf\}\times \{0\}\right),\\
            -L_a u&\ge&0 &\quad \mbox{in } B_{1},\\
		u(x,0)&\ge& \varphi(x) & \quad\mbox{for }(x,0)\in B'_{1},\\
		u(x,y)&=& u(x,-y) &\quad\mbox{for } (x,y) \in B_{1}.\\
\end{array}
\right.
\ee
This fact is easily obtained noticing that \eqref{eq:La-harmonic-extension} implies
\begin{equation}\label{formula:y-derivative-on-thin-space}
         L_{a}u = 2\left(\lim_{y \downarrow 0}y^{a}\partial_{y}u(\cdot,y)\right)\mathcal{H}^{n}\res B_{1}'.
\end{equation}
Problem \eqref{def:soluzione-estesa} is called the \textit{thin obstacle problem} for the operator $L_{a}$, since now the obstacle lives in a \lq\lq thin'' domain $\{y=0\}$, of codimension $1$.
When we consider a solution $u:B_{1}\to \R$ of the extended problem \eqref{def:soluzione-estesa}, we denote the contact set and the free boundary as
$$\Lambda(u):= \left\{u(\cdot,0)=\varphi\right\}\times \left\{0\right\},\qquad \Gamma(u):= \partial'\Gamma(u),$$
where $\partial'$ stands for the boundary in the relative topology of $\R^{n}\times \left\{0\right\}$.

Thanks to the extension procedure we just described, we have reduced to proving a generic regularity result for the solutions of the local extended problem \eqref{def:soluzione-estesa}. 
For the notion of genericity of \eqref{def:soluzione-estesa}, as done in \cite{fr21, ft23}, we consider a monotone increasing family $u:B_{1}\times [0,1]\rightarrow \mathbb{R}$ such that $u(\cdot,t)$ solves \eqref{def:soluzione-estesa} for every $t\in [0,1]$ and the following conditions hold:
    \begin{equation}\label{def:soluzione-famiglia}
        \left\{
        \begin{array}{rclll}
            \lVert u(\cdot,t)\rVert_{C^{2s}(B_{1})}&\le& 1,\\
            u(\cdot,t')-u(\cdot,t)&\ge& 0&\quad \mbox{in } B_{1},\\
            u(\cdot,t')-u(\cdot,t)&\ge& t'-t&\quad \mbox{in } \partial B_{1} \cap \left\{|y|\ge 1/2\right\},
        \end{array}
        \right. \qquad \mbox{for every } \ -1\le t<t'\le 1.
    \end{equation}
In the extended framework we will prove the following theorem.
\begin{theorem}\label{thm:main_esteso}
    Let $u:B_{1}\times [0,1]\to \R$ be a family of solutions of \eqref{def:soluzione-estesa}, \eqref{def:soluzione-famiglia}, with obstacle $\varphi$ satisfying \eqref{e:hypo-phi}, with $k= 4$ and $\gamma \in(a^-,1]$. Then, there is $\alpha>0$ depending only on $n$, $s$ and $\gamma$ such that, for almost every $t\in [0,1]$, we have:
\begin{itemize}
\smallskip
    \item if $n\le 3$, $\text{Deg}(u(\cdot,t))=\emptyset$;
\smallskip
    \item if $n\ge4$, $\text{dim}_{\HH}(\text{Deg}(u(\cdot,t)))\le n-3-\alpha$.
\end{itemize}
\end{theorem}
In particular, when $s=1/2$ and the obstacle is zero, we recover the main result from \cite[Theorem 1.2]{ft23}.
We also notice that when the obstacle $\vf$ is $C^\infty$, \cref{thm:main_esteso} and \cite{jn17,krs19} give the generic smoothness of the free boundary up to dimension $3$.

We point out that, although the conclusions of \cref{thm:mainthm} and \cref{thm:main_esteso} are the same for all $s\in (0,1)$, a separate analysis of the two cases $s\le1/2$ and $s>1/2$ is required in several steps of their proof (see \Cref{outline} for more details).   
\subsection{Fine structure of the free boundary and points with frequency $\boldsymbol{2m+2s}$}\label{subsec:2m+2sIntro}
As it is well-known (see \cite{acs08,css08,gr19}), for a solution $u:B_{1}\to \R$ of \eqref{def:soluzione-estesa}, points $X_{0}\in \Lambda(u)$ in the contact set can be classified according to the value of their \textit{frequency}, which gives the asymptotic rate of detachment of $u$ from the obstacle $\vf$ around $X_{0}$, whenever strictly less than the regularity of $\vf$. More precisely, the frequency at $X_{0}$ corresponds to the homogeneity of blow-ups, i.e.~global solutions of \eqref{def:soluzione-estesa} with zero obstacle obtained as the limit of a sequence of rescalings of $\widetilde u^{X_0}$, a suitable extension of $u(\cdot, 0)-\vf$ to $\R^{n+1}$ (see \eqref{def-normalization-solution}). We refer the reader to \cref{section-prelim}, and in particular to \cref{generalized-Almgren} for the precise statements.
An important role is then played by the set of admissible frequencies in dimension $n+1$, namely
\begin{equation}\label{def:admissible-frequencies}
    \mathcal{A}_{n,s}:= \{\lambda \in \R: \text{there is a non-zero $\lambda$-homogeneous solution of \eqref{def:soluzione-estesa} with zero obstacle}\}.
\end{equation}
Dimension $2$ homogeneous solutions are explicitly characterized (see e.g.~\cite[Proposition A.1]{fs18}), and in fact it is known that \be\label{dim1}\mathcal{A}_{1,s}=\{2m\}_{m\in\N}\cup\{2m+2s\}_{m\in\N}\cup\{2m+1+s\}_{m\in\N}\subset \mathcal{A}_{n,s}.\ee

One of the main lines of investigation in the study of the thin obstacle problem \eqref{def:soluzione-estesa} regards the analysis of the contact set $\Lambda_{\lambda}(u)$ and the free boundary $\Gamma_{\lambda}(u)$ corresponding to admissible frequencies $\lambda\in\mathcal{A}_{n,s}$. 
The following are some of the main known results in this direction. 
\begin{itemize}
\item At points with frequency $1+s$, the blow-up is unique and coincides, up to a rotation, with a positive multiple of an even $L_{a}$-harmonic extension of $(x_1)_+^{1+s}$ to $\R^{n+1}$. Moreover, the rate of convergence to the blow-up is polynomial. As a consequence, $\Gamma_{1+s}(u)=\Lambda_{1+s}(u)=\text{Reg}(u)$ is an open $C^{1,\alpha}$-manifold of dimension $(n-1)$ inside the free boundary. We refer to \cite{acs08,css08,gps16,fs16,gpps17,csv20,car24}. 
\smallskip
\item At points with frequency $2m$, the blow-up is unique and it is a $2m$-homogeneous $L_{a}$-harmonic polynomial, non-negative on the thin space. The rate of convergence to the blow-up is logarithmic. As a consequence, $\Gamma_{2m}(u)=\Lambda_{2m}(u)$ is contained in the union of at most countably many $C^{1,\log}$-manifolds of dimension less than or equal to $n-1$.
We refer to \cite{gp09,gr19,csv20,car24}.
\smallskip
\item For $s=1/2$, at points with frequency $2m+1$, the blow-up is unique and it is a $(2m+1)$-homogeneous polynomial, which vanishes on the thin space. The rate of convergence to the blow-up is polynomial. As a consequence, $\Lambda_{2m}(u)$ is contained in the union of at most countably many $C^{1,\alpha}$-manifolds of dimension less than or equal to $n-1$. We refer to \cite{fs18,frs20,survey,sy23,cv24}.
\smallskip
\item For other interesting results concerning the rectifiability of the free boundary and points with frequencies different from $1+s, 2m$ and $2m+2s$, we refer to \cite{fs18,csv21,fs22,sy22:7/2,fs24}. 
\end{itemize}
\medskip

As a complementary result of this paper, we extend to the general case $s\in (0,1)$ the study of points with frequency $2m+2s$. We perform this analysis following the strategy employed in the very recent \cite{cv24}, which is based on the proof and application of an epiperimetric inequality for the Weiss energy corresponding to the frequency $2m+2s$ (see \cref{thm:epi}). Besides its own interest, a careful blow-up analysis at such points (when $m=1$) is in fact needed in the proof of \cref{thm:main_esteso}. 

First of all, we show that the set
\be\label{eq:def-p-2m+2s}\begin{aligned}
 \mathcal{P}_{2m+2s}:=\Big\{p: \ & L_a p=0\ \mbox{ in } \{y\neq 0\},\ -L_a p\ge0\ \mbox{ in } \R^{n+1},
 \nabla p\cdot X=(2m+2s)p,\\& p\equiv 0\ \mbox{ on } \{y=0\},\ p(x,y)=p(x,-y), \  p(x,y)=|y|^{2s} q(x,y), \ q\mbox{ polynomial}\Big\}.
\end{aligned}\ee coincides with the class of all $(2m+2s)$-homogeneous functions in $\R^{n+1}$ solving globally the extended problem (\ref{def:soluzione-estesa}) with zero obstacle.
Then, we prove the uniqueness of blow-up limits with a polynomial rate of convergence at points with frequency $2m+2s$.
\begin{theorem}[Uniqueness of the blow-up and rate of convergence at $(2m+2s)$-frequency points]\label{prop:rate} Let $u$ be a solution of \eqref{def:soluzione-estesa}, with obstacle $\vf$ satisfying \eqref{e:hypo-phi}.
Suppose that $0 \in\Lambda_{2m+2s}(u)$ with $2m+2s< k+\gamma$, and let $\widetilde{u}:=\widetilde{u}^{0}$ as in \eqref{def-normalization-solution}. 
Then, there exist $r_{0}>0$ and a non-zero $p\in\mathcal P_{2m+2s}$ such that $$\|\widetilde{u}-p\|_{L^\infty(B_r)}\le Cr^{2m+2s+\alpha}\quad\mbox{for every } r\in(0,r_{0}),$$ where $C>0$ and $\alpha \in (0,1)$ depend on $n$, $s$, $m$, $\vf$, $k$, $\gamma$ and $\|u\|_{L^\infty(B_1)}$.
\end{theorem}
As a consequence of \cref{prop:rate}, we get the stratification of the contact set $\Lambda_{2m+2s}(u)$. This is obtained with standard arguments based on the combination of Whitney's extension theorem with the implicit function theorem. 
\begin{corollary}[Stratification of $\Lambda_{2m+2s}(u)$]\label{cor:stratification}
Let $u$ be a solution of \eqref{def:soluzione-estesa}, with obstacle $\vf$ satisfying \eqref{e:hypo-phi}, and suppose that $2m+2s< k+\gamma$.
Then, the set $\Lambda_{2m+2s}(u)$ is contained in the union of at most countably
many manifolds of class $C^{1,\alpha}$. More precisely
$$\Lambda_{2m+2s}(u)=\overunderset{n-1}{j=1}{\bigcup}\Lambda^{j}_{2m+2s}(u),$$
where $\Lambda^{j}_{2m+2s}(u)$ is locally contained in a $j$-dimensional manifold of class $C^{1,\alpha}$, for every $j\in \{0,\ldots, n-1\}$.
\end{corollary}

\subsection{ Explicit frequency gaps}
One of the most challenging open problems in the study of the thin obstacle problem \eqref{def:soluzione-estesa} is the characterization of the set of admissible frequencies $\mathcal{A}_{n,s}$ from \eqref{def:admissible-frequencies}, for $n+1\ge3$. Clearly, those appearing in \eqref{dim1} are admissible for every $n$, however, a complete description of the set of \lq\lq other frequencies''
\begin{equation*}\label{freq-adm-star}
\mathcal{A}_{n,s}^*:=\mathcal{A}_{n,s}\setminus\mathcal{A}_{1,s}
\end{equation*}
is still missing.
A well-known fact, coming from the classification of convex blow-ups, is that $\mathcal{A}_{n,s}^*\subset(2,+\infty)$ (see e.g. \cite{css08}). Further interesting partial results were obtained in the few past years in ruling out the presence of other admissible frequencies around points of $\mathcal{A}_{1,s}$. 
More precisely, in \cite{csv20,sy22}, for $s=1/2$, and then in \cite{car24} for $s\in(0,1)$, the authors proved that even frequencies $2m$ are isolated in $\mathcal{A}_{n,s}$. The analogue result at odd frequencies $2m+1$ was derived in \cite{sy22,cv24}, for $s=1/2$. 
In the very recent \cite{fs24-gap}, the first \lq\lq thick'' gap was shown; precisely, for $s=1/2$, there are no admissible frequencies belonging to intervals of the form $(2m,2m+1)$, with $m\in\N$. In addition, as it is shown in the forthcoming \cite{fs24-2d}, the set of \lq\lq other frequencies'' is not empty as soon as $n+1\ge 3$.

The generic regularity for the fractional obstacle problem up to dimension $3$ requires as a crucial step to know that $\mathcal{A}_{n,s}\cap (2,1+2s+\sigma)=\emptyset,$ for some small $\sigma>0$, when $s>1/2$. This motivated us to investigate in this direction, showing several new frequency gaps for all $s\in (0,1)$:
\begin{enumerate}
    \item \label{1} we derive a small gap around $(2m+2s)$-frequencies, obtained via epiperimetric inequalities (see \cref{prop:gap-odd}), which extends the corresponding result in \cite{sy22, cv24} for $s=1/2$;
    \smallskip
    \item \label{2} we establish a gap between $\N+2s$ and $\N$ relying on domain monotonicity of eigenvalues (see \cref{prop:explicit-frequency-gap}). Together with the one proved in \cite{fs24-gap}, this is the first explicit gap in fractional obstacle problem which is uniform in the dimension and in the values of the frequencies;
    \smallskip
    \item \label{3} we show that, when $s\neq 1/2$, the frequencies $2m+1$ and $2m+1+2s$ and nearby are not admissible (see \cref{prop:snot=1/2});
    \smallskip
    \item\label{4} we extend to every $s\in (0,1)$ the result from \cite{fs24-gap}, proving that there are no admissible frequencies in intervals of the form $(2m,2m+2s)$ (see \cref{prop:gap-fs24}). 
\end{enumerate}

These results for fractional obstacle problem may be summarized in the following theorem.

\begin{theorem}[Frequency gaps]\label{teo:gaps}
Let $\mathcal{A}_{n,s}$ be the set of admissible frequencies in dimension $n+1$ from \eqref{def:admissible-frequencies} and let 
$\mathcal{A}_{n,s}^{*}:= \mathcal{A}_{n,s}\setminus \mathcal{A}_{1,s}$ be the set of "other frequencies". Then,
        for every $m\in \N$, there exist constants $c^{\pm}_{n,s,m}>0$, depending only on $n$, $s$ and $m$, such that:
\begin{itemize}
    \item if $s\le1/2$, then
$$\mathcal{A}_{n,s}^*\cap(2m-1+2s-c_{n,s,m}^-, 2m+1+c_{n,s,m}^+)=\emptyset;$$
    \item if $s>1/2$, then
    $$\mathcal{A}_{n,s}^*\cap(2m-c_{n,s,m}^-, 2m+2s+c_{n,s,m}^+)=\emptyset.$$
\end{itemize}       
\end{theorem}
Observe that, when $s=1/2$, \cref{teo:gaps} reduces to the known results from \cite{csv20,sy22,cv24,fs24-gap}. In the two limits $s\downarrow 0$ and $s\uparrow 1$ instead, \cref{teo:gaps} tells us that the admissible frequencies concentrate around integers and even integers, respectively. 

We point out that the result in \cref{teo:gaps}, in the case $s>1/2$, is a consequence of \eqref{1}, \eqref{4} and \cite[Proposition 1.3]{car24} only, even though, for the application to the proof of \cref{thm:mainthm}, the conclusions from \eqref{2} and \eqref{3} would have sufficed. On the other hand, \eqref{2} and \eqref{3} start giving distinct information with respect to \eqref{4} as soon as $s<1/2$.

\begin{figure}[H]
\centering
\includegraphics[width=0.95\textwidth]{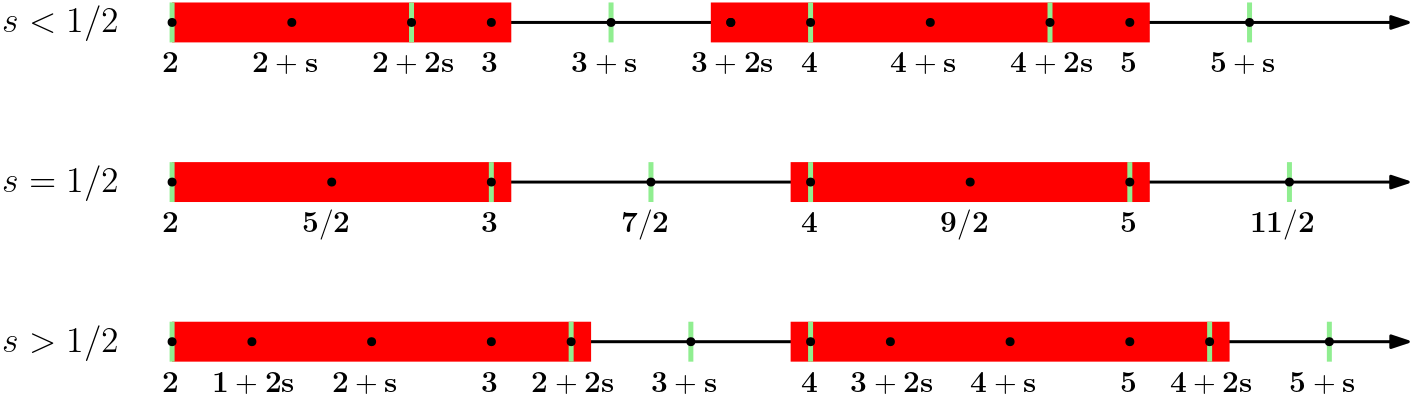}
\caption{The gaps of \cref{teo:gaps} for frequencies larger than $2$}
\label{fig:gaps}
\end{figure}
In \cref{fig:gaps}, presently known admissible frequencies are marked with vertical green lines, while red stripes indicate the intervals of frequencies excluded by \cref{teo:gaps}.

\subsection{Outline of the paper and sketch of the proof}\label{outline} To prove \cref{thm:mainthm} and \cref{thm:main_esteso}, we use the ideas developed in \cite{frs20,ft23} in combination with some results from \cite{fr21,car24}, properties on the fine structure of the free boundary that we extend from \cite{fj21,sy23,cv24} to the case of general $s\in(0,1)$ and $\vf\not\equiv0$, as for instance the rate of convergence at $(2m+2s)$-frequencies in \cref{prop:rate}, and finally the new frequency gaps from \cref{teo:gaps}.

We focus on \cref{thm:main_esteso}, since \cref{thm:mainthm} is easily deduced from it via the Caffarelli-Silvestre extension (see \cref{lemma:link_thm_main}).
We thus consider $u:B_1\times[0,1]\to\R$, a family of solutions of \eqref{def:soluzione-estesa}, \eqref{def:soluzione-famiglia}, with obstacle $\vf$ satisfying \eqref{e:hypo-phi}, with $k=4$ and $\gamma\in(a^-,1]$. 
We use bold font to indicate the union of sets over all $t\in[0,1]$. Namely, we denote by
$$
    \mathbf{\Gamma}(u):=\bigcup_{t\in[0,1]}\Gamma(u(\cdot,t))=\mathbf{Reg}(u) \cup \mathbf{Deg}(u),\\
    $$
    where
    $$
    \mathbf{Reg}(u):= \bigcup_{t\in[0,1]}\text{Reg}(u(\cdot,t)),\qquad \mathbf{Deg}(u):= \bigcup_{t\in[0,1]}\text{Deg}(u(\cdot,t)).
$$
We may split the whole degenerate set according to different values of the frequency: 
\be\label{splitting-degenerate-set}\mathbf{Deg}(u)= \Gammab_{2}^{\rm o}\cup \Gammab_{2}^{\rm a}\cup \Gammab_{2+2s}\cup \left(\Gammab_{\ge 3+s}\setminus \Gammab_{\ge4+\gamma}\right)\cup\Gammab_{\ge4+\gamma}\cup\Gammab_{*},\ee 
where (see \eqref{eq:def-gamma} and \eqref{def:gammao-and-gammaa}),
\medskip
\begin{enumerate}
    \item[i)] $\Gammab_{2}^{\rm o}$ and $\Gammab_{2}^{\rm a}$ are respectively the points with frequency $2$ for which the second term in the expansion of the corresponding solution is at least cubic (ordinary quadratic points) or otherwise (anomalous quadratic points).
    \smallskip
    \item[ii)] $\Gammab_{2+2s}$ are the points with frequency $2+2s$.
    \smallskip
    \item[iii)] $\Gammab_{\ge \lambda}$ are the points with frequency greater than or equal to $\lambda$.
    \smallskip
    \item[iv)] $\Gammab_{\ast}$ are the points whose frequency is not admissible in dimension $n+1=2$.
\end{enumerate}
\medskip

To prove \cref{thm:main_esteso}, the two main ingredients are dimension reduction arguments and cleaning-type results for each component of the splitting \eqref{splitting-degenerate-set}. After that, the application of an abstract GMT lemma (\cref{lemma:gmt}) coming from \cite{frs20} yields generic bounds on the dimension of the degenerate set. We outline hereafter the organization of the paper and the sketch of the proof of \cref{thm:main_esteso}. 

\begin{itemize}
\item In \cref{section-quadratic} we generalize the results in \cite{fj21} concerning the second blow-up at quadratic points for solutions of \eqref{def:soluzione-estesa} to the case of general obstacles $\vf\not\equiv0$, proving a new monotonicity formula for a truncated Almgren's frequency function in \cref{second-blowup-quadratic-points}. This allows us to further split the set of quadratic points in ordinary and anomalous quadratic points $\Gammab_{2}^{\rm o}$ and $\Gammab_{2}^{\rm a}$, defined in \eqref{def:gammao-and-gammaa}.
We point out that if $s>1/2$ and the space of invariant directions of the first blow-up is $(n-1)$-dimensional, then the second blow-up is a solution to the very thin obstacle problem \eqref{def:very-thin-obstacle-problem}; in all the other cases, the second blow-up is an $L_a$-harmonic polynomial. This dichotomy, first obtained in \cite{fj21}, descends from the fact that $(n-1)$-dimensional sets have zero $L_{a}$-harmonic capacity if and only if $s\le 1/2$.
\smallskip
\item In \cref{section-odd} we analyze points with frequency $2m+2s$, extending the results of \cite{frs20,sy23,cv24} to the case of a general fractional exponent $s\in (0,1)$. In particular, we prove the uniqueness of the blow-up limit with a polynomial rate of convergence in \cref{prop:rate}, together with a stratification result in \cref{cor:stratification}. This is obtained by
proving epiperimetric inequalities (see \cref{thm:epi}) and applying them at sufficiently small scales, as first done in \cite{cv24} in the case $s=1/2$. We will use the polynomial rate of convergence in \cref{prop:rate} to prove cleaning results for the set $\Gammab_{2+2s}$ (see \cref{prop:cleaning-(2+2s)}).
\smallskip
\item In \cref{section-frequency_gaps} we focus on the classification of admissible frequencies, proving in particular \cref{teo:gaps}. 
This follows from four main ingredients. First, in \cref{prop:gap-odd}, using the epiperimetric inequalities from \cref{thm:epi}, we show that frequencies $2m+2s$ are isolated.
Next, in \cref{prop:explicit-frequency-gap}, with an argument based on domain monotonicity of eigenvalues, we prove explicit gaps for frequencies between $\N+2s$ and $\N$.
Afterwards, in \cref{prop:snot=1/2}, we show that when $s\neq 1/2$, frequencies in $2\N+1$ and $2\N+1+2s$ are not admissible.
Finally, in \cref{prop:gap-fs24}, we extend \cite{fs24-gap}, proving that there are no admissible frequencies in $(2m,2m+2s)$.
The explicit gap obtained in \cref{teo:gaps} will be crucial in the proof of \cref{thm:main_esteso}, precisely, in the cleaning results for the set $\Gammab_{*}$, when $s>1/2$ (see \cref{prop:other-cleaning-results}).
\smallskip
\item In \cref{section-dim_reduc}, we prove the following dimensional bounds (see \cref{prop:dimension-reduction-ge2e*} and \cref{prop:gamma-anomalous}): 
$$\dim_{\HH}(\Gammab_{\ge2}\setminus \Gammab_{\ge 4+\gamma})\le n-1,\quad \dim_{\HH}(\Gammab_{*})\le n-2,\quad\dim_{\HH}(\Gammab_{2}^{\rm a})\le n-2.$$ 
This is obtained via dimension reduction arguments, which are based on the classification of admissible frequencies in low dimension. Here the monotonicity assumption in \eqref{def:soluzione-famiglia} is key to get the same results as one would get for a single solution.
The corresponding dimensional bounds were first proved in \cite{ft23} in the case $s=1/2$ and $\vf\equiv0$. However, we find new difficulties in our context, especially when $s>1/2$, where a thorough analysis is needed for the case where the second blow-up at quadratic points solves the very thin obstacle problem \eqref{def:very-thin-obstacle-problem}.
\smallskip
\item In \cref{section-clean}, we prove cleaning results for some subsets of the free boundary (see \cref{prop:gamma-ordinary}, \cref{prop:cleaning-(2+2s)} and \cref{prop:other-cleaning-results}). In particular, we obtain the following:
$$\{\Gamma_{2}^{\rm o}(u(\cdot,t)\}_{t\in[0,1]\}}\in
\begin{cases}\text{Clean}(3) &\mbox{if } s\le 1/2,\\
\text{Clean}(4-2s) &\mbox{if } s> 1/2,\end{cases}$$ $$\{\Gamma_{2+2s}(u(\cdot,t))\}_{t\in[0,1]\}}\in\text{Clean}(2+\eta),\quad \text{for some }\eta>0.$$
Here, for a family of sets $\{E_t\}_{t\in [0,1]}$ and a number $\beta>0$, the notation $\{E_{t}\}_{t\in [0,1]}\in \text{Clean}(\beta)$ is introduced in \cref{def:cleaned}, meaning, roughly speaking, that $t\mapsto E_{t}$ detaches from $E_{t_{0}}$ with rate at least $|t-t_{0}|^{1/\beta}$. 
These cleaning results follow by an expansion of the solution near the corresponding free boundary points and barrier arguments. 
Concerning the set $\Gammab_{2}^{\rm o}$, the difference between the cases $s\le1/2$ and $s>1/2$ is due to the fact that, when $s>1/2$, the second blow-up can either be $L_a$-harmonic or solve the very thin obstacle problem \eqref{def:very-thin-obstacle-problem}.
Regarding the set $\Gammab_{2+2s}$, we use, as a fundamental ingredient, the rate of convergence to the blow-up limit proved in \cref{prop:rate}.
The cleaning results for the other sets in the splitting \eqref{splitting-degenerate-set} are proved in \cref{prop:other-cleaning-results} and they are consequence of \cref{prop:cleaning}, coming from \cite{fr21}, and of the frequency gaps in \cref{teo:gaps}. 
\smallskip
\item Finally, in \cref{section-proofs_main}, we combine the dimensional bounds of \cref{section-dim_reduc} with the cleaning results of \cref{section-clean} through the GMT \cref{lemma:gmt}, thus proving \cref{thm:main_esteso}.
In fact, in \cref{prop:final}, we obtain even a more detailed description of the generic dimension of each component in the splitting \eqref{splitting-degenerate-set}.
\end{itemize}

In the following tables we summarize the dimensional bound, the cleaning exponent and the generic dimension of each subset in the splitting \eqref{splitting-degenerate-set} that we obtain in this work (see \cref{prop:final} for more details). 
\begin{table}[H]
    \centering
    \begin{tabular}{|c|c|c|c|}
\hline
Set & $\text{dim}_\HH\Gammab$ & Cleaning exponent & Generic $\text{dim}_\HH\Gamma$\\
\hline
$\Gammab_{2}^{\rm o}$ & $n-1$ & $3$ & $n-4$\\
\hline
$\Gammab_{2}^{\rm a}$ & $n-2$ & $2$ & $n-4$  \\
\hline
$\Gammab_{2+2s}$ & $n-1$ & $2+\eta$ & $n-3-\eta$\\
\hline
$\Gammab_{\ge 3+s}\setminus\Gammab_{\ge 4+\gamma}$ & $n-1$ & $3-s$ & $n-4+s$\\
\hline
$\Gammab_{\ge 4+\gamma}$ & $n$ & $3+\gamma$ & $n-3-\gamma$\\
\hline
$\Gammab_{*}$ & $n-2$ & $3-2s+\sigma$ & $n-5+2s-\sigma$ \\
\hline
\end{tabular}
\vspace{0.2cm}
    \caption{The case $s\le 1/2$}
    \label{tab1}
\end{table}
\vspace{-0.3cm}
\begin{table}[H]
    \centering
    \begin{tabular}{|c|c|c|c|}
\hline
Set & $\text{dim}_\HH\Gammab$ & Cleaning exponent & Generic $\text{dim}_\HH\Gamma$\\
\hline
$\Gammab_{2}^{\rm o}$ & $n-1$ & $4-2s$ & $n-5+2s$\\
\hline
$\Gammab_{2}^{\rm a}$ & $n-2$ & $(4-2s)/(1+s)$ & $n-6/(1+s)\ $ \\
\hline
$\Gammab_{2+2s}$ & $n-1$ & $2+\eta$ & $n-3-\eta$\\
\hline
$\Gammab_{\ge 3+s}\setminus\Gammab_{\ge 4+\gamma}$ & $n-1$ & $3-s$ & $n-4+s$\\
\hline
$\Gammab_{\ge 4+\gamma}$ & $n$ & $4+\gamma-2s$ & $n-4-\gamma+2s$\\
\hline
$\Gammab_{*}$ & $n-2$ & $2+\sigma$ & $n-4-\sigma$ \\
\hline
\end{tabular}
    \vspace{0.2cm}
    \caption{The case $s>1/2$}
    \label{tab2}
\end{table}
\noindent The tables should be interpreted keeping in mind that whenever the number in the last column is strictly negative, the corresponding set is generically empty.
Here $\eta>0$ and $\sigma>0$ are constants depending only on $n$ and $s$. More precisely, $\eta$ is the constant in \cref{prop:cleaning-(2+2s)}, which is related to the polynomial rate of convergence to the blow-up at points with frequency $2+2s$ (see \cref{prop:rate}).
While, $\sigma$ is the amplitude of the gap from \cref{teo:gaps} around the frequency $3$, for $s\le1/2$, and around the frequency $2+2s$, for $s>1/2$ (see also \cref{prop:other-cleaning-results}). Notice that we require the condition $\gamma\in(a^-,1]$ to estimate the generic dimension of $\Gammab_{\ge4+\gamma}$ in the case $s>1/2$.

\section{Preliminaries}\label{section-prelim}
\subsection{Notations} Throughout the paper we will make use of the following notations.
\begin{itemize}
    \item We assume that $n\in \mathbb{N}$, $s\in (0,1)$, $a:=1-2s \in (-1,1)$ and the obstacle $\varphi$ satisfying \eqref{e:hypo-phi} are fixed.
    \smallskip
    \item Points in $\R^{n+1}$ will be indicated by capital letters. We typically split $X\in \R^{n+1}$ in its two components $X=(x,y)\in \R^{n}\times \R$. Balls in $\R^{n+1}$ are indicated by $B_{r}(X_{0}):=\left\{X\in \R^{n+1}: |X-X_{0}|<r\right\}$.
    \smallskip
\item For every given set $E\subset \R^{n+1}$, we call
$$E^+:=E\cap \{y>0\},\quad E':=E\cap \{y=0\},\quad E^-:=E\cap \{y<0\}.$$
When $E=\R^{n+1}$, we write for simplicity
$$\R^{n+1}_{\pm}:= (\R^{n+1})^{\pm},\quad \R^{n+1}_{0}:= (\R^{n+1})'.$$
\item For a given open set $\Omega\subset \R^{n+1}$, we denote by $L^{2}(\Omega, \y)$ and $W^{1,2}(\Omega,\y)$ the spaces of square integrable functions and Sobolev functions in $\Omega$ with respect to the weighted measure $\y \mathscr{L}^{n+1}$. Similarly, whenever $\Sigma \subset \partial B_{R}$ is an open set of the sphere $\partial B_{R}$, then $L^{2}(\Sigma, \y)$ and $W^{1,2}(\Sigma,\y)$ will indicate the spaces of square integrable functions and Sobolev functions in $\Sigma$ with respect to the weighted measure $\y \mathcal{H}^{n}$.
\smallskip
\item For a given open set $\Omega\subset\R^{n+1}$, we define $$C^{1,\alpha}_a(\Omega):=\{w\in C^1(\Omega):\|w\|_{C^{1,\alpha}_a}<+\infty\},$$ where $$\|w\|_{C^{1,\alpha}_{a}(\Omega)}:=\|w\|_{L^\infty(\Omega)}+\|\nabla_x w\|_{C^\alpha(\Omega)}+\|\y\partial_y w\|_{C^\alpha(\Omega)}.$$
\end{itemize}
 
\subsection{Almgren's frequency function}
Given a general $v\in W_{\loc}^{1,2}(\R^{n+1},|y|^{a})$ and a point $X_{0}=(x_{0},0)\in \R^{n}\times \left\{0\right\}$ on the thin space, we define the following quantities, for every $r>0$:
\be\label{def:HD}
    H^{X_{0}}(r,v):= \int_{\partial B_{r}(X_{0})}|v|^{2}|y|^{a}\,d\mathcal{H}^{n},\qquad
    D^{X_{0}}(r,v):= \int_{B_{r}(X_{0})}|\nabla v|^{2}|y|^{a}\,dX.
\ee
In the study of obstacle-type problems, an important role is played by the so-called \lq\lq pure'' Almgren's frequency function, namely:
\bea N^{X_0}(r,v):=\frac{rD^{X_{0}}(r,v)}{H^{X_{0}}(r,v)}.\eea 
Indeed, in \cite{acs08,css08} it was proved that if $u$ is a solution of \eqref{def:soluzione-estesa} with zero obstacle $\varphi\equiv 0$ and $X_{0}\in \Lambda(u)$ is a contact point, then $r\mapsto N^{X_0}(r,u)$ is non-decreasing and so, the limiting value $N^{X_0}(0^{+},u):=\lim_{r\downarrow 0} N^{X_0}(r,u)$ is well-defined. It is called the \lq\lq frequency'' at $X_{0}$, and gives precise information about the local behavior of $u$ near $X_{0}$. The monotonicity of Almgren's frequency function is a key tool for carrying out a blow-up analysis with all its consequences.

When dealing with non-zero obstacles, the analysis has to be carried on with some modifications, as done for instance in \cite{gr19}. Suppose that $\varphi$ satisfies \eqref{e:hypo-phi}. 
Let $Q^{X_{0}}:\R^{n}\to \R$ be the $k$-th order Taylor polynomial of $\varphi$ at $X_{0}$, and let $\widetilde{Q}^{X_{0}}:\R^{n+1}\to \R$ be the even $L_{a}$-harmonic extension of $Q^{X_{0}}$ to the whole $\R^{n+1}$, which is well-defined by \cref{lemma:unique-extension-polynomial-aHarmonic}. 
We call
\begin{equation}\label{def:phitilde}
    \widetilde{\varphi}^{X_{0}}(x,y):= \varphi(x)-Q^{X_{0}}(x)+\widetilde{Q}^{X_{0}}(x,y).
\end{equation}
Then, the function
\begin{equation}\label{def-normalization-solution}
    \widetilde{u}^{X_{0}}(x,y):=u(x,y)-\widetilde{\varphi}^{X_{0}}(x,y).
\end{equation}
is an even $L_a$-harmonic extension of $u(\cdot,0)-\varphi$ to the whole $\R^{n+1}$, and moreover, by construction, it solves a thin obstacle problem with zero obstacle and a suitable right-hand side:
\begin{equation}
\left\{
    \begin{array}{rclll}\label{def:soluzione-estesa-tilde}
            -L_a \widetilde{u}^{X_{0}}(x,y)&=&g^{X_{0}}(x)|y|^{a} & \quad \mbox{for } (x,y)\in B_{1}\setminus \{\widetilde{u}^{X_{0}}=0\}',\\
            -L_a \widetilde{u}^{X_{0}}(x,y) &\ge& g^{X_{0}}(x)|y|^{a} & \quad \mbox{for } (x,y)\in B_{1},\\
		\widetilde{u}^{X_{0}}(x,0)&\ge& 0 & \quad \mbox{for } (x,0)\in B'_{1},\\
		\widetilde{u}^{X_{0}}(x,y)&=& \widetilde{u}^{X_{0}}(x,-y) &\quad \mbox{for } (x,y)\in B_{1}.
\end{array}\right.
\end{equation}
Here $g^{X_{0}}(x):= \left(\Delta_{x}\varphi(x)-\Delta_{x}Q^{X_{0}}(x)\right)$ satisfies the estimate 
\begin{equation}\label{eq:stima-g}
    |g^{X_{0}}(x)|\le C|x-x_{0}|^{k+\gamma-2}\quad \mbox{for every } x\in \R^{n}: |x|<1,
\end{equation}
where $C=C(\vf)>0$ depends on $\varphi$, $n$, $k$ and $\gamma$ as can be easily obtained by Taylor-expanding $g^{X_{0}}$ around $x_{0}$.
In the sequel we will drop the superscript $X_{0}$ whenever $X_{0}=0$.

We now recall from \cite{gr19} the monotonicity of a \lq\lq truncated'' Almgren's frequency function, with the existence of blow-up limits.
\begin{proposition}[\protect{\cite[Section 6]{gr19}}]\label{generalized-Almgren}
    Let $u$ be a solution of \eqref{def:soluzione-estesa} with obstacle $\varphi$ satisfying \eqref{e:hypo-phi} and $\|u\|_{L^\infty(B_1)}\le 1$. Given $\theta\in(0,\gamma)$, there exist $C_{0}, r_{0}>0$, such that, for every $X_{0}\in \Lambda(u)$, the function
    $$r \mapsto \phi^{X_{0}}(r,u):=(r+C_{0}r^{1+\theta})\frac{d}{dr}\log \max \left\{H^{X_{0}}(r,\widetilde{u}^{X_{0}}),r^{n+a+2(k+\gamma-\theta)}\right\}$$
    is non-decreasing in $(0,r_{0})$. In particular it is well-defined $\phi^{X_{0}}(0^{+},u):= \lim_{r\downarrow0}\phi^{X_0}(r,u)$ and $$\phi^{X_{0}}(0^{+},u)\in \left\{n+3\right\}\cup [n+a+4,n+a+2(k+\gamma)].$$
    
    \noindent If in addition $\phi^{X_{0}}(0^{+},u)=n+a+2\lambda$, for some $\lambda<k+\gamma$, then the following holds.
    \begin{itemize}
        \item [i)] $N^{X_{0}}(0^{+},\widetilde{u}^{X_{0}})=\lambda$. Moreover, there exists $C>0$ such that
        $$H^{X_{0}}(r,\widetilde{u}^{X_{0}})\le Cr^{n+a+2\lambda}\quad \mbox{for every } r\in (0,r_{0}),$$
        and, for every $\eps>0$, there exist $C_{\eps}>0$ and $r_{\eps}\in (0,r_{0})$ such that 
        $$H^{X_{0}}(r,\widetilde{u}^{X_{0}})\ge C_{\eps}r^{n+a+2\lambda+\eps}\quad \mbox{for every } r\in (0,r_{\eps}).$$
        \item [ii)] For every $r_{j}\downarrow 0$, the blow-up sequence
        $$\widetilde{u}^{X_{0}}_{r_{j}}:= \frac{\widetilde{u}^{X_{0}}(X_{0}+r_{j}\cdot)}{\lVert \widetilde{u}^{X_{0}}(X_{0}+r_{j}\cdot)\rVert_{L^{2}(\partial B_{1},|y|^{a})}}$$
        is uniformly bounded in $W_{\loc}^{1,2}(\R^{n+1},|y|^{a})$ and in $C_{a,\loc}^{1,s}(\R^{n+1}_+)$. In particular, up to subsequences, $\widetilde{u}^{X_{0}}_{r_{j}}$ converges to some $p$ strongly in $W_{\loc}^{1,2}(\R^{n+1},|y|^{a})$ and in $C_{a,\loc}^{1,\alpha}(\R^{n+1}_+)$, for any $\alpha \in (0,s)$, and $p$ is a non-zero global $\lambda$-homogeneous solution of \eqref{def:soluzione-estesa} with zero obstacle. 
    \end{itemize}  
\end{proposition}

Thanks to \cref{generalized-Almgren} we can classify free boundary points of a solution to \eqref{def:soluzione-estesa} according to the corresponding asymptotic value of the generalized Almgren's frequency function. Let $u:B_{1}\to \R$ be a solution to \eqref{def:soluzione-estesa}, with obstacle $\varphi$ satisfying \eqref{e:hypo-phi}. For every $\lambda <k+\gamma$, we define the following subsets of the free boundary and of the contact set respectively, which are made of those points that have precisely frequency $\lambda$: 
\be\label{eq:gammalambda}
    \Gamma_{\lambda}(u):=\left\{X_{0}\in \Gamma(u): N^{X_{0}}(0^{+},\widetilde{u}^{X_{0}})=\lambda\right\},\quad
    \Lambda_{\lambda}(u):=\left\{X_{0}\in \Lambda(u): N^{X_{0}}(0^{+},\widetilde{u}^{X_{0}})=\lambda\right\}.
\ee
We also denote, for every $\lambda\le k+\gamma$, the set of free boundary points with frequency larger than $\lambda$ as 
\bea \Gamma_{\ge \lambda}(u):= \Gamma(u)\setminus \underset{\mu < \lambda}{\bigcup}\Gamma_{\mu}(u). \eea
Remember the definition of the set of admissible homogeneities
$\mathcal{A}_{n,s}$ from \eqref{def:admissible-frequencies}.
In dimension $n=0,1$ they are completely classified (see for instance the appendix of \cite{fs18}):
\begin{proposition}\label{prop:classification-solution-dim2}
    We have:
    \begin{itemize}
        \item [i)] $\mathcal{A}_{0,s}= \{2s\}$;
        \smallskip
        \item [ii)] $\mathcal{A}_{1,s}=\bigcup_{m\in \mathbb{N}}\left\{2m+1+s, 2m, 2m+2s\right\}$.
    \end{itemize}
\end{proposition}
\noindent Then, we also introduce a notation for the set of free boundary points with frequency which is not admissible in dimension $n+1=2$:  
\bea
    \Gamma_{*}(u):=\bigcup_{\lambda\in(2,k+\gamma)\setminus \mathcal{A}_{1,s}}\Gamma_{\lambda}(u).
\eea

For a family of solutions $u:B_{1}\times [0,1]\to \R$ of \eqref{def:soluzione-estesa}, \eqref{def:soluzione-famiglia}, we use bold font to denote the corresponding subsets of the union of free boundaries over $t\in [0,1]$:
\be\label{eq:def-gamma}
\mathbf{\Gamma}_{\lambda}:=\bigcup_{t\in[0,1]}\Gamma_\lambda(u(\cdot,t)),\quad \mathbf{\Gamma}_{\ge\lambda}:=\bigcup_{t\in [0,1]}\Gamma_{\ge \lambda}(u(\cdot, t)),\quad \mathbf{\Gamma_*}:=\bigcup_{t\in [0,1]}\Gamma_{*}(u(\cdot, t)).
\ee
\subsection{The Weiss energy} Next we recall the definition and some basic properties of the Weiss energy. Given $v\in W_{\loc}^{1,2}(\R^{n+1},|y|^{a})$, $\lambda\in \R$ and $r>0$, we call
\begin{equation}\label{weiss}
    W_{\lambda}(r,v):=\frac{rD(r,v)-\lambda H(r,v)}{r^{n+a+2\lambda}}.
\end{equation}
For a solution $u$ of \eqref{def:soluzione-estesa} with non-zero obstacle, calling $\widetilde{u}=\widetilde{u}^{0}$ and $g=g^{0}$ the right-hand side in \eqref{def:soluzione-estesa-tilde}, we will consider the following modification of the Weiss energy:  
\be \label{def:weiss-tilde} \widetilde W_{\lambda}(r,\widetilde{u}):=W_{\lambda}(r,\widetilde{u})-\frac{1}{r^{n+a+2\lambda-1}}\int_{B_{r}}\widetilde{u}g\y\,dX.\ee
We will drop the dependence on $r$ when $r=1$.
It is important to notice that, by the Euler-Lagrange equations in \eqref{def:soluzione-estesa-tilde}, $\widetilde{u}$ minimizes the functional 
$$v\mapsto W_{\lambda}(v)-\frac{1}{r^{n+a+2\lambda-1}}\int_{B_{1}}v g\y\,dX$$
among all functions $v \in W^{1,2}(B_{1},\y)$ even in $y$, such that $v\ge 0$ on $B_{1}'$ and $v|_{\partial B_{1}}=\widetilde{u}|_{\partial B_{1}}$. 

To conclude this part, we recall the following monotonicity formula for the modified Weiss energy.
\begin{proposition}[\protect{\cite[Proposition 2.7]{car24}}]\label{prop:monotonicity-weiss-tilde}
    Let $u$ be a solution of \eqref{def:soluzione-estesa} with obstacle $\varphi$ satisfying \eqref{e:hypo-phi}. Suppose that $0\in \Lambda_{\lambda}(u)$ for some $\lambda<k+\gamma$, and call $\widetilde{u}=\widetilde{u}^{0}$. Then there is a constant $C_{\widetilde W}=C_{\widetilde W}(\widetilde u)>0$
    for which
    \bea
    \frac{d}{dr}\left( \widetilde W_{\lambda}(r,\widetilde{u})+ C_{\widetilde W}r^{k+\gamma-\lambda}\right)\ge \frac{2}{r^{n+a+2\lambda+1}}\int_{\partial B_{r}(X_{0})}(\nabla \widetilde{u}\cdot X-\lambda \widetilde{u})^{2}|y|^{a}\,d\HH^n
    \eea
    for every $r\in(0,1)$.
    Moreover $\widetilde W_{\lambda}(0^{+},\widetilde{u}):= \lim_{r\downarrow 0}\widetilde W_{\lambda}(r,\widetilde{u})=0$. In particular,
    \bea
        \widetilde W_{\lambda}(0^{+},\widetilde{u})\ge -C_{\widetilde W}r^{k+\gamma-\lambda}\quad \mbox{for every } r\in (0,1).
    \eea
\end{proposition}

\subsection{Quadratic points}
Let us now recall some fundamental known results regarding quadratic points, i.e.~the free boundary points with frequency $2$. We define the set of polynomials
\bea
        \mathcal{P}_2:=\{p \ \mbox{polynomial }: L_ap=0,\ \nabla p\cdot X=2p, \ p(x,0)\ge0,\ p(x,y)=p(x,-y)\}.
\eea
Then it is well-known (see \cite{gp09,gr19}) that the admissible blow-ups at quadratic points are exactly the polynomials in $\mathcal{P}_2$. We recall the uniqueness of the blow-up limit with explicit rate of convergence, which was proved in \cite{csv20,car24}.
\begin{proposition}[Uniqueness of the blow-up and rate of convergence at quadratic points]
\label{prop:classification-blowup-freq-2}
    Let $u$ be a solution of \eqref{def:soluzione-estesa}, with obstacle $\varphi$ satisfying \eqref{e:hypo-phi}. Suppose that $0\in\Lambda_2(u)$, and let $\widetilde{u}:= \widetilde{u}^{0}$ as in \eqref{def-normalization-solution}. Then, there exist $r_0>0$ a non-zero $p\in\mathcal{P}_2$, such that  
    $$\lVert\widetilde{u}-p\rVert_{L^\infty(B_r)}\le Cr^{2}(-\log(r))^{-c}\quad \mbox{for every } r\in (0,r_{0}),$$ where $C>0$ and $c\in(0,1)$ depend on $n$, $s$, $\vf$, $\gamma$ and $\|u\|_{L^\infty(B_1)}$. 
\end{proposition}
For a non-zero polynomial $p\in \mathcal{P}_{2}$ we define the \lq\lq spine'' as the linear subspace of directions in the thin space along which $p$ is invariant: 
\be\label{eq:L(p)}
    L(p):= \left\{\xi \in \R^{n}\times \left\{0\right\}: \nabla p \cdot \xi \equiv 0\right\}=\{(x,0)\in \R^{n+1}_{0}:p(x,0)=0\},
\ee
where the second equality is due to the fact that $p$ is a $2$-homogeneous polynomial, non-negative on the thin space.
We also denote the dimension of $L(p)$ by 
\begin{equation}\label{eq:m(p)}
    m(p):= \dim L(p)\in \left\{0,1,\dots,n-1\right\}.
\end{equation}
See \cref{section-quadratic} for a second blow-up analysis at quadratic points extending the corresponding results of \cite{fj21} to the case of non-zero obstacles. 

\subsection{Geometric Measure Theory tools}

The key tool in the proof of \cref{thm:main_esteso} is the following GMT Lemma coming from \cite{frs20}.
\begin{lemma}[\protect{\cite[Corollary 7.8]{frs20}}]
\label{lemma:gmt}
    Consider a family of sets $\left\{E_{t}\right\}_{t\in[0,1]}$ with $E_{t}\subset \mathbb{R}^{n}$, and let $E:=\underset{t\in [0,1]}{\bigcup}E_{t}$.
    \vspace{-0.3cm}
    
    \noindent Let $1\le \alpha \le n$, and assume that the following hold:
    \begin{itemize}
        \item[i)] $\text{dim}_{\mathcal{H}} E \le \alpha$;
        \smallskip
        \item[ii)] for all $\eps>0$, $t_{0}\in [0,1]$ and $X_{0}\in E_{t_{0}}$, there exists $\rho >0$ such that
        \begin{equation*}
            \{X\in B_{\rho}(X_{0}):t>t_{0}+|X-X_0|^{\beta-\eps}\}\cap E_{t} = \emptyset \quad\mbox{for every } t\in[0,1].
        \end{equation*}
    \end{itemize}
    Then:
    \begin{itemize}
        \item [a)] if $\beta >\alpha$, $\text{dim}_{\mathcal{H}}\left(\left\{t: E_{t}\neq \emptyset\right\}\right)\le \alpha / \beta$;
        \smallskip
        \item [b)] if $\beta \le \alpha$, $\text{dim}_{\mathcal{H}}\left(E_{t}\right)\le \alpha - \beta$, for almost every $t\in [0,1]$.
    \end{itemize}
\end{lemma}
\noindent We will apply \cref{lemma:gmt} to each subset of the free boundary of a family of solutions appearing in the splitting \eqref{splitting-degenerate-set}. For each of them we will prove that i) and ii) hold for a suitable choice of $\alpha$ and $\beta$. 

To perform dimension reduction arguments and prove condition i) in \cref{lemma:gmt} we will often take advantage of the following general lemma coming from \cite{frs20}:
\begin{lemma}[\protect{\cite[Lemma 7.3]{frs20}}]
\label{lemma:characterization-of-m-dim-set}
    Let $E\subset \R^{n}$ and $f:E\to \R$ be any function. Assume that for every $\eps>0$ and $x\in E$ there exists $\rho >0$ such that, for all $r\in (0,\rho)$,
    \begin{equation*}
        E\cap \overline{B}_{r}(x)\cap f^{-1}\left([f(x)-\rho,f(x)+\rho]\right)\subset \left\{y: \dist(y,\Pi_{x,r})\le \eps r\right\},
    \end{equation*}
    for some $m$-dimensional space $\Pi_{x,r}$ passing through $x$. Then, $\text{dim}_{\mathcal{H}}(E)\le m$.
\end{lemma}
\begin{remark}\label{remark:use-of-lemma-char-m-dim-sets} In the sequel we will use \cref{lemma:characterization-of-m-dim-set} taking some subset of the free boundary of a family of solutions as $E$, and Almgren's frequency function as $f$. We will typically proceed by contradiction using the contronominal implication, that is: if $\text{dim}_{\mathcal{H}}(E)> m$, then there exist $\eps>0$ and $X\in E$ such that, for every $\rho>0$, there exists $r\in (0,\rho)$ for which, given any $m$-dimensional plane $\Pi$ passing through $X$, we have
    \begin{equation*}
        E\cap \overline{B}_{r}(X)\cap f^{-1}\left([f(X)-\rho,f(X)+\rho]\right)\nsubseteq \left\{Y: \dist(Y,\Pi)\le \eps r\right\}.
    \end{equation*}
    In particular, there exist a sequence of radii $r_{j}\downarrow 0$ and $(m+1)$ sequences of points $X^{1}_{j},\dots,X^{m+1}_{j}\in E$ for which
    $$|X^{\ell}_{j}-X|\le r_{j},\quad \frac{X^{\ell}_{j}-X}{r_{j}}\to Y^{\ell}\neq 0,\quad f(X^{\ell}_{j})\to f(X),$$
    and such that $Y^{1},\dots,Y^{m+1}$ are linearly independent.
\end{remark}

Since it will appear often in the sequel, it is worth enclosing condition ii) of \cref{lemma:gmt} in a definition.
\begin{definition}\label{def:cleaned}
    Let $\left\{E_{t}\right\}_{t\in[0,1]}$ and $E$ be as in \cref{lemma:gmt}. Given a number $\beta >0$, we say that $E$ is cleaned up with rate $\beta$ if condition ii) of \cref{lemma:gmt} holds. In this case, we write $$\{E_{t}\}_{t\in[0,1]}\in\text{Clean}(\beta).$$
\end{definition}

\noindent In the following proposition we collect some known \lq\lq cleaning results'' from \cite{fr21} that will be used in the sequel.
\begin{proposition}[\protect{\cite[Proposition 2.4, Proposition 2.9]{fr21}}]\label{prop:cleaning}
    Let $u:B_1\times[0,1]\to\R$ be a family of solutions of \eqref{def:soluzione-estesa} \eqref{def:soluzione-famiglia}, with obstacle $\vf$ satisfying \eqref{e:hypo-phi}. Then 
    \begin{itemize}
        \item [i)] For every $\lambda \ge 1+s$ such that $\lambda\le \min\{k+\gamma, k+\gamma-a\}$, we have $$\{\Gamma_{\ge \lambda}(u(\cdot,t))\}_{t\in[0,1]}\in\text{Clean}(\lambda-2s).$$
        \item [ii)] Regarding quadratic free boundary points:
        \smallskip
        \begin{itemize}
        \smallskip
            \item [a)] if $s\le 1/2$ we have $$\{\Gamma_{2}(u(\cdot,t))\}_{t\in[0,1]}\in\text{Clean}(2);$$ 
            \item [b)] if $s>1/2$ we have
        $$\{\Gamma_{2}(u(\cdot,t))\}_{t\in[0,1]}\in\text{Clean}\left(\frac{4-2s}{1+s}\right).$$
        \end{itemize}
    \end{itemize}
        
\end{proposition}

We end this subsection with the following useful result. 
\begin{proposition}[\protect{\cite[Corollary 2.7]{fr21}}]\label{prop:continuity-tau}
    Let $u$ be a family of solutions of \eqref{def:soluzione-estesa}, \eqref{def:soluzione-famiglia}, with obstacle $\vf$ satisfying \eqref{e:hypo-phi}. Then the function $\tau:\mathbf{\Gamma}\to[0,1]$ which assigns to any $X_{0}\in \Gammab(u)$ the unique $t_{0}=:\tau(X_0)\in [0,1]$ such that $X_0\in\Gamma({u(\cdot,t_0)})$ is well-defined and continuous. Moreover, for every $\eps>0$, the map $$\mathbf{\Gamma}\cap B_{1-\eps}\ni X_0\mapsto \widetilde{u}^{X_{0}}(X_0+\cdot,\tau(X_0))$$ is continuous with respect to the uniform convergence. 
\end{proposition}

\section{The second blow-up at quadratic points}\label{section-quadratic}
 In this section we split the set of quadratic free boundary points according to the behavior of a second blow-up. The existence and properties of the second blow-up at free boundary points with even frequency was extensively studied in \cite{fj21} in the case of zero obstacles. Here we will extend these results for quadratic points to the case of non-zero obstacles.

Let us first introduce some notations. We fix $\beta>0$, and we call, for any $v\in W^{1,2}(B_1,\y)$ and $r>0$,
$$\widetilde D(r,v):=D(r,v)+\beta r^{n+a+2\beta-1},\qquad \widetilde H(r,v):=H(r,v)+r^{n+a+2\beta},$$ where $D$ and $H$ were defined in \eqref{def:HD}.
Then, we consider the modified Almgren's frequency function \bea\widetilde N(r,v):=\frac{r\widetilde D(r,v)}{\widetilde H(r,v)}.\eea 
   Finally, whenever $L_{a}v$ is a measure and $v$ is continuous, we define
   \begin{equation}\label{def:F}
    F(r,v):= \int_{B_{r}}vL_{a}v.
\end{equation}
   For a non-zero polynomial $p\in\mathcal{P}_2$, recall also the definitions of the spine $L(p)$ and its dimension $m(p)$ as in \eqref{eq:L(p)} and \eqref{eq:m(p)}, respectively.
   The main result of this section is the following proposition.
   \begin{proposition}\label{second-blowup-quadratic-points}
    Let $u$ be a solution of \eqref{def:soluzione-estesa}, with obstacle $\varphi$ satisfying \eqref{e:hypo-phi} and $\|u\|_{L^\infty(B_1)}\le 1$. Given $X_{0}\in \Gamma_{2}(u)$, let $p\in \mathcal{P}_{2}$ be an admissible blow-up, and call $v:= \widetilde{u}^{X_{0}}(X_{0}+\cdot)-p$, where $\widetilde{u}^{X_{0}}$ is given by \eqref{def:soluzione-estesa-tilde}. 
    Let $\beta>0$ be such that $k+\gamma+2-2\beta>0$. Then, there exist constants $C,r_0>0$ such that
    $$r\mapsto \hat{N}(r,v):= e^{Cr^{k+\gamma+2-2\beta}}\left(\widetilde N(r,v)+Cr^{k+\gamma+2-2\beta}\right)$$
    is non-decreasing in $(0,r_{0})$. 
    More precisely,
    \be\label{eq:formula-n-tilde}\frac{d}{dr} \hat N(r,v) \ge r\frac{F(r,v)^2}{\widetilde H(r,v)^2}\quad\mbox{for every } r\in(0,r_0),\ee where $F$ is defined in \eqref{def:F}.
    In particular, $\widetilde N(0^{+},v):= \lim_{r\downarrow 0}\widetilde N(r,v)$ is well-defined, and
    \be\label{eq:lowup}\widetilde N(0^{+},v)\in [2,\beta].\ee
    In addition, the following facts hold.
    \begin{itemize}
         \item [i)] If $\ 0<r<R<r_{0}$ are such that $\underline{\lambda}\le \widetilde N(\rho,v)\le \overline{\lambda}$ for every $\rho \in [r,R]$, then, for every $\eps>0$, there is $C_\eps>0$ such that
        \begin{equation}\label{growth-secondblowup}
            c\left(\frac{R}{r}\right)^{n+a+2\underline{\lambda}}\le \frac{ \widetilde{H}(R,v)}{\widetilde{H}(r,v)}\le C_{\eps}\left(\frac{R}{r}\right)^{n+a+2\overline{\lambda}+\eps}
        \end{equation} for some constant $c>0$.
        \smallskip
    \item [ii)] For every $\underline \lambda\le \widetilde N(0^{+},v)$, the following Monneau-type monotonicity formula holds:
        \be\label{monneau}\frac{d}{dr}\left(\frac{\widetilde H(r,v)}{r^{n+a+2\underline \lambda}}\,
        e^{Cr^{n+a+2\underline \lambda}}\right)\ge 0\quad \mbox{for every } r\in (0,r_{0}).\ee 
    \item [iii)] Suppose now that $p\in \mathcal{P}_{2}$ is the first blow-up of $\widetilde{u}^{X_{0}}$ at $X_0$ and that $\widetilde N(0^{+},v)<\beta$. Then, for every $r_{j}\downarrow 0$, the blow-up sequence
        \be\label{def:resclaings-vrj}v_{r_{j}}:= \frac{v(r_{j}\cdot)}{\lVert v(r_{j}\cdot)\rVert_{L^{2}(\partial B_{1},|y|^{a})}}\ee
        is uniformly bounded in $W_{\loc}^{1,2}(\R^{n+1},|y|^{a})$. In particular, up to subsequences, it weakly converges in such space to a non-zero function $q$ which is $\lambda$-homogeneous, with $\lambda=N(0^{+},v)$. Moreover, if $k+\gamma+2-2\beta>\lambda-2$, then the following orthogonality properties hold:
        \begin{equation}\label{eq:ortogonality-first-second-bu}
            \int_{\partial B_1}p q\,d\HH^n=0\qquad\mbox{and}
\qquad\int_{\partial B_1}\bar{p} q \,d\HH^n\le 0\quad\mbox{for every } \bar{p}\in\mathcal{P}_2.
        \end{equation}
        Finally
        \begin{itemize}
            \item[a)] If $s\le 1/2$ or $m(p)\le n-2$, then $q$ is an $L_{a}$-harmonic polynomial and $\lambda\in \N_{\ge 2}$.
            \smallskip
            \item[b)]  If $s>1/2$ and $m(p)=n-1$, then $q$ is a solution of the very thin obstacle problem on $L(p)$, namely:
        \begin{equation}\label{def:very-thin-obstacle-problem}
        \left\{
        \begin{array}{rclll}
            -L_{a}q&=&0 &\quad \mbox{in } \R^{n+1}\setminus L(p),\\
            -L_{a}q &\ge& 0 &\quad\mbox{in } \R^{n+1},\\
            qL_{a}q &=& 0 &\quad \mbox{in } \R^{n+1},\\
            q&\ge& 0 &\quad \mbox{on } L(p),\\
            q(x,y)&=&q(x,-y) &\quad \mbox{for } (x,y)\in \R^{n+1}.
        \end{array}\right.
        \end{equation}
        Moreover, in this case, $\lambda\ge 2+\omega$ for some constant $\omega\in (0,1)$ depending only on $n$ and $s$.
        \end{itemize}
        
    \end{itemize}
\end{proposition} 
We now consider a solution $u$ of \eqref{def:soluzione-estesa} with obstacle $\vf$ satisfying \eqref{e:hypo-phi}. For every $X_{0}\in \Gamma_{2}(u)$ let us call $p^{X_{0}}\in \mathcal{P}_{2}$ the first blow-up of $\widetilde{u}^{X_{0}}$ at $X_{0}$, and $v^{X_{0}}= \widetilde{u}^{X_{0}}(X_{0}+\cdot)-p^{X_{0}}$.
In order to define a partition of the set of quadratic points $\Gamma_{2}(u)$, we take $$k\ge 4\qquad \mbox{and}\qquad \beta=3.$$ Then, thanks to \cref{second-blowup-quadratic-points}, we may define, for every $\lambda\in [2,3)$,
 $$\Gamma_{2,\lambda}(u):=\{X_0\in\Gamma_2(u): N(0^+,v^{X_0})=\lambda\}$$
 and, for every $\lambda\in[2,3]$,
 $$\Gamma_{2,\ge \lambda}(u):=\Gamma_2(u)\setminus \bigcup_{\mu<\lambda}\Gamma_{2,\mu}(u).$$
We thus introduce the sets of ordinary and anomalous quadratic points as follows:
\begin{equation*}
    \Gamma_{2}^{\rm o}(u):= \Gamma_{2,\ge 3}(u),\qquad \Gamma_{2}^{\rm a}(u):= \Gamma_{2}(u)\setminus \Gamma_{2}^{\rm o}(u)=\underset{\lambda \in [2,3)}{\bigcup}\Gamma_{2,\lambda}(u).
\end{equation*}
Finally, for a family of solutions $u:B_{1}\times [0,1]\to \R$ of \eqref{def:soluzione-estesa} and \eqref{def:soluzione-famiglia} we use the notations
\be \label{def:gammao-and-gammaa}\mathbf{\Gamma}_{2}^{o}:=\bigcup_{t\in[0,1]}\Gamma_{2}^{\rm o}(u(\cdot,t)),\quad \mathbf{\Gamma}_{2}^{\rm a}:=\bigcup_{t\in[0,1]}\Gamma_{2}^{\rm a}(u(\cdot,t)).\ee
The rest of this section is devoted to the proof of \cref{second-blowup-quadratic-points}. 
\begin{proof}[Proof of \cref{second-blowup-quadratic-points}]
We assume without loss of generality that $X_{0}=0$, and we set $\widetilde{u}:=\widetilde{u}^{0}$.
Let us start with some preliminary observations. 
First, by \eqref{def:soluzione-estesa-tilde} together with the fact that $\widetilde{u}(\cdot,0) \in C_{\loc}^{1,s}$ we get 
\bea
    \widetilde{u}L_{a}\widetilde{u}= -\widetilde{u}g|y|^{a},\quad  (\nabla\widetilde{u}\cdot X)L_{a}\widetilde{u}=-(\nabla\widetilde{u}\cdot X)g|y|^{a},
\eea where $g:=g^{0}$ is defined in \eqref{def:soluzione-estesa-tilde}.
In particular, since $0$ is a quadratic point, by \eqref{eq:stima-g} we can bound
\begin{equation}\label{bound-utildeLautilde}
\left|\int_{B_{r}}\widetilde{u}L_{a}\widetilde u\right|+\, \left|\int_{B_{r}}(\nabla\widetilde{u}\cdot X)L_{a}\widetilde{u}\right|\le Cr^{n+a+k+\gamma+1}.
\end{equation}
Moreover, since $p$ is $L_{a}$-harmonic and positive on the thin space, we have
\bea
    vL_{a}v=\widetilde{u}L_{a}\widetilde{u}+p (-L_{a}\widetilde{u})\ge -\widetilde{u}g|y|^{a}+pg|y|^{a}=-vg|y|^{a}.
\eea
Hence, recalling the definition of $F$ in \eqref{def:F} and by \eqref{eq:stima-g} again, we get
\begin{equation}\label{lowerbound-Lrv}
    F(r,v)\ge -Cr^{n+a+k+\gamma+1}.
\end{equation}

\smallskip
\noindent\textbf{Step 1. }(Monotonicity formula \eqref{eq:formula-n-tilde}).
We define $$Z(r,v):=\int_{B_r}(\nabla v\cdot X)L_av.$$ Arguing as in \cite[Lemma 2.3]{frs20}, we obtain the formula $$\widetilde N'(r,v)=2r\frac{F(r,v)^2}{\widetilde H(r,v)^2}+2\frac{rF(r,v)\widetilde D(r,v)-Z(r,v) \widetilde H(r,v)}{\widetilde H(r,v)^2}.$$   
Being $p$ a $2$-homogeneous function, $\nabla p \cdot X=2p$. Hence, using also that $L_{a}p=0$, we may re-write
\bea
    Z(r,v)=2F(r,v)+\int_{B_{r}}(\nabla \widetilde{u} \cdot X-2\widetilde{u})L_{a}\widetilde{u}.
\eea
Then, \bea \widetilde N'(r,v)&=2r\frac{F(r,v)^2}{\widetilde H(r,v)^2}+2\frac{F(r,v)(r\widetilde D(r,v)-2 \widetilde H(r,v))}{\widetilde H(r,v)^2}-\frac{\int_{B_{r}}(\nabla \widetilde{u} \cdot X-2\widetilde{u})L_{a}\widetilde{u}}{\widetilde H(r,v)}\\&=:2r\frac{F(r,v)^2}{\widetilde H(r,v)^2}+I_1(r)+I_2(r).\eea 
To bound $I_1(r)$ below we consider the two cases $F(r,v)\ge 0$ and $F(r,v)< 0$ separately.
Suppose first that $F(r,v)\ge0$. By \cref{prop:monotonicity-weiss-tilde} and the properties of $p$, $$r\widetilde D(r,v)-2\widetilde H(r,v)=r\widetilde D(r,\widetilde u)-2\widetilde H(r,\widetilde u)\ge -Cr^{n+a+k+\gamma+2}.$$ Then, in this case, from Young's inequality, $$I_{1}(r)\ge -2\frac{F(r,v)}{\widetilde H(r,v)^2}r^{n+a+k+\gamma+2}\ge -r\frac{F(r,v)^2}{\widetilde H(r,v)^2}-Cr^{2k+2\gamma+3-4\beta},$$ where we also used the lower bound $\widetilde H(r,v)\ge r^{n+a+2\beta}.$
If instead $F(r,v)<0$, $$I_1(r)\ge \frac{F(r,v)}{\widetilde H(r,v)}\frac{r\widetilde D(r,v)}{\widetilde H(r,v)} =\frac{F(r,v)}{\widetilde H(r,v)}\widetilde N(r,v)\ge -Cr^{k+\gamma+1-2\beta}\widetilde N(r,v),$$ owing to \eqref{lowerbound-Lrv}.
Hence, taking both cases into account, $$I_1(r)\ge -Cr^{k+\gamma+1-2\beta}\widetilde N(r,v)-r\frac{F(r,v)^2}{\widetilde H(r,v)^2}-Cr^{2k+2\gamma+3-4\beta}.$$
For what concerns $I_2(r)$, by \eqref{bound-utildeLautilde} we have
$$I_2(r)\ge -Cr^{k+\gamma+1-2\beta}.$$
Overall, collecting estimates together, \bea \widetilde N'(r,v)\ge r\frac{F(r,v)^2}{\widetilde H(r,v)^2}-Cr^{k+\gamma+1-2\beta}\widetilde N(r,v)-Cr^{k+\gamma+1-2\beta},\eea which concludes the proof of the monotonicity formula \eqref{eq:formula-n-tilde}.

\smallskip
\noindent \textbf{Step 2:} (Growth estimate \eqref{growth-secondblowup}). Let $0<r<R<r_{0}$ be such that $\underline{\lambda}\le N(\rho,v)\le \overline{\lambda}$ for every $\rho \in [r,R]$. We call
$G(r,v):= F(r,v)/\widetilde{H}(r,v)$. 
The following formula comes from straightforward computations:
\be\label{formula-H'}
    \widetilde H(r,v)'= \frac{n+a}{r}\widetilde H(r,v)+2\widetilde D(r,v)+2F(r,v).\ee
Then, by \eqref{lowerbound-Lrv} and \eqref{formula-H'} we have
\begin{align*}
    \frac{\widetilde{H}(\rho,v)'}{\widetilde{H}(\rho,v)}=\frac{n+a}{\rho}+\frac{2\widetilde{N}(r,v)}{\rho}+2G(r,v)\ge \frac{n+a+2\underline{\lambda}}{\rho}-Cr^{k+\gamma+1-2\beta}\quad \mbox{for } \rho \in (r,R),
\end{align*}
which, integrated, gives
\begin{align*}
    \log\left(\frac{\widetilde H(R,v)}{\widetilde H(r,v)}\right)&\ge (n+a+2\underline{\lambda})\log(R/r)-C\frac{R^{k+\gamma+2-2\beta}-r^{k+\gamma+2-2\beta}}{k+\gamma+2-2\beta}\\
    &\ge (n+a+2\underline{\lambda})\log(R/r) -C,    
\end{align*}
from which we deduce the lower bound in \eqref{growth-secondblowup}. To prove the upper bound, let us observe first that the monotonicity formula \eqref{eq:formula-n-tilde} gives
\begin{align*}
    \int_{r}^{R}G(\rho, v)\,d\rho&\le \left(\int_{r}^{R}\rho G(\rho, v)^{2}\,d\rho\right)^{1/2}\left(\int_{r}^{R}\frac{1}{\rho}\,d\rho\right)^{1/2}\\
    &\le \left(\hat{N}(R)-\hat{N}(r)\right)^{1/2}\log(R/r)^{1/2}\le C\log(R/r)^{1/2}.
\end{align*}
Therefore, as above,
\begin{align*}
    \log\left(\frac{\widetilde{H}(R,v)}{\widetilde{H}(r,v)}\right)&\le (n+a+2\overline{\lambda})\log(R/r)+C\log(R/r)^{1/2}\\
    &\le (n+a+2\overline{\lambda}+\eps)\log(R/r)+C_{\eps},
\end{align*}
where we used that $a^{1/2}\le \eps a+ C_{\eps}$, for $C_{\eps}>0$ large enough. The upper bound in \eqref{growth-secondblowup} follows after exponentiation. 
Now, \eqref{eq:lowup} is a consequence of \eqref{growth-secondblowup}. Moreover, similar computations as above give also \eqref{monneau}.

\smallskip
\noindent \textbf{Step 3:} (Existence of the second blow-up). 
Let us assume now that $p$ is the first blow-up of $u$ at $0$ and $\widetilde N(0^+,v)<\beta$.
Let $r_{j}\downarrow 0.$ From the monotonicity of $\hat N(r, v)$ in \eqref{eq:formula-n-tilde} we get that $\widetilde{N}(r_{j},v)$ is uniformly bounded. Thanks to the growth estimates in \eqref{growth-secondblowup}, since we are assuming that $\lambda<\beta$, also $N(r_{j},v)$ is uniformly bounded. Therefore, given $v_{r_j}$ as in \eqref{def:resclaings-vrj}, noting that
$$\lVert v_{r_{j}}\rVert_{L^{2}(\partial B_{1},|y|^{a})}=1,\qquad \lVert \nabla v_{r_{j}}\rVert_{L^{2}(B_{1},|y|^{a})}=N(r_{j},v)^{1/2},$$
we deduce that $v_{r_{j}}$ is uniformly bounded in $W^{1,2}(B_{1},\y).$ In particular, it weakly converges in such space up to a subsequence to some function $q$.

\smallskip
\noindent \textbf{Step 4:} (Orthogonality condition \eqref{eq:ortogonality-first-second-bu}). We proceed as in \cite[Lemma 2.1, Lemma 2.2]{frs20} or \cite[Lemma 3.3]{fj21}. First we notice that, since $2\le \widetilde N(0^+,v)=\lambda$, by \eqref{monneau}, for every $\bar p\in\mathcal{P}_2$, \bea e^{Cr^{k+\gamma+2-2\beta}}\left(\frac{1}{r^{n+a+4}}\int_{\partial B_r}(u-\bar p)^2\y\,d\HH^n+r^{2\beta-4}\right)&\ge\lim_{\rho\downarrow 0} \int_{\partial B_\rho}\left(\frac{u(\rho x)}{\rho^2}-\bar{p}\right)\y\,d\HH^n\\&= \int_{\partial B_1}(p-\bar p)^2\y\,d\HH^n.\eea We set $h_r:=\|v(r\cdot)\|_{L^2(\partial B_1,\y)}$. By \cref{prop:classification-blowup-freq-2}, $h_{r}=o(r^{2})$, and so $\eps_r:=h_r/r^2=o(1)$ as $r\downarrow 0$. Then, 
\bea e^{Cr^{k+\gamma+2-2\beta}}\int_{\partial B_1}\left(\frac{v(rx)}{r^2}+p-\bar p\right)^2\y\,d\HH^n&=\frac{e^{Cr^{k+\gamma+2-2\beta}}}{r^{n+a+4}}\int_{\partial B_r}(u-\bar p)^2\y\,d\HH^n\\&\ge \int_{\partial B_1}(p-\bar p)^2\y\,d\HH^n-e^{Cr^{k+\gamma+2-2\beta}}r^{2\beta-4}.\eea Expanding the left-hand side, we get
\bea e^{Cr^{k+\gamma+2-2\beta}}\biggl(\eps_r^2\int_{\partial B_1}&v_r^2\y \,d\HH^n +2\eps_r\int_{\partial B_1}v_r(p-\bar p)\y\,d\HH^n\biggl)\\&+(e^{Cr^{k+\gamma+2-2\beta}}-1)\int_{\partial B_1}(p-\bar p)^2\y\,d\HH^n+e^{Cr^{k+\gamma+2-2\beta}}r^{2\beta-4}\ge0.\eea
Let us choose $\eps>0$ sufficiently small such that $k+\gamma+2-2\beta>\lambda-2+\eps$. By \eqref{growth-secondblowup}, $h_r\ge C r^{\lambda+\eps}$. Then, dividing by $\eps_r$ and using that $e^{Cr^{k+\gamma+2-2\beta}}-1\approx Cr^{k+\gamma+2-2\beta}$ we get
$$O(\eps_r)+2e^{Cr^{k+\gamma+2-2\beta}}\int_{\partial B_1}v_r(p-\bar p)\y\,d\HH^n+\frac{r^{k+\gamma+2-2\beta}}{r^{\lambda+\eps-2}}O(1)+\frac{r^{2\beta-4}}{r^{\lambda+\eps-2}}O(1)\ge 0.$$ Finally, sending $r\downarrow 0$, we deduce $$\int_{\partial B_1}q(p-\bar p)\y\,d\HH^n\ge 0\quad \mbox{for every } \bar{p}\in \mathcal{P}_{2}.$$ The proof of \eqref{eq:ortogonality-first-second-bu} is then concluded after making the specific choices $\bar p=2^{-1}p$ and $\bar p=2p$.
\smallskip

\noindent \textbf{Step 5:} (Properties of the second blow-up). 
First we observe that, by \eqref{def:soluzione-estesa-tilde} and \eqref{eq:stima-g}, we have \be\label{eq:nonpositive}
-L_av_{r_j}\ge -Cr^{k+\gamma-\lambda-\eps}
\qquad\mbox{and}\qquad
|L_av_{r_j}|\le Cr^{k+\gamma-\lambda-\eps}\quad\mbox{in }\R^{n+1}\setminus \{\widetilde u(r\cdot)=0\}',
\ee where we used that $ H(r,v)\ge Cr^{\lambda+\eps}$, for some $\eps>0$ small enough, by \eqref{growth-secondblowup}.
Now, since $v_{r_{j}}$ is uniformly bounded in $W_{\loc}^{1,2}(\R^{n+1},|y|^{a})$ (see Step 3), the measures $L_av_{r_{j}}$ have locally bounded total variation, and hence locally weakly$^{*}$-converge, up to subsequences, to the measure $L_{a}q$, which is non-positive by \eqref{eq:nonpositive}.
Moreover, since $\{u(r_j\cdot)=0\}'$ converges to $L(p)$ as $r_j\downarrow 0$ in the Hausdorff distance, $L_{a}q=0$ in $\R^{n+1}\setminus L(p)$.
We now distinguish two cases.

\textit{Case a)} Suppose that $s\le 1/2$ or $m(p)\le n-2$. Then $L(p)$ is a set of zero $L_a$-harmonic capacity in $\R^{n+1}$, thus $L_aq=0$ in $\R^{n+1}$. Moreover, arguing as in \cite[Proposition 3.2]{fj21} and using \eqref{monneau}, one can prove that $q$ is $\lambda$-homogeneous. In particular, by \cref{prop:Liouville}, $q$ is a polynomial and $\lambda\in\N_{\ge2}$.

\textit{Case b)} Suppose instead that $s>1/2$ and $m(p)=n-1$. In this case we need to prove that $q$ is a solution of the very thin obstacle problem \eqref{def:very-thin-obstacle-problem}. The only properties that remain to be shown are $q\ge 0$ on $L(p)$ and $qL_{a}q\equiv 0$. The fact that $q\ge 0$ on $L(p)$ follows from a trace argument as in \cite[Proposition 3.4]{fj21}, by the weak convergence of $v_{r_j}$ to $q$, and the non-negativity of $v$ on $L(p)$.
To prove that $qL_{a}q\equiv 0$, leveraging on the weak$^{*}$ convergence of $L_{a}v_{r_{j}}$ to $L_{a}q$ and the equicontinuity of $\{v_{r_{j}}\}$ given by \cref{prop:uniform-convergence}, it only remains to show that
\begin{equation}\label{eq-convergence-intvLv}
    \int_{B_{1}}v_{r_{j}}L_{a}v_{r_{j}}\to 0.
\end{equation}
The monotonicity formula \eqref{eq:formula-n-tilde} can be re-written as follows:
\begin{equation}\label{eq:monoton-second-bu-2}
    r\hat{N}(r,v)'\ge \left(\int_{B_{1}}v_{r}L_{a}v_{r}\right)^{2}\left(\frac{H(r,v)}{\widetilde{H}(r,v)}\right)^{2}.
\end{equation}
Now, since the frequency of the second blow-up is strictly less than $\beta$, then $H(r,v)/\widetilde{H}(r,v)\to 1$ as $r\downarrow 0$. Moreover, by the mean value theorem, we can find $\bar{r}_{j}\in [r_{j},2r_{j}]$ such that $\bar r_j \hat N(\bar r_j,v)'$ converges to zero as $j\to \infty$.
Then \eqref{eq:monoton-second-bu-2} gives
$$\limsup_{j\to \infty}\left(\int_{B_{1}}v_{\bar{r}_{j}}L_{a}v_{\bar{r}_{j}}\right)^{2}= \limsup_{j\to \infty}\left(\int_{B_{1}}v_{\bar{r}_{j}}L_{a}v_{\bar{r}_{j}}\right)^{2}\left(\frac{H(\bar{r}_{j},v)}{\widetilde{H}(\bar{r}_{j},v)}\right)^{2} \le \limsup_{j\to \infty}\bar r_j \hat N(\bar r_j,v)'=0.$$
Re-scaling and using \eqref{eq:nonpositive} we get \eqref{eq-convergence-intvLv}.

\noindent The homogeneity of $q$ is obtained arguing as in the case $s\le1/2$.
Finally, the fact that $\lambda\ge 2+\omega$ for some $\omega>0$ depending only on $n$ and $s$, follows by a compactness argument as in \cite[Proposition 3.4]{fj21}.
\end{proof}
\noindent We now prove the following technical lemma that we used in the last step of the proof of \cref{second-blowup-quadratic-points}.
\begin{lemma}\label{prop:uniform-convergence}
    Let $u$ be a solution of \eqref{def:soluzione-estesa}, with obstacle $\vf$ satisfying \eqref{e:hypo-phi}, and $\|u\|_{L^\infty(B_1)}\le 1$. Suppose that $X_{0}\in \Gamma_{2}(u)$ and let $p\in \mathcal{P}_{2}$ be the first blow-up of $\widetilde{u}^{X_{0}}$ at $X_{0}$. We consider $v=\widetilde u^{X_0}(X_0+\cdot)-p$, and $v_r$ the rescalings in \eqref{def:resclaings-vrj}. We also suppose that $s>1/2$ and $m(p)=n-1$. Using the notation in \cref{second-blowup-quadratic-points}, if $\widetilde N(0^+,v)<\beta$, then, for every $\eps>0$, there is a constant $C_\eps>0$ such that $$[v_r]_{C^{-a-\eps}(B_{1/2})}\le C_\eps\quad \mbox{for every } r\in (0,r_{0}).$$ 
\end{lemma}
\begin{proof} Without loss of generality, we may assume that $X_0=0$. We denote by $g=g^{0}$ the corrective term appearing in the right-hand side of \eqref{def:soluzione-estesa}, and by $C(\vf)>0$ the constant in \eqref{eq:stima-g}. Let $$\Gamma_a(X):=C_{n,a}|X|^{-n+1-a}$$ be the fundamental solution of $L_a$ in $\R^{n+1}$ (see e.g.~\cite{cs07}).
We divide the proof in some steps.

\smallskip
\noindent\textbf{Step 1.} First of all we show that for every $e\in L(p)\cap\mathbb{S}^n$ we have 
\be\label{est:gamma}\|\Gamma_a\ast(\y g )\|_{L^\infty(B_1)}+\|\partial_{e}\Gamma_a\ast(\y g )\|_{L^\infty(B_1)}+\|\partial_{ee}\Gamma_a\ast(\y g )\|_{L^\infty(B_1)}\le CC(\vf),\ee for some constant $C>0$. 
By the definition of $\Gamma_a$, \eqref{eq:stima-g} and Young's inequality, it is sufficient to show that $\partial_{ee}\Gamma_a$ and $|x|^\gamma\y$ are respectively $L^p$ and $L^{p/(p-1)}$-integrable in $B_1$, for some $p>1$.
This is true because $$\int_{B_1}|\partial_{ee}\Gamma_a(X)|^p\,dX<+\infty \quad\mbox{for any } p<\frac{n+1}{n+1+a}$$
and $$\int_{B_1}\left(|x|^\gamma\y \right)^\frac{p}{p-1}\,dX<+\infty \quad\mbox{for any } p>\frac{n+1}{n+1+a+\gamma}.$$ 

\smallskip
\noindent\textbf{Step 2.} Next we prove $L^{\infty}$-estimates for $v$: \be\label{eq:l2linfty} \| v\|_{L^\infty(B_{1/2})} \le C\left(\| v\|_{L^2(B_1,\y)}+C(\vf)\right).
\ee
Following the lines of \cite[Lemma 2.5]{fj21},
we claim that \be\label{eq:estimates}-L_a v^{\pm}\le|y|^a\|g\|_{L^\infty(B_1)}\quad\text{in }B_1.\ee
The inequality for $v^{-}= \max\{p-\widetilde{u}, 0\}$ is a direct consequence of the fact that $p$ is $L_{a}$-harmonic, and maximum of subsolutions is a subsolution. To prove the same for $v^{+}$ it is enough to show that $-L_{a}v^{+}\le |y|^{a}\|g\|_{L^\infty(B_1)}$ in $\{v^{+}>0\}$ (see \cite[Lemma 2.5]{fj21} or \cite[Exercise 2.6]{psu12}). However, since $p\ge 0$ on $\{y=0\}$, we have $\{v^+>0\}'\subset\{\widetilde u>0\}'$, thus
$$-L_a v^+=-L_a v = \y g\le \y \|g\|_{L^\infty(B_1)}\quad\text{in }\{v^+>0\}.$$
Finally, applying to the $L_a$-subharmonic functions $v^{\pm}+C\|g\|_{L^\infty(B_1)}|X|^2$ the $L^{\infty}$-estimates from \cite[Theorem 2.3.1]{fks82}, we get 
\bea \| v^{\pm}\|_{L^\infty(B_{1/2})}&\le \|v^{\pm}+C\|g\|_{L^\infty(B_1)}|X|^2\|_{L^\infty(B_{1/2})}\le C\|v^{\pm}+C\|g\|_{L^\infty(B_1)}|X|^2\|_{L^2(B_1,\y)}\\&\le C(\| v^{\pm}\|_{L^2( B_1,\y)}+\|g\|_{L^\infty(B_1)})\le C(\| v^{\pm}\|_{L^2( B_1,\y)}+C(\vf)).\eea 

\smallskip
\noindent\textbf{Step 3.} Now we prove Lipschitz estimates for $v$ on the spine of $p$, namely, for every $e\in L(p)\cap \mathbb{S}^{n}$:
\be\label{eq:derivativel2}\|\partial_e v\|_{L^\infty(B_{1/2})}\le C \left(\|v\|_{L^2(B_1,\y)}+C(\vf)\right).
\ee
We proceed as in \cite[Lemma 2.8]{fj21}, using the Bernstein's technique from \cite{cdv20}.
We consider the function $w:=v-\Gamma_a\ast(\y g),$ which satisfies $L_aw=L_a(\partial_ew)=0$ in $B_{1/2}\setminus \Lambda(u)$. We also set
$$\psi:=\eta^2(\partial_e w)^2+\mu w^2,$$ where $\eta\in C^\infty_c(B_{1/2})$ is a cut-off function, even in $y$, such that $\eta\in [0,1]$ and $\eta\equiv 1$ in $B_{1/4}$. One can prove (see \cite{cdv20} or \cite[Lemma 2.8]{fj21}) that 
$-L_a\psi\le0$ in $B_{1/2}\setminus \Lambda(u)$,
provided that $\mu>0$ is chosen sufficiently large. 
Then, by the maximum principle in $B_{1/2}\setminus\Lambda(u)$, using that $\eta\equiv0$ on $\partial B_{1/2}$ and $\partial_e\widetilde u=\partial_ep=0$ on $\Lambda(u)$ as well as \eqref{est:gamma}, we obtain 
\bea \|\psi\|_{L^\infty(B_{1/4})}\le\|\psi\|_{L^\infty(\partial B_{1/2})}+\|\psi\|_{L^\infty(\Lambda(u))}\le C\left(\|v\|_{L^\infty(B_{1/2})}+C(\vf)\right)^{2}.\eea
Then, since $\eta\equiv1$ in $B_{1/4}$, we get
\begin{equation*}
    \|\partial_ev\|_{L^\infty(B_{1/4})}\le \|\partial_ew\|_{L^\infty(B_{1/4})}+CC(\vf)\le \|\psi\|_{L^\infty(B_{1/4})}^{1/2}+CC(\vf)\le C\left(\|v\|_{L^\infty(B_{1/2})}+C(\vf)\right),
\end{equation*}
where we used \eqref{est:gamma} again. Then \eqref{eq:derivativel2} follows from \eqref{eq:l2linfty} and a standard covering argument.

\smallskip
\noindent\textbf{Step 4.} In this step we prove semiconvexity estimates for $v$ on the spine of $p$, that is, for every $e\in L(p)\cap \mathbb{S}^n$:
\be\label{eq:infderivativeee} \inf_{B_{1/2}}\partial _{ee}v\ge -C\left(\|v\|_{L^2(B_1,\y)}+C(\vf)\right).
\ee 
We proceed as in \cite[Lemma 2.9]{fj21}. 
For every $\ell>0$, we define $ \widetilde u_\ell$ the solution of
\begin{equation*}
    \left\{
\begin{array}{rclll}
    -L_a  \widetilde u_\ell(x,y)&=& g(x)\y  & \quad \mbox{for } (x,y)\in B_{7/8}\setminus \Lambda( \widetilde u_\ell),\\
            -L_a  \widetilde u_\ell (x,y) &\ge& g(x)\y & \quad \mbox{for }(x,y)\in B_{7/8},\\
		 \widetilde u_\ell &\ge& 0 & \quad \mbox{on } B'_{7/8},\\
		 \widetilde u_\ell&= &\widetilde u+\ell &\quad \mbox{on } \partial B_{7/8}.
\end{array}
    \right.
\end{equation*}
Observe that $\widetilde u_{\ell}$ is $L_{a}$-harmonic in $B_{7/8}\setminus B_{7/8-2\beta}$, for some $\beta=\beta(\ell)>0$, and $\widetilde u_{\ell}\downarrow \widetilde u$ uniformly in $B_{7/8}$ as $\ell\downarrow 0$.
We then consider $w_{\ell}:=\widetilde u_\ell-p-\Gamma_a\ast(\y g)$, which satisfies $L_{a}w_{\ell}=L_{a}(\partial_{e}w_{\ell})=0$ in $B_{7/8}\setminus \Lambda(\widetilde{u}_{\ell})$. 
We claim that, for every $\ell>0$,
\begin{equation}\label{claim:semiconvexity}
    -L_{a}(\partial_{ee}w_{\ell})^{-}\le 0\quad \text{in } B_{3/4}\setminus \Lambda(\widetilde{u}_{\ell}),\qquad (\partial_{ee}w_{\ell})^{-}\le  CC(\vf)\quad \text{on } \Lambda(\widetilde{u}_{\ell}).
\end{equation}
For any $w:B_1\to\R$, $e\in L(p)\cap \mathbb{S}^n$ and $h\in(0,1/8)$, we call
$$\delta^2_{e,h} w(X)=\frac{ w(X+he)+ w(X-he)-2 w(X)}{h^2}\quad\mbox{for every } X\in B_{7/8}.$$
We notice that, since $\widetilde u_\ell= 0$ on $\Lambda(\widetilde u_\ell)$ and $e\in L(p)$, then \be\label{up}\delta^2_{e,h} \widetilde u_\ell\ge 0\quad\text{on }\Lambda(\widetilde u_\ell) \qquad\text{and}\qquad \delta^2_{e,h}p=0\quad\text{on }\Lambda(\widetilde u_\ell),
\ee
We also define $f^\ell_{\eps,e,h}:=\max\{-\delta^2_{e,h} w_\ell,\eps\}$, for every $\eps>0$. We observe that $-L_a(\delta^2_{e,h}w_\ell)\ge 0$ in $B_{7/8}\setminus \Lambda(\widetilde u_\ell)$. Since the maximum of subsolutions is a subsolution, then $$-L_a(f^\ell_{\eps,e,h})\le 0\quad\text{on }B_{7/8}\setminus \Lambda(\widetilde u_\ell).$$ Moreover, by \eqref{est:gamma} and \eqref{up}, we obtain $$f^\ell_{\eps,e,h}\le (\delta^2_{e,h}w_\ell)^-+\eps\le(\delta^2_{e,h}(\Gamma_a\ast (\y g)))^++\eps\le CC(\vf)+\eps\quad\mbox{on }\Lambda(\widetilde u_\ell).$$
In order to prove the two estimates in \eqref{claim:semiconvexity}, we pass to the limit as $\eps\downarrow 0$ and then $h\downarrow0$ in the last two inequalities. In particular, the second estimate in \eqref{claim:semiconvexity} holds, while, for the first estimate in \eqref{claim:semiconvexity}, it is sufficient to show that \be\label{claim:seminconvexity2}
\|f_{\eps,e,h}^\ell\|_{L^\infty(B_{3/4}\setminus\Lambda(\widetilde u_\ell))}\le C(\ell),
\ee for some $C(\ell)>0$ which does not depend on $\eps$ and $h$. 
Since $\widetilde u_\ell$ is $L_a$-harmonic in $B_{7/8}\setminus B_{7/8-2\beta}$, then, by \eqref{est:gamma} and \eqref{up}, we have
\be\label{est1}\|f_{\eps,e,h}^\ell\|_{L^\infty(\partial B_{7/8-\beta})}\le \|\delta^2_{e,h}w_\ell\|_{L^\infty(\partial B_{7/8-\beta})}+\eps \le \|\delta^2_{e,h}\widetilde u_\ell\|_{L^\infty(\partial B_{7/8-\beta})}+CC(\vf)+\eps\le C(\beta)+CC(\vf),\ee by the $C^2$ estimates in the tangential direction for $ L_a$-harmonic functions.
Moreover, if $\bar X\in \Lambda(\widetilde u_\ell)$ such that $f_{\eps,e,h}^\ell(\bar X)=\|f_{\eps,e,h}^\ell\|_{L^\infty(\Lambda(\widetilde u_\ell))}>\eps$, then $$0\le f_{\eps,e,h}^\ell(\bar X)=-\delta^2_{e,h}w_\ell(\bar X)\le \delta^2_{e,h}(\Gamma_a\ast(\y g)(\bar X)) \le CC(\vf)\quad\text{on }\Lambda(\widetilde u_\ell),$$ by \eqref{est:gamma} and \eqref{up}. Then \be\label{est2}\|f_{\eps,e,h}^\ell\|_{L^\infty(\Lambda(\widetilde u_\ell))}\le C(\beta)+CC(\vf).\ee
Since $f^\ell_{\eps,e,h}$ is $ L_a$-subharmonic, by the maximum principle, using \eqref{est1} and \eqref{est2}, we obtain $$\|f_{\eps,e,h}^\ell\|_{L^\infty(B_{3/4}\setminus\Lambda(\widetilde u_\ell))}\le\|f_{\eps,e,h}^\ell\|_{L^\infty(\partial B_{7/8-\beta})}+\|f_{\eps,e,h}^\ell\|_{L^\infty(\Lambda(\widetilde u_\ell))}\le C(\ell)$$ 
This is exactly \eqref{claim:seminconvexity2}, which concludes the proof of the claim \eqref{claim:semiconvexity}.

Once \eqref{claim:semiconvexity} is known, we may proceed by Bernstein's technique again, using this time the auxiliary function $$\psi_\ell:=\eta^2((\partial_{ee}w_\ell)^-)^2+\mu(\partial_ew_\ell)^2,$$ where $\eta\in C_c^\infty(B_{1/2})$ is the same cut-off function as in the previous step, namely $\eta$ is even in $y$, with $\eta\in [0,1]$ and $\eta\equiv 1$ in $B_{1/4}$. 
One can prove (see \cite{cdv20} or \cite[Lemma 2.9]{fj21}) that $-L_{a}\psi_{\ell}\le 0$ in $B_{1/2}\setminus \Lambda(\widetilde{u}_{\ell})$, provided that $\mu>0$ is chosen large enough (independently of $\ell$). 
Then, by the maximum principle in $B_{1/2}\setminus\Lambda(\widetilde{u}_{\ell})$, using that $\eta\equiv0$ on $\partial B_{1/2}$ and that $\partial_e\widetilde u=\partial_e p=0, (\partial_{ee}w_{\ell})^{-}\le  CC(\vf)$ on $\Lambda(\widetilde{u}_{\ell})$ as well as \eqref{est:gamma}, we obtain
\bea \|\psi_\ell\|_{L^\infty(B_{1/4})}\le\|\psi_\ell\|_{L^\infty(\partial B_{1/2}\setminus \Lambda(\widetilde{u}_{\ell}))}+\|\psi_\ell\|_{L^\infty(\Lambda(\widetilde{u}_{\ell}))}\le C\left(\lVert \partial_ew_\ell\rVert_{L^{\infty}(B_{1/2})}+C(\vf)\right)^{2}.
\eea
Then, since $\eta\equiv1$ in $B_{1/4}$, we get
\begin{align*}
    \lVert (\partial_{ee}(\widetilde{u}_{\ell}-p))^{-}\rVert_{L^{\infty}(B_{1/4})}&\le \lVert (\partial_{ee}w_{\ell})^{-}\rVert_{L^{\infty}(B_{1/4})}+CC(\vf)\le \lVert \psi_{\ell}\rVert_{L^{\infty}(B_{1/4})}^{1/2}+CC(\vf)\\ &\le C\left(\lVert \partial_e(\widetilde{u}_{\ell}-p)\rVert_{L^{\infty}(B_{1/2})}+C(\vf)\right),
\end{align*}
where we used \eqref{est:gamma} again. \eqref{eq:infderivativeee} is finally obtained after applying step 3, passing to the limit as $\ell \downarrow 0$ and using a standard covering argument.

\smallskip
\noindent\textbf{Step 5.} In this last step, we conclude the proof of \cref{prop:uniform-convergence}, following the lines of \cite[Proposition 3.4]{fj21}.
We can apply the estimate \eqref{eq:infderivativeee} in the last step to the function $u(r\cdot)/\|v(r\cdot)\|_{L^2(\partial B_1,\y)}$ which is a solution of \eqref{def:soluzione-estesa-tilde} with obstacle $\vf(r\cdot)/\|v(r\cdot)\|_{L^2(\partial B_1,\y)}$. Since $\widetilde N(0^+,v)<\beta$, by \eqref{eq:nonpositive} the corresponding right-hand side vanishes as $r\downarrow 0$ and so $$C\left(\frac{\vf(r\cdot)}{\|v(r\cdot)\|_{L^2(\partial B_1,\y)}}\right)\to 0\quad\text{as }r\downarrow 0.$$
Using also that $\|v_r\|_{L^2(\partial B_1,\y)}=1$, from the previous stap we get \be\label{semiconv} \inf_{B_{1/2}}\partial _{ee}v_r\ge -C\quad\mbox{for every } e\in L(p)\cap \mathbb{S}^n,\ee i.e.~$v_r$ are locally uniformly semiconvex in the directions parallel to $L(p)$.

Without loss of generality, we may assume that $L(p)=\{x_n=y=0\}$. Now we define $Q_{1}:=B_{1}''\times D_{1}\subset \R^{n-1}\times \R^2$ and we write $X=(x'',x_{n},y)\in\R^{n-1}\times \R\times \R$.
We also set $$w_r:=v_r-\frac{\Gamma_a\ast (\y g)(r\cdot)}{\|v(r\cdot)\|_{L^2(\partial B_1,\y)}}.$$ Then, the functions $w_r$ satisfy
\be\label{w_r:est}\lim_{y\downarrow0}y^a\partial_y w_r\le0\qquad\text{and}\qquad L_aw_r=0\quad\text{in }B_1\cap \{|y|>0\}\ee
and, by \eqref{semiconv}, the functions $w_r$ are equi-Lipschitz, and so 
\be\label{w_r:est2}\|w_r(x'',\cdot,\cdot)\|_{L^2(D_1,\y)}\le C\ee for some constant $C>0$ which not depends on $r$.
A simple integration by parts as in \cite[Proposition 3.4]{fj21}, together with \eqref{est:gamma}, \eqref{semiconv}, \eqref{w_r:est} and \eqref{w_r:est2}, shows that the measure
$$\mu_r:=\lim_{y\downarrow 0}y^a\partial_{y}w_r$$ is finite on each $x''$ slice, or equivalently \be\label{eq:estim}0\ge \int_{-1}^1\zeta(|(x'',x_n)|)\mu_r(x'',x_{n})\,dx_n\ge -C,\ee for some constant $C>0$ which does not depend on $r$. Here $\zeta:[0,+\infty]\to[0,1]$ is such that $\zeta\equiv 1$ in $[0,1/2]$ and $\zeta\equiv0$ in $[3/4,+\infty)$. 
We set $\bar w_r:=\Gamma_a(\cdot,y)\ast_x(\zeta \mu_r)$, then \be\label{barw_r}\lim_{y\downarrow0}y^a\partial_y \bar w_r=\zeta \mu_r\qquad\text{and}\qquad L_a\bar w_r=0\quad\text{in }B_1\cap \{|y|>0\}.\ee
By \eqref{eq:estim}, we see that $\bar w_r$ is uniformly bounded. Moreover $$(-\Delta)^{\bar s}_X \bar w_r=(-\Delta)^{\bar s}_X\Gamma_a\ast_x(\zeta \mu_r),$$ then, if $2\bar s<-a$, since $(-\Delta)^{\bar s}_X \Gamma_a=C|X|^{-n+1-a-2\bar s}$ and by Young inequality, we get $ (-\Delta)^{\bar s}_X \bar w_r$ is uniformly bounded provided $2\bar s<-a$. By interior regularity for the fractional Laplacian (see \cite[Theorem 1.1]{rs16}), we get the sequence $\bar w_r\in C^{2\overline s}_{\loc}(B_{1})$, with uniform estimates.
Finally, by \eqref{w_r:est} and \eqref{barw_r}, the functions $w_r-\bar w_r$ are $L_a$-harmonic in $B_{1/2}$ and thus $w_r-\bar w_r\in C^{1}_{\loc}(B_{1/2})$ with uniform estimates. This implies that $w_r$ (and consequently $v_r$) is $C^{2\bar s}$ regular, with uniform estimates, as desired.
\end{proof}

\section{Points with frequency \texorpdfstring{$2m+2s$}{2m+2s}}\label{section-odd}

In this section we prove the results stated in \Cref{subsec:2m+2sIntro} for points with frequency $2m+2s$. 
We follow the strategy of \cite{cv24}, which is based on the proof of an epiperimetric inequality. In \Cref{subsection:classification_bu_2m+2s} we show that the admissible blow-ups at $(2m+2s)$-frequency points are exactly the functions in $\mathcal{P}_{2m+2s}$, defined in \eqref{eq:def-p-2m+2s}; then, in \Cref{subsection:epiperimetric}, we prove epiperimetric inequalities (see \cref{thm:epi}) for the Weiss energy $W_{2m+2s}$; 
finally, in \Cref{subsection:proofs_2m+2s} we apply the epiperimetric inequalities at all sufficiently small scales to prove \cref{prop:rate}.

\subsection{Characterization of the admissible blow-ups}\label{subsection:classification_bu_2m+2s}
In this subsection we prove that the admissible blow-ups belong to the class $\mathcal{P}_{2m+2s}$.
We first build an explicit homogeneous solution of \eqref{def:soluzione-estesa} with zero obstacle, which will be used as a competitor.
\begin{lemma}
    \label{lemma:explicit-(2m+2s)}
    For a given $m\in \N$, we consider the sequence of numbers $\gamma_{0},\dots,\gamma_{m}$ defined recursively as follows:
    \begin{equation*}
    \left\{
\begin{array}{rclll}
     \gamma_{0}&=&-1,\\
            \gamma_{k+1}&=&-\frac{(2m-2k)(2m-2k-1)}{(2k+2)(2k+2+2s)}\gamma_{k},\quad \mbox{for } k=0,\dots,m-1.
\end{array}
    \right.
    \end{equation*}
    Then, the $(2m+2s)$-homogeneous function
    \begin{equation*}
        P(x,y):=|y|^{2s}\overunderset{m}{k=0}{\sum}\gamma_{k}y^{2k}\overunderset{n}{j=1}{\sum}x_{j}^{2m-2k}
    \end{equation*}
    is a solution of \eqref{def:soluzione-estesa} with zero obstacle. Moreover, 
    \begin{equation*}
       \left\{\begin{array}{rclll}
            P(x,0)&=&0&\quad \text{for } x\in \R^{n},\\
            \lim_{y\downarrow 0}y^{a}\partial_{y}P(x,y)&<&0&\quad \text{for }x\in \R^{n}\setminus \left\{0\right\}.
        \end{array}\right.
    \end{equation*}
\end{lemma}
\begin{proof}
    To check that $L_{a}P = 0$ in $\R^{n+1}\setminus \{y=0\}$ it is enough to use the expression $$L_{a}=|y|^{a}\left(\Delta_{x}+\partial^{2}_{yy}+\frac{a}{y}\partial_{y}\right)$$ on each homogeneous component in the sum and exploit the definition of the sequence $\gamma_{0},\dots,\gamma_{m}$. The fact that $P$ vanishes on the thin space is obvious by definition. Finally, we compute, on $\R^{n+1}_{+}$,
    \begin{equation*}
        y^{a}\partial_{y}P(x,y)=\overunderset{m}{k=0}{\sum}\gamma_{k}(2k+2s)y^{2k}\overunderset{n}{j=1}{\sum}x_{j}^{2m-2k}.
    \end{equation*}
    Thus, 
    \begin{equation*}
         \underset{y\downarrow 0}{\lim}y^{a}\partial_{y}P(x,y)=2s\gamma_{0}\overunderset{n}{j=1}{\sum}x_{j}^{2m},
    \end{equation*}
    which is strictly negative if $x\neq 0$ because $\gamma_{0}=-1$. 
\end{proof}
We are now ready to give the aforementioned characterization of homogeneous global solutions.
\begin{proposition}\label{prop:zero-thin-space-(2m+2s)}
    Let $p:\R^{n+1}\to \R$ be a $\kappa$-homogeneous solution of \eqref{def:soluzione-estesa} with zero obstacle. Then,
    \begin{itemize}
        \item [i)] $\kappa = 2m+2s$ for some $m\in \N$ if and only if $p\equiv 0$ on $\R^{n+1}_{0}$;
        \smallskip
        \item [ii)] if the conditions in i) hold, then $p\in\mathcal{P}_{2m+2s}$. Moreover, $p$ can be written (uniquely) as   \begin{equation}\label{decomposition_polynomials_2m+2s}
            p(x,y)=-|y|^{2s}(p_0(x)+y^2p_1(x,y)),
        \end{equation}
        for some homogeneous polynomials $p_0, p_{1}$, with $p_{0}\ge 0$ and $p_{1}$ even in $y$.
    \end{itemize}
\end{proposition}
\begin{proof}
    Let $p$ be a $(2m+2s)$-homogeneous global solution of \eqref{def:soluzione-estesa}. We wish to prove that $p$ vanishes identically on the thin space.
    Let $P$ be the explicit $(2m+2s)$-homogeneous solution from \cref{lemma:explicit-(2m+2s)}. Given any cut-off function $\psi \in C_{c}^{\infty}([0,\infty))$, we call $\Psi(X):=\psi(|X|)$ and compute
    \begin{align*}
        2\int_{\R^{n+1}_{0}}\left(\lim_{y\downarrow 0}y^{a}\partial_{y}P\right)\Psi p \,d\mathcal{H}^{n}&= \int_{\R^{n+1}} \Psi p L_{a}P \\
        &=-\int_{\R^{n+1}}|y|^{a}\nabla P \cdot \nabla \Psi p\,dX-\int_{B_{1}}|y|^{a}\nabla P \cdot \nabla p \Psi\,dX \\
        &=\int_{\R^{n+1}}|y|^{a}\Big(-\nabla P \cdot \nabla \Psi p + \nabla p \cdot \nabla \Psi P\Big)\,dX+ \int_{\R^{n+1}}P \Psi L_{a}p\\
        &=\int_{\R^{n+1}}|y|^{a}\frac{\psi'(|X|)}{|X|}\Big(\nabla P \cdot X p- \nabla p \cdot X P\Big)\,dX =0.
    \end{align*}
    We used formula \eqref{formula:y-derivative-on-thin-space}, integration by parts, and the following facts:
    \begin{equation*}
        PL_{a}p\equiv 0,\qquad \nabla P \cdot X=(2m+2s)P,\qquad \nabla p\cdot X=(2m+2s)p.
    \end{equation*}
    From the fact that $\lim_{y\downarrow 0}y^{a}\partial_{y}P(x,y)<0$ for every $x\in \R^{n}\setminus \left\{0\right\}$,  that $p\ge 0$ on $\R^{n+1}_{0}$ and the arbitrariness of $\psi$, we deduce that $p$ is identically zero on the thin space.

    Conversely, suppose that $p$ vanishes on the thin space. Then, by \cref{prop:odd-extension} and \cref{prop:Liouville} we have that $$p(x,y)=|y|^{2s}q(x,y),$$ for some polynomial $q$ even in $y$ and homogeneous of degree $\ell \in \mathbb{N}$. We claim that $\ell$ is even. Suppose by contradiction that $\ell=2m+1$ for some $m\in \mathbb{N}$. Then, $p$ would have the form
    $$p(x,y)=|y|^{2s}\overunderset{m}{k=0}{\sum} q_{k}(x)y^{2k},$$
    for some polynomials $q_{k}$ in $\R^{n}$ with odd homogeneities $2m-2k+1$. We may compute
    $$\lim_{y\downarrow 0}y^{a}\partial_{y}p(x,y)=\underset{y\downarrow 0}{\lim}\overunderset{m}{k=0}{\sum}(2k+2s)q_{k}(x)y^{2k}=2sq_{0}(x).$$
    Since $p$ solves \eqref{def:soluzione-estesa}, we must have $q_{0}\le 0$ in $\R^{n}$, which implies that $q_{0}\equiv 0$. This in turn yields that $p$ is $L_a$-harmonic in the whole $\R^{n+1}$, and so, by \cref{prop:Liouville}, $p$ must be a polynomial. Finally by \cref{lemma:unique-extension-polynomial-aHarmonic}, $p\equiv 0$ since it vanishes on the thin space, a contradiction.
    \end{proof}
\noindent When $m=1$, we derive from \cref{prop:zero-thin-space-(2m+2s)} an even more precise characterization of homogeneous blow-ups.
\begin{corollary}\label{corollary:classification-(2+2s)}
    Let $p:\R^{n+1}\to \R$ be a $(2+2s)$-homogeneous solution of \eqref{def:soluzione-estesa} with zero obstacle. Then, 
    $$p(x,y)=|y|^{2s}\left(x\cdot Ax-by^{2}\right),$$
    where $A$ is a non-negative definite symmetric $(n\times n)$-matrix and $b=\Tr A /(2+2s)>0$. 
\end{corollary}
From Proposition \ref{prop:zero-thin-space-(2m+2s)} and the uniform convergence to the blow-up, one immediately gets the following characterization of $(2m+2s)$-frequency points (see \cite{survey} for the corresponding result in the case $s=1/2$).
\begin{corollary}
Let $u$ be a solution of \eqref{def:soluzione-estesa}, then we have
\bea
    \underset{m\ge 1}{\bigcup}\Gamma_{2m+2s}(u)= \left\{X\in \Gamma(u): \underset{r\downarrow 0}{\lim}\,\frac{\mathcal{H}^{n}(\Lambda(u) \cap B_{r}(X))}{\mathcal{H}^{n}\left(\left\{y=0\right\}\cap B_{r}(X)\right)}=1\right\}.
\eea
\end{corollary}

\subsection{The epiperimetric inequality}\label{subsection:epiperimetric}
In this subsection we prove epiperimetric inequalities for the Weiss energy $W_{2m+2s}$ (see \eqref{weiss} for its definition), which are the main tools in the derivation of \cref{prop:rate} and the frequency gap around $2m+2s$ (see \cref{prop:gap-odd}). 

By \cref{prop:zero-thin-space-(2m+2s)}, given $p\in \mathcal{P}_{2m+2s}$, there are unique homogeneous polynomials $p_0(x), p_{1}(x,y)$, with $p_{0}\ge 0$ and $p_{1}$ even in $y$ such that $p(x,y)=-|y|^{2s}(p_0(x)+y^2p_1(x,y))$. Therefore, the operator
\begin{equation}\label{e:definition-of-T}
T[p]:=p_0= -\frac{\lim_{y\downarrow 0}y^a\partial_{y}p(\cdot, y)}{2s}
\end{equation}
is well-defined. We will prove the following proposition.

\begin{proposition}[Epiperimetric inequalities for $W_{2m+2s}$] \label{thm:epi} 
     There are constants $\delta>0$, $\eps>0$ and $\kappa>0$ depending only on $n$ and $s$ such that the following holds.

    Given $c\in W^{1,2}(\partial B_1,\y)$, with $c\ge 0$ on $\partial B_1'$ and $c$ even in $y$, let $z(r,\theta)=r^{2m+2s}c(\theta)$ be the $(2m+2s)$-homogeneous extension in $\R^{n+1}$ of $c$. We suppose that  
     \be\label{eq:quasizero} c\equiv 0 \quad\text{on } \{T[p]\ge\delta\}\cap \partial B_1'\quad \mbox{for some } p\in \mathcal{P}_{2m+2s}, \ee 
     where $\|p\|_{L^2(\partial B_1,|y|^{a})}=1$ and $T$ is the operator in \eqref{e:definition-of-T}. Then the following epiperimetric inequalities hold.
     \begin{itemize}
     \smallskip
         \item [i)] (Positive energies). There is a function $\zeta\in W^{1,2}(B_1, |y|^{a})$, even in $y$, such that $\zeta\ge 0$ on $B_1'$, $\zeta=c$ on $\partial B_1$, and 
     \be\label{eq:epi}W_{2m+2s}(\zeta)\le (1-\kappa)W_{2m+2s}(z).\ee 
         \item [ii)] (Negative energies). If $|W_{2m+2s}(z)|\le \eps$ and $\|c-p\|_{L^2(\partial B_1,\y)}\le \eps$, then there is a function $\hat{\zeta}\in W^{1,2}(B_1,\y)$, even in $y$, such that $\hat{\zeta}\ge 0$ on $B_1'$, $\hat{\zeta}=c$ on $\partial B_1$, and
     \be\label{eq:epi-neg}
     W_{2m+2s}(\hat{\zeta})\le (1+|W_{2m+2s}(z)|)W_{2m+2s}(z).
     \ee
     \end{itemize}
 \end{proposition}

\smallskip
\begin{remark}
    For simplicity,  
    in the proof of \cref{thm:epi} we fix $p\in\mathcal{P}_{2m+2s}$ with $\lVert p\rVert_{L^{2}(\partial B_{1},|y|^{a})}=1$ and we derive a small constant $\delta(p)>0$ such that \eqref{eq:epi} and \eqref{eq:epi-neg} hold for every admissible trace $c\in W^{1,2}(\partial B_1,\y)$ satisfying \eqref{eq:quasizero}. From this, noting that the constant $\delta(p)$ from \cref{lemma:convergence-of-s_delta} has a continuous dependence on $p$, by compactness, one we may then find a small $\delta>0$ independent of $p$, as stated in \cref{thm:epi}.
\end{remark}
\noindent Before proving \cref{thm:epi} we need to introduce a few notations and some constructions.

Let us fix a function $p\in\mathcal{P}_{2m+2s}$ with $\|p\|_{L^2(\partial B_1,|y|^{a})}=1$. We define, for every $\delta \ge 0$, the following subsets of the unit sphere:
   $$\mathcal{Z}_\delta:=\{T[p]\ge\delta\}\cap \partial B_1',\qquad S_\delta:=\partial B_1\setminus \mathcal{Z}_\delta.$$  
   Let $L_a^{\mathbb S^{n}}$ be the trace of the operator $L_a$ on $\partial B_1$. On the Hilbert space
   \bea W^{1,2}_0(S_\delta,\y):=\{\phi \in W^{1,2}(\partial B_1,|y|^a): \phi=0 \mbox{ on }\mathcal{Z}_\delta, \phi \mbox{ even in } y\},\eea
  we have that $-L_a^{\mathbb S^{n}}$ is a positive self-adjoint compact operator. In particular, there is a non-decreasing sequence of eigenvalues (counted with multiplicity) 
	$$0\le\lambda_1^\delta\le\lambda_2^\delta\le\ldots\le\lambda_j^\delta\le\ldots$$ 
 and a sequence of eigenfunctions $\{\phi_j^{\delta}\}\subset W^{1,2}_0(S_\delta,\y)$, forming an orthonormal basis of $W^{1,2}_{0}(S_\delta,\y)$, such that, for every $j\ge 1$: \bea
 \left\{
 \begin{array}{rclll}
      -L_a^{\mathbb S^{n}} \phi_j^{\delta}&=&\lambda_j^\delta \phi_j^{\delta}\y&\quad \mbox{in }S_\delta,\\
        \phi_j^{\delta}&=&0&\quad \mbox{on }\mathcal{Z}_\delta.\\
 \end{array}
 \right.
 \eea
 
When $\delta=0$, we simply write $\lambda_{j}$ and $\phi_{j}$ in the place of $\lambda_{j}^{\delta}$ and $\phi_{j}^{\delta}$. In this case, $W^{1,2}_{0}(S_{0},|y|^{a})$ corresponds to the space of even extensions of Sobolev functions in the half sphere $\partial B_{1}^{+}$ which are zero on $\partial B_{1}'$. 

We recall the following fact concerning the relationship between homogeneous $L_{a}$-harmonic funcitons in $\R^{n+1}$ and their traces on the unit sphere. Given $\phi \in W^{1,2}_{0}(S_{0},|y|^{a})$ and $\alpha \ge 0$, then 
\begin{gather*}
    r^{\alpha}\phi(\theta)\text{ is $L_a$-harmonic in $\{y\neq 0\}$} \iff
    -L_{a}^{\mathbb{S}^{n}}\phi=\lambda(\alpha)\phi|y|^{a} \text{ in $S_{0}$, with $\lambda(\alpha):=\alpha(n+a+\alpha-1)$}. 
\end{gather*}
In this case $r^{\alpha}\phi(\theta)$ is a polynomial multiplied by $|y|^{2s}$ and $\alpha\in \mathbb{N}+2s$, as can be easily checked using \cref{prop:odd-extension} and \cref{prop:Liouville}.
In particular:
\begin{itemize}
		\item $\lambda_1=\lambda(2s)$ and the corresponding eigenfunction $\phi_1$ is a multiple of $|y|^{2s}$;
        \smallskip
		\item $\lambda_2=\ldots=\lambda_{n+1}=\lambda(1+2s)$ and the corresponding eigenspace, of dimension $n$, is generated by the restriction to $\partial B_1$ of $L_{a}$-harmonic functions in $\{y\neq 0\}$ obtained from the multiplication of linear functions with $|y|^{2s}$;
		\smallskip
		\item in general, there is an explicit function $f:\N\to\N$ such that 
     $$\lambda_{f(j-1)+1}=\ldots=\lambda_{f(j)}=\lambda(j+2s),$$ 
  and the corresponding eigenspace, of dimension $f(j)-f(j-1)$, is generated by the restriction to $\partial B_1$ of $L_a$-harmonic functions in $\{y\neq 0\}$ obtained from the multiplication of $j$-homogeneous polynomials with $|y|^{2s}$.
\end{itemize}
Notice that the trace of the function $p\in \mathcal{P}_{2m+2s}$ itself is an eigenfunction of $-L_{a}^{\mathbb{S}^{n}}$ in $W^{1,2}_{0}(S_{0},|y|^{a})$ with eigenvalue $\lambda(2m+2s)$. Therefore, calling
\bea \ell:=f(2m)\eea
    the number of eigenvalues with homogeneity less than or equal to $2m+2s$, we may assume without loss of generality that 
    \bea
        p= r^{2m+2s}\phi_{\ell}(\theta).
    \eea
    
    The first step in the construction of the competitor $\zeta$ in \cref{thm:epi} consists in decomposing the trace $c$ with respect to the hybrid system $\{\phi_{j}\}_{j=1}^{\ell}\cup \{\phi^{\delta}_{j}\}_{j=\ell+1}^{\infty}$. This is possible whenever $\delta$ is small enough, owing to the stability of eigenvalues and eigenfunctions of $L_{a}^{\mathbb{S}^{n}}$ with respect to domain variation (see \cite[Proposition 3.1, Lemma 3.2]{cv24}). 
\begin{lemma}\label{lemma:convergence-of-s_delta} 
There exists $\delta>0$ for which we can chose $\{\phi_{j}^{\delta}\}$ in such a way that the following hold.
\begin{itemize}
    \item [i)] The linear map $F:\R^\ell\to\R^\ell$ defined as
$$F(\nu):=\left(\int_{\partial B_1}p_\nu \phi_1^\delta\y\,d\HH^n,\ldots,\int_{\partial B_1}p_\nu \phi_\ell^\delta\y\,d\HH^n\right),\qquad p_{\nu}:= \overunderset{\ell}{j=1}{\sum}\nu_{j}\phi_{j}$$ 
is invertible.
\smallskip
\item [ii)] $\lambda_j^\delta\ge \lambda(2m+2s+1/2)$ for every $j>\ell$.
\end{itemize}
     \end{lemma}
The next lemma, which is obtained exactly as \cite[Lemma 2.11]{car24}, contains explicit formulas for the Weiss energy of homogeneous functions.
\begin{lemma}\label{lemma:formulas-homog-epi}
     Let $\psi\in W^{1,2}_0(S_\delta,\y)$ be decomposed as follows: $$\psi(\theta)=\sum_{j=1}^\infty c_j \phi_j^\delta(\theta),$$
     where $\{\phi_j^\delta\}$ is an ortonormal basis of $W^{1,2}_0(S_\delta,\y)$ of eigenfunctions for $-L_a^{\mathbb S^{n}}$. Then, the following formulas hold
     \begin{gather*}
         W_\mu(r^\mu \psi)=\frac{1}{n+a+2\mu-1}\sum_{j=1}^\infty(\lambda_j^\delta-\lambda(\mu))c_j^2, \label{eq:lower_modes}\\
         W_\mu(r^\alpha\psi)-(1-\kappa_{\alpha,\mu})W_\mu(r^\mu\psi)=\frac{\kappa_{\alpha,\mu}}{n+a+2\alpha-1}\sum_{j=1}^\infty(\lambda(\alpha)-\lambda_j^\delta)c_j^2, \label{eq:higher_modes}
     \end{gather*}
     where we set \be\label{eq:kappa_alpha_mu}\kappa_{\alpha,\mu}:=\frac{\alpha-\mu}{n+a+\alpha+\mu-1}.\ee
\end{lemma}

For any two given functions $v,w\in W^{1,2}(B_1,\y)$, we define the bilinear form \bea R_\mu(v,w):=\int_{B_1}\nabla v\cdot \nabla w \y\,dX-\mu\int_{\partial B_1} vw\y\,d\HH^n.
 \eea 
In the sequel we also need the following lemma which can be easily proved as \cite[Lemma 3.5]{cv24}.
\begin{lemma} \label{lemma:double-product} Let $\psi,\phi\in W^{1,2}(\partial B_1,\y)$, even with respect to $\{y=0\}$, with $$\psi(\theta)=\sum_{j=1}^\infty c_j\phi_j(\theta),$$
  where $\{\phi_j\}\subset W^{1,2}_0(S_0,\y)$ is an ortonormal basis of $W^{1,2}_0(S_{0},\y)$ of eigenfunctions for $-L_a^{\mathbb S^{n}}$.
  Then $$R_\mu(r^\mu\psi(\theta),r^\alpha\phi(\theta))=\frac1{n+a+\alpha+\mu-1}\,\beta_\mu(\psi,\phi),$$ where 
  \begin{equation*}
      \begin{aligned}
            \beta_\mu(\psi,\phi):=\int_{\partial B_1}\sum_{j=1}^\infty(\lambda_j-\lambda(\mu))c_j\phi_j(\theta)\phi(\theta)\y\,d\HH^n
            -2\int_{\partial{B_1'}}\Big(\lim_{\theta_{n+1}\downarrow 0} \theta_{n+1}^{a}\partial_{\theta_{n+1}}\psi\Big)\phi\,d\HH^{n-1}.
      \end{aligned}
  \end{equation*}
  
\end{lemma} 
\noindent Now we are ready to prove \cref{thm:epi}
\begin{proof}[Proof of \cref{thm:epi}]
    Let us take $c\in W^{1,2}(\partial B_1,\y)$ even in $y$ and such that $c\ge 0$ on $\partial B_{1}'$. We take $\delta>0$ as in \cref{lemma:convergence-of-s_delta}. Let us assume that condition \eqref{eq:quasizero} holds for some $p\in \mathcal{P}_{2m+2s}$ with $\|p\|_{L^2(\partial B_1,|y|^{a})}=1$. Then, by definition, we will have $c\in W^{1,2}_{0}(S_{\delta},|y|^{a})$. Let $z(r,\theta)=r^{2m+2s}c(\theta)$ be the $(2m+2s)$-homogeneous extension of $c$.
    
    We prove point i).
    Thanks to \cref{lemma:convergence-of-s_delta}, we may chose a suitable orthonormal basis $\{\phi_{j}^{\delta}\}$ of eigenfunctions for $-L_{a}^{\mathbb{S}^{n}}$ in $W^{1,2}_{0}(S_{\delta},|y|^{a})$ that makes possible the following decomposition:
\bea c(\theta)=\psi(\theta)+\phi(\theta),\quad \mbox{where}\quad \psi(\theta)=\sum_{j=1}^\ell c_j\phi_j\quad\mbox{and}\quad \phi(\theta)=\sum_{j=\ell+1}^\infty c_j\phi_j^\delta,\eea and for which \be\label{eq:lambda_j_grandi}\lambda_j^\delta\ge \lambda(2m+2s+1/2)=:\alpha\quad\mbox{for every } j>\ell.
\ee  
We chose the competitor $\zeta$ as \bea \zeta(r,\theta):=r^{2m+2s}\psi(\theta)+r^{\alpha}\phi(\theta).\eea 
Notice that  $\zeta=c$ on $\partial B_{1}$, and since $\psi \in W^{1,2}_{0}(S_{0},|y|^{a})$ and $c \ge 0$ on $\partial B_{1}'$ by hypothesis, then $\zeta\ge0$ on $B_1'$. Renaming $\mu:=2m+2s$, we now show that \eqref{eq:epi} holds for $\kappa=\kappa_{\alpha,\mu}$ as in \eqref{eq:kappa_alpha_mu}. Indeed we can write
\begin{align*}
    W_{\mu}(\zeta)&=W_{\mu}(r^{\mu}\psi)+W_{\mu}(r^{\alpha}\phi)+2R_{\mu}(r^{\mu}\psi,r^{\alpha}\phi),\\ 
    W_{\mu}(z)&=W_{\mu}(r^{\mu}\psi)+W_{\mu}(r^{\mu}\phi)+2R_{\mu}(r^{\mu}\psi,r^{\mu}\phi).
\end{align*}
Now, since $\psi$ is composed only by modes lower than $\lambda(\mu)$, by \eqref{eq:lower_modes},
\begin{equation*}
    W_{\mu}(r^{\mu}\psi)-(1-\kappa_{\alpha,\mu})W_{\mu}(r^{\mu}\psi) = \kappa_{\alpha, \mu}W_{\mu}(r^{\mu}\psi) \le 0.
\end{equation*}
On the other hand, since $\psi$ containes only modes higher than $\alpha$, by \eqref{eq:higher_modes} and \eqref{eq:lambda_j_grandi},
\begin{equation*}
    W_{\mu}(r^{\alpha}\phi)-(1-\kappa_{\alpha,\mu})W_{\mu}(r^{\mu}\phi)\le 0.
\end{equation*}
Finally, owing to \cref{lemma:double-product} and the definition of $\kappa_{\alpha, \mu}$ in \eqref{eq:kappa_alpha_mu}, we get
$$
    R_{\mu}(r^{\mu}\psi,r^{\alpha}\phi)-(1-\kappa_{\alpha, \mu})R_{\mu}(r^{\mu}\psi,r^{\mu}\phi)
    =  \left(\frac{1}{n+a+\alpha+\mu-1}-\frac{(1-\kappa_{\alpha, \mu})}{n+a+2\mu-1}\right)\beta_{\mu}(\psi, \phi)=0.
$$
The last three inequalities above together give \eqref{eq:epi} with $\kappa=\kappa_{\alpha, \mu}$, as desired.

We now address point ii). We may assume that $W_{\mu}(z)<0$ since otherwise we can simply take $\hat{\zeta}=z$ as a competitor. Thanks to the fact that $c\in W^{1,2}_{0}(S_{\delta},|y|^{a})$, choosing $\delta>0$ small enough, as in \cref{lemma:convergence-of-s_delta} we may decompose $c$ as
\bea c(\theta)=\chi(\theta)+\omega(\theta),\quad \mbox{where}\quad \chi(\theta)=b_{\ell}\phi_{\ell}\quad\mbox{and}\quad \omega(\theta)=\sum_{j\neq \ell}b_j\phi_j^\delta.\vspace{-0.3cm}\eea 
Moreover, since $\lVert c-\phi_{\ell}\rVert_{L^{2}(\partial B_{1},|y|^{a})}\le \eps$, then
$$\lVert \omega\rVert_{L^{2}(\partial B_{1},|y|^{a})}\le C\eps.$$
We choose $\alpha<\mu$ for which $\kappa_{\mu,\alpha}=-W_{\mu}(z)>0$ and we define the competitor $\hat{\zeta}$ as 
$$\hat{\zeta}(r,\theta):=r^{\mu}\chi(\theta)+r^{\alpha}\omega(\theta).$$
Notice that, since $\chi$ is a $\mu$-homogeneous solution of \eqref{def:soluzione-estesa} with zero obstacle, and $c\ge 0$ on $\partial B_{1}'$, we will have $W_{\mu}(r^{\mu}\chi)=0$. Then
$$R_\mu(r^\mu\chi,r^\mu\omega)=R_\mu(r^\mu\chi,r^\mu c)=-\int_{B_1}L_a(r^\mu\chi)r^\mu c\ge0.$$ Therefore
    \be\label{w1}W_{\mu}(z)=W_{\mu}(r^\mu\omega)+2R_\mu(r^\mu\chi,r^\mu\omega) \ge W_\mu(r^\mu\omega).\ee
Moreover, exactly as in the proof of point i),
\bea R_{\mu}(r^{\mu}\chi,r^{\alpha}\omega)-(1+\kappa_{\mu, \alpha})R_{\mu}(r^{\mu}\chi,r^{\mu}\omega)=0.
\eea
Therefore, using \cref{lemma:formulas-homog-epi} and \eqref{w1}, we get
\bea W_{\mu}(\hat{\zeta})-(1+\kappa_{\mu,\alpha})&W_{\mu}(z)=W_{\mu}(r^{\alpha}\omega)-(1+\kappa_{\mu,\alpha})W_{\mu}(r^{\mu}\omega)\\
&=\frac{-\kappa_{\mu,\alpha}}{n+a+2\alpha-1}\sum_{j\neq \ell}(\lambda(\alpha)-\lambda_j^\delta)c_j^2\\&=\frac{-\kappa_{\mu,\alpha}}{n+a+2\alpha-1}\sum_{j\neq\ell}(\lambda(\alpha)-\lambda(\mu))c_j^2+\frac{-\kappa_{\mu,\alpha}}{n+a+2\alpha-1}\sum_{j\neq\ell}(\lambda(\mu)-\lambda_j^\delta)c_j^2
 \\&=C\kappa_{\mu,\alpha}^2\|\omega\|^2_{L^2(\partial B_1,\y)}
 +\kappa_{\mu,\alpha}\frac{n+a+2\mu-1}{n+a+2\alpha-1}W_\mu(r^\mu\omega)\\&\le C\kappa_{\mu,\alpha}^2\eps
 -\frac{n+a+2\mu-1}{n+a+2\alpha-1}\kappa_{\mu,\alpha}^{2},
 \eea
and the right-hand side is negative provided that we choose $\eps$ small enough.
\end{proof}
\subsection{Proofs of \cref{prop:rate} and \cref{cor:stratification}}\label{subsection:proofs_2m+2s}
To prove \cref{prop:rate} we will apply the epiperimetric inequality in point i) of \cref{thm:epi} to the rescalings of a solution $u$ of \eqref{def:soluzione-estesa}. We will use the following notations:
\begin{equation}\label{notation:application-epi}
    \widetilde{u}:= \widetilde{u}^{0},\qquad \widetilde{u}_{\rho}:= \frac{\widetilde{u}(\rho \cdot)}{\lVert \widetilde{u}(\rho \cdot)\rVert_{L^{2}(\partial B_{1},|y|^{a})}},\qquad \widetilde{u}_{\rho,r}:= \frac{\widetilde{u}_{\rho}(r\cdot)}{r^{2m+2s}}.
\end{equation}
Notice that the two subscripts refer to different types of rescalings. We can write $\widetilde{u}_{\rho,r}$ in the following equivalent ways:
\begin{equation}\label{eq:riscritture-rescalings}
    \widetilde{u}_{\rho,r}= \left(\frac{\rho^{n+a+2\mu}}{H(\rho,\widetilde{u})}\right)^{1/2}\frac{\widetilde{u}(\rho r \cdot)}{(\rho r)^{\mu}}=\left(\frac{\rho^{n+a+2\mu}}{H(\rho,\widetilde{u})}\right)^{1/2} \left(\frac{H(\rho r, \widetilde{u})}{(\rho r)^{n+a+2\mu}}\right)^{1/2}\widetilde{u}_{\rho r}.
\end{equation}
In particular, thanks to \cref{generalized-Almgren}, for a fixed $\rho$, $\widetilde{u}_{\rho,r}$ has uniform $C^{1,s}_{a,\loc}$-bounds for every $r\in (0,1)$ and converge, as $r\downarrow 0$ to a blow-up in $\mathcal{P}_{2m+2s}$. 

We recall the following consequences of \cref{prop:monotonicity-weiss-tilde} (see e.g.~\cite{car24}) for the Weiss energy $\widetilde W_{\mu}$ defined in \eqref{def:weiss-tilde}.

\begin{lemma}\label{prop:monotonicity-weiss-smart}
    Let $u$ be a solution to \eqref{def:soluzione-estesa}, with $\vf$ satisfying \eqref{e:hypo-phi} and $\|u\|_{L^\infty(B_1)}\le 1$. Suppose that $0\in \Lambda_{\mu}(u)$, with $\mu< k+\gamma$. Using the notations in \eqref{notation:application-epi}, we call $z_{\rho,r}$ the $\mu$-homogeneous extension of $\widetilde{u}_{\rho,r}|_{\partial B_{1}}$. Calling $C(\rho):= C_{\widetilde{W}}(\widetilde{u}_{\rho})$ as in \cref{prop:monotonicity-weiss-tilde}, we have
    \begin{equation}\label{eq:monot-furba-weiss}
        \begin{aligned}
            \frac{d}{dr}\left(\widetilde{W}_{\mu}(\widetilde{u}_{\rho,r})+C(\rho)r^{k+\gamma-\mu}\right)&\ge \frac{n+a+2\mu-1}{r}\left(W_{\mu}(z_{\rho,r})-\widetilde{W}_{\mu}(\widetilde{u}_{\rho,r})\right)\\
        &+\frac{1}{r}\int_{\partial B_{1}}(\nabla \widetilde{u}_{\rho,r}\cdot \nu- \mu \widetilde{u}_{\rho,r})^{2}|y|^{a}\,d\HH^n,
        \end{aligned}
    \end{equation} and 
    \begin{equation}\label{eq:weiss-argomento-diadico}
        \int_{\partial B_1}\left|\widetilde{u}_{\rho,r}-\widetilde{u}_{\rho,r'}\right|\y\,d\HH^n\le C \log\left(\frac r{r'}\right)^{1/2}\left(\widetilde{W}_\mu(\widetilde{u}_{\rho,r})+C(\rho)r^{k+\gamma-\mu}\right)^{1/2},
    \end{equation}
    for every $ 0<r'\le r\le 1$ and for some constant $C>0$. Moreover $C(\rho)\to 0$ as $\rho \to 0$.
\end{lemma}
\cref{prop:rate} is a consequence of \cref{prop:monotonicity-weiss-smart} and the following lemma, which allows the use of the epiperimetric inequality at all sufficiently small scales. 
   \begin{lemma}\label{lemma:every-rescaled-fundamental} Let $u$ be a solution of \eqref{def:soluzione-estesa}, with obstacle $\vf$ satisfying \eqref{e:hypo-phi} and $\|u\|_{L^\infty(B_1)}\le 1$. Suppose that $0\in\Lambda_{2m+2s}(u)$ and $2m+2s< k+\gamma$. Then, using the notations in \eqref{notation:application-epi}, there exist some $\rho \in (0,1)$ and $p\in \mathcal{P}_{2m+2s}$ with $\lVert p\rVert_{L^{2}(\partial B_{1},|y|^{a})}=1$ for which the epiperimetric inequality in point i) of \cref{thm:epi} can be applied to the trace $\widetilde u_{\rho,r}|_{\partial B_1}$, for every $r\in(0,1)$.
\end{lemma}
\begin{proof}
    We can take $\rho$ small enough so that, for some $p\in \mathcal{P}_{2m+2s}$ with $\lVert p\rVert_{L^{2}(\partial B_{1},|y|^{a})}=1$ we have
    $$\widetilde W_{\mu}(\widetilde{u}_{\rho})+C(\rho)\le \eps_{1},\quad \lVert \widetilde{u}_{\rho}-p\rVert_{L^{2}(\partial B_{1},|y|^{a})}\le \eps_{1},\quad \lVert \widetilde{u}_{\rho}-p\rVert_{L^{\infty}(B_{2})}\le \eps_{1},$$
    and
    $$|L_{a}\widetilde{u}_{\rho}|\le \eps_{1}\quad \mbox{in } B_{2}\setminus \{\widetilde{u}_{\rho}=0\}',$$
    where $\eps_{1}>0$ has to be chosen sufficiently small. Given $\delta>0$ as in \cref{thm:epi}, we define
    $$r_{0}:= \inf \{r\in (0,1): \widetilde{u}_{\rho,t}=0 \mbox{ on } \mathcal Z_{\delta}(p), \mbox{ for every } t\in (r,1)\}.$$

    \noindent \textbf{Step 1}.
    As a first step we prove that $r_{0}\le 1/3$, using a barrier argument as in \cite[Lemma B.3]{frs20} or \cite[Lemma 4.4]{cv24}. To do so, it is enough to show that if $Z=(z,0)\in B_{1}'\setminus B_{1/3}'$ is such that $T[p](z)\ge \delta/3^{2m}$, then $\widetilde{u}_{\rho}(Z)=0$. For every $c\ge 0$ we consider the barrier $$\phi_{c}(X):= |X-Z|^{2}-\left(\frac{n}{1+a}+2\right)y^{2}+c.$$
    Recalling the decomposition \eqref{decomposition_polynomials_2m+2s}, if $\eps_{1}=r_{1}^{2+2s}$, and $\eps_{1}$ is small enough, on $\partial B_{r_{1}}(Z)$, we can bound
    \begin{align*}
        \widetilde{u}_{\rho}(X)&\le p(X)+\eps_{1}=-|y|^{2s}\left(p_{0}(x)+y^{2}p_{1}(x,y)\right)+\eps_{1}\le -|y|^{2s}p_{0}(z)+Cr_{1}^{2+2s}+\eps_{1}\\
        &\le -C\delta |y|^{2s}+Cr_{1}^{2+2s}\le \phi_{c}(X),
    \end{align*}
    for every $c\ge 0$. Suppose by contradiction that there is $c_{*}>0$ for which $\phi_{c_{*}}$ touches $\widetilde{u}$ from above at some point $X_{*} \in \overline B_{r_{1}}(Z)$. Then, by the previous computation, $X_{*}\in B_{r_{1}}(Z)$. However, $X_{*}\not \in \{\widetilde{u}_{\rho}=0\}'$ since $\phi_{c_{*}}\ge c_{*}>0$ there. Hence $X_{*}\not\in  B_{r_{1}}(Z)\setminus \{\widetilde{u}_{\rho}=0\}'$, but then, $-L_{a}\phi_{c}(X_{*})=2>\eps_{1}\ge -L_{a}\widetilde{u}_{\rho}(X_{*})$ which is not possible as $\phi_{c_{*}}$ is touching $\widetilde{u}_{\rho}$ from above. Therefore, $\phi_{0}\ge \widetilde{u}_{\rho}$ on $B_{r_{1}}(Z)$, and in particular $\widetilde{u}_{\rho}(Z)=0$.   
   
   \smallskip
    \noindent \textbf{Step 2}.
    We assume that $r_{0}\in (0,1/3)$ and try to find a contradiction for $\eps_{1}$ sufficiently small. 
    For every $r\in (r_{0},1)$, let $\zeta_{\rho,r}$ be the competitor given by point i) of \cref{thm:epi} for the trace $\widetilde{u}_{\rho, r}|_{\partial B_{1}}$. If $\mu:=2m+2s$, the epiperimetric inequality gives 
    \begin{align*}
        \widetilde W_{\mu}(\widetilde{u}_{\rho,r})\le W_{\mu}(\zeta_{\rho,r})+C(\rho)r^{k+\gamma-\mu}\le  (1-\kappa)W_{\mu}(z_{\rho,r}) +C(\rho)r^{k+\gamma-\mu}.
    \end{align*}
    Then, \eqref{eq:monot-furba-weiss} yields
    \begin{align*}
        \frac{d}{dr}\left(\widetilde W_{\mu}(\widetilde{u}_{\rho,r})+C(\rho)r^{k+\gamma-\mu}\right)&\ge \frac{n+a+2\mu-1}{r}\left(W_{\mu}(z_{\rho,r})-\widetilde W_{\mu}(\widetilde{u}_{\rho,r})\right)\\
        &\ge \frac{\kappa}{1-\kappa}\frac{(n+a+2\mu-1)}{r}\widetilde W_{\mu}(\widetilde{u}_{\rho,r})-CC(\rho)r^{k+\gamma-\mu-1}.
    \end{align*}
    Integrating such inequality we get 
    $$\widetilde{W}_{\mu}(\widetilde{u}_{\rho,r})+C(\rho)r^{k+\gamma-\mu}\le C\eps_1r^{\alpha}\quad \mbox{for every } r\in (r_{0},1),$$
    for some $C,\alpha>0$.
    As a consequence, we may apply \eqref{eq:weiss-argomento-diadico} with a dyadic argument to obtain 
    $$\int_{\partial B_{1}}|\widetilde{u}_{\rho,r}-\widetilde{u}_{\rho}||y|^{a}\,d\HH^n\le C\eps_{1}^{1/2}\quad \mbox{for every } r\in (r_{0}/8, 1).$$
    Then, by triangular inequality, for every $r\in (r_{0}/8, 1)$, we will have
    \bea \int_{\partial B_{1}}|\widetilde{u}_{\rho,r}-p||y|^{a}\,d\HH^n&\le \int_{\partial B_{1}}|\widetilde{u}_{\rho,r}-\widetilde{u}_{\rho}||y|^{a}\,d\HH^n+\int_{\partial B_{1}}|\widetilde{u}_{\rho}-p||y|^{a}\,d\HH^n\\&\le C\eps_{1}^{1/2}+C\eps_{1}=:\eps_{2}.\eea
    Using the fact that $r_{0}\le 1/3$ we may compute
    \begin{align*}
        \int_{B_{2}\setminus B_{1/8}}|\widetilde{u}_{\rho,r_{0}}-p||y|^{a}\,d\HH^n&= \int_{1/8}^{2}\int_{\partial B_{t}}|\widetilde{u}_{\rho,r_{0}}-p||y|^{a}\,d\HH^n\,dt \\&= \int_{1/8}^{2}t^{n+a+\mu}\int_{\partial B_{1}}|\widetilde{u}_{\rho,r_{0}t}-p||y|^{a}
        \,d\HH^n\,dt\\&\le C\eps_{2}.
    \end{align*}
    Hence, by uniform $C^{1,s}_{a}$-bounds on the family $\{\widetilde{u}_{\rho,r}\}_{r\in (0,1)}$ and arguing as in \cite[Lemma 4.3]{cv24}, we deduce the $L^{\infty}$-estimate
    $$\lVert \widetilde{u}_{\rho,r_{0}}-p\rVert_{L^{\infty}(B_{3/2}\setminus B_{1/4})}\le \eps_{3}.$$
    Finally, up to chosing $\eps_{1}$ (and consequently $\eps_{2}$ and $\eps_{3}$) small enough, this $L^{\infty}$-bound is sufficient to run exactly the same argument of step 1 on $\widetilde{u}_{\rho,r_{0}}$, deducing that the epiperimetric inequality may be also applied to $\widetilde{u}_{\rho,r}|_{\partial B_{1}}$ for every $r\in (r_{0}/3,r_{0}]$, thus contradicting the definition of $r_{0}$.  
\end{proof}
We are now in the position to prove the main result of this section.
\begin{proof}[Proof of \cref{prop:rate}]
    Take $\rho\in (0,1)$ and $p\in \mathcal{P}_{2m+2s}$ a blow-up as in \cref{lemma:every-rescaled-fundamental}. We rename $\mu:= 2m+2s$. Arguing exactly as in step 2 of \cref{lemma:every-rescaled-fundamental}, applying the epiperimetric inequality at each scale, in combination with \eqref{eq:monot-furba-weiss}, we get
    $$\widetilde{W}_{\mu}(\widetilde{u}_{\rho,r})+C(\rho)r^{k+\gamma-\mu}\le Cr^{\alpha}\quad \mbox{for every }r\in (0,1),$$
    for some $C, \alpha>0$.
    Then, using \eqref{eq:weiss-argomento-diadico} with a dyadic argument, we obtain
    \begin{align*}
        \int_{\partial B_{1}}|\widetilde{u}_{\rho,r}-\widetilde{u}_{\rho,r'}||y|^{a}\,d\HH^n\le Cr^{\alpha/2}\quad\mbox{for every } 0<r'< r \le 1.
    \end{align*}
    Hence, up to multiplying $p\in \mathcal{P}_{2m+2s}$ and $\alpha$ by some constant, we get
    $$
    \int_{\partial B_{1}}|\widetilde{u}_{\rho,r}-p||y|^{a}\,d\HH^n\le Cr^{\alpha}\quad \mbox{for every } r\in (0,1),
    $$
    and arguing as in \cite[Lemma 7.2]{csv20}, one can prove that $p\not\equiv0$.
    Now, since $(\widetilde{u}_{\rho,r}-p)^{\pm}$ are subsolutions for $L_{a}$ (with a small right-hand side of order $Cr^{k+\gamma-\mu}$), we get the $L^{\infty}$-estimate
    $$\lVert \widetilde{u}_{\rho,r}-p\rVert_{L^{\infty}(B_{1/2})}\le Cr^{\alpha}\quad \mbox{for every } r\in (0,1).$$
    Finally, we see from \eqref{eq:riscritture-rescalings} that the latter can be re-written, up to multiplying $p$ with a positive constant,  as
    $$\lVert \widetilde{u}-p\rVert_{L^{\infty}(B_{r})}\le Cr^{2m+2s+\alpha}\quad \mbox{for every } r\in (0,\rho),$$
    and the theorem is proved.
\end{proof}
Once the rate of convergence to the blow-up limit is proved, the stratification of the contact set follows by standard arguments as in \cite{gp09,csv20, car24}.
\begin{proof}[Proof of \cref{cor:stratification}]
    For every $X_{0}\in \Lambda_{2m+2s}(u)$ let $p^{X_{0}}\in \mathcal{P}_{2m+2s}$ be the first blow-up at $X_{0}$. By \cref{prop:rate} we know that $p^{X_{0}}\not \equiv 0$ and
    $$\|\widetilde u(X_0+\cdot)-p^{X_0}\|_{L^\infty(B_r)}= O(r^{2m+2s+\alpha})\quad \mbox{as } r\downarrow 0.$$
    Calling $p_0^{X_0}:=T[p^{X_0}]$ the non-negative $2m$-homogeneous polynomial from \eqref{e:definition-of-T}, we define 
    $$d^{X_0}:=\text{dim}\{\xi\in\R^n:\nabla p_0^{X_0}\cdot \xi\equiv0 \mbox{ in } \R^n\}\in \{0,\dots,n-1\}.$$
    Then we set
    $$\Lambda_{2m+2s}^j(u):=\{X_0\in\Lambda_{2m+2s}(u):\ d^{X_0}=j\}\qquad \mbox{for } j\in \{0,\dots,n-1\}.$$
    Arguing as in \cite{csv20, car24}, using the rate of convergence to the blow-up, one can show that the function $X_0\mapsto p^{X_0}_{0}$ is locally $\alpha$-H\"{o}lder continuous. Then the standard Whitney's extension argument applies (see for instance \cite{gp09}) and gives that $\Lambda_{2m+2s}^{j}$ is locally contained in the zero set of a $C^{1,\alpha}$ function whose differential has $j$-dimensional kernel at each point of $\Lambda_{2m+2s}^{j}$. The proof is thus concluded by the implicit function theorem. 
\end{proof}

\section{Frequency gaps}\label{section-frequency_gaps}
This section is devoted to the proof of  the frequency gaps in \cref{teo:gaps}.
In \Cref{subsection:gap-odd} we prove the frequency gap around  $2m+2s$ using the epiperimetric inequalities of \cref{thm:epi}. In \Cref{subsection:explicit-gap} we exploit the monotonicity of eigenvalues for the spherical $L_{a}^{\mathbb{S}^{n}}$ operator with respect to domain inclusion to prove explicit gaps between $\N $ and $\N+2s$, and to show that the frequencies in $2\N+1$ and $2\N+1+2s$ are not admissible when $s\neq1/2$. Finally, in \Cref{subsection:prop:gap-fs24}, we extend the result \cite{fs24-gap} proving that the frequencies in $(2m,2m+2s)$ are not admissible.
\subsection{Frequency gap at $\boldsymbol{2m+2s}$}\label{subsection:gap-odd}
In the following proposition we show that $(2m+2s)$-frequencies are isolated.
\begin{proposition}[Frequency gap at $2m+2s$]\label{prop:gap-odd} 
   Let $\mathcal{A}_{n,s}$ be the set of admissible frequencies in dimension $n+1$ from \eqref{def:admissible-frequencies} and let 
$\mathcal{A}_{n,s}^{*}:= \mathcal{A}_{n,s}\setminus \mathcal{A}_{1,s}$ be the set of "other frequencies".
   Then, for every $m\in \N$, there exist constants $c^\pm_{n,s,m}>0$ depending only on $n$, $s$ and $m$, such that     $$\mathcal{A}_{n,s}^*\cap \left((2m+2s-c_{n,s,m}^-,2m+2s+c_{n,s,m}^+)\right)=\emptyset.$$ 
\end{proposition}

Before giving the proof of \cref{prop:gap-odd}, we recall the following lemma which comes from a straightforward computation.
\begin{lemma}[\protect{\cite[Lemma 2.11]{car24}}]
\label{lemma:gap}
    	Let $c\in W^{1,2}(\partial B_1,\y)$ be such that $r^{\mu+t}c$ is a solution of \eqref{def:soluzione-estesa} with zero obstacle, then 
		\begin{equation*}\label{terza}
			W_\mu(r^{\mu+t}c)=t\| c\|_{L^2(\partial B_1,\y)}^2
		\quad\mbox{and}\quad
			W_\mu(r^{\mu}c)=\left(1+\frac{t}{n+a+2\mu-1}\right)W_\mu(r^{\mu+t}c).
		\end{equation*}
	\end{lemma}
\begin{proof}[Proof of \cref{prop:gap-odd}] 
The proof is based on a compactness argument with the epiperimetric inequalities in \cref{thm:epi}, combined with \cref{lemma:gap}. Call $\mu:= 2m+2s$ and suppose by contradiction that there is a sequence $t_{k}\to 0$ and homogeneous solutions $u_{k}=r^{\mu+t_{k}}c_{k}(\theta)$ of \eqref{def:soluzione-estesa} with zero obstacle, and $\lVert c_{k}\rVert_{L^{2}(\partial B_{1},|y|^{a})}=1$. Then, up to subsequences, $u_{k}$ converges in $C^{1,\alpha}_{a}(B_{1/2}^{+})$ to some $p\in \mathcal{P}_{2m+2s}$. In particular, for $k$ large enough, $u_{k}$ satisfies all the hypotheses of \cref{thm:epi}. Up to extracting a further subsequence, we may assume that either $t_{k}\downarrow 0$ or $t_{k}\uparrow 0$.

\textit{Case a): }($t_{k}\downarrow 0$). Let $\zeta_{k}$ be the competitor in \eqref{eq:epi}. Since $u_{k}$ solves \eqref{def:soluzione-estesa}, it minimizes $W_{\mu}$ among all functions with its same trace on the unit sphere. Using also \cref{lemma:gap} we get
\begin{align*}
    t_{k}= W_{\mu}(r^{\mu+t_{k}}c_{k})\le W_{\mu}(\zeta_{k})\le (1-\kappa)W_{\mu}(r^{\mu}c_{k})=(1-\kappa)(1+Ct_{k})t_{k}.
\end{align*}
Hence, since $t_{k}>0$,
$$1\le (1-\kappa)+O(t_{k}),$$
which is a contradiction for $k$ large.

\textit{Case b): }($t_{k}\uparrow 0$). Let $\hat{\zeta}_{k}$ be the competitor in \eqref{eq:epi-neg}. Arguing as above we get
\begin{align*}
    t_{k}= W_{\mu}(r^{\mu+t_{k}}c_{k})\le W_{\mu}(\hat{\zeta}_{k})\le (1+|W_{\mu}(r^{\mu}c_{k})|)W_{\mu}(r^{\mu}c_{k})=(1-(1+Ct_{k})t_{k})(1+Ct_{k})t_{k},
\end{align*}
where $C= 1/(n+a+2\mu-1)<1$. Hence, since $t_{k}<0$,
$$0\ge -t_{k}(1-C)+O(t_{k}^{2}),$$
which is again a contradiction for $k$ large.
\end{proof}
\subsection{Frequency gaps for $\boldsymbol{s\neq 1/2}$}\label{subsection:explicit-gap}
For any closed cone $\Lambda \subset \{y=0\}$ and any $\alpha \ge 0$ we denote by $V_{\alpha}^{a}(\Lambda)$ the vector space of $\alpha$-homogeneous solutions $u\in W^{1,2}_{\loc}(\R^{n+1}, \y)$ of the problem

\begin{equation*}\label{problema:La-harmonic-with-boundary-conditions}
    \left\{
        \begin{array}{rclll}
             -L_{a}u&=&0&\quad \mbox{in } \R^{n+1}\setminus \Lambda,\\
             u&=&0&\quad \mbox{on } \Lambda,\\
             u(x,y)&=&u(x,-y)&\quad \mbox{for } (x,y)\in \R^{n+1}.
        \end{array}
        \right.
\end{equation*}
\begin{remark}\label{remark:derivative-dualproblem}
    Notice that if $u$ is a global $\alpha$-homogeneous solution of the thin obstacle problem \eqref{def:soluzione-estesa} with zero obstacle, then $u\in V_{\alpha}^{a}(\Lambda(u))$. In addition, by the regularity theory for \eqref{def:soluzione-estesa}, calling $v$ the even extension of $y^{a}\partial_{y}u$ (defined on $\R^{n+1}_{+}$) to the whole $\R^{n+1}$, we have that $v\in V_{\alpha-2s}^{-a}(\{y=0\}\setminus \Lambda(u))$.
\end{remark}
Observe that $u\in V_{\alpha}^{a}(\Lambda)$ if and only if the trace of $u$ on $\mathbb{S}^{n}$ is an eigenfunction of the spherical $-L_{a}^{\mathbb{S}^{n}}$ operator (see \Cref{subsection:epiperimetric}) with Dirichlet boundary conditions on $\mathbb{S}^{n}\cap \Lambda$ and corresponding eigenvalue $\lambda(\alpha)=\alpha(n+a+\alpha-1)$, i.e.
\begin{equation*}
    \left\{
        \begin{array}{rclll}
             -L_{a}^{\mathbb{S}^{n}}u&=&\lambda(\alpha)u|y|^{a}&\quad \mbox{in } \mathbb{S}^{n}\setminus \Lambda,\\
             u&=&0&\quad \mbox{on } \mathbb{S}^{n}\cap\Lambda,\\
             u(x,y)&=&u(x,-y)&\quad \mbox{for } (x,y)\in \mathbb{S}^{n}.
        \end{array}
        \right.
\end{equation*}
Let us denote by $\{\lambda_{j}^{a}(\Lambda)\}_{j=1}^{\infty}$ the eigenvalues of the operator $-L_{a}^{\mathbb{S}^{n}}$ (in increasing order and counted with multiplicity) in the Hilbert space of Sobolev functions in $\mathbb{S}^{n}$ which are even in $y$ and vanish on $\mathbb{S}^{n}\cap \Lambda$.
\begin{remark}\label{remark:monotonicity-eigenvalues}
    It is a standard consequence of min-max formulas for eigenvalues of Dirichlet elliptic operators that all eigenvalues are monotone with respect to domain inclusion. Precisely, if $\Lambda \subseteq \widetilde{\Lambda}\subset \{y=0\}$ are two closed cones, then $\lambda_{j}^{a}(\Lambda)\le \lambda_{j}^{a}(\widetilde{\Lambda})$ for every $j\ge 1$. Moreover, one may prove by unique continuation arguments that if equality holds for some $j$, then $\widetilde{\Lambda}\setminus \Lambda$ has empty interior in the relative topology of $\{y=0\}$.
\end{remark}
In the following lemma we give a characterization and explicitly compute the dimension of the spaces $V_{\alpha}^{a}(\Lambda)$ in the particular cases when $\Lambda =\emptyset$ or $\Lambda= \{y=0\}$.
\begin{lemma}\label{lem:eigenspaces-empty-full}
    The set $V_{\alpha}^{a}(\emptyset)\neq \{0\}$ if and only if $\alpha \in \N$. Similarly, $V_{\alpha}^{a}(\{y=0\})\neq\{0\}$ if and only if $\alpha \in \N +2s$.  Moreover, for every $k\in \N$:
    \begin{gather*}
        V_{k}^{a}(\emptyset)=\{p: \mbox{ $p$ polynomial} , \ \nabla p \cdot X=kp, \ p(x,y)=p(x,-y), \ L_{a}p =0\},\\
        V_{k+2s}^{a}(\{y=0\})=\{|y|^{2s}p: \mbox{ $p$ polynomial}, \ \nabla p \cdot X=kp,  \ p(x,y)=p(x,-y), \ L_{a}(\sgn(y)|y|^{2s}p)=0\}.        
    \end{gather*}
    In particular:
    $$\dim V_{k}^{a}(\emptyset)=\dim V_{k+2s}^{a}(\{y=0\})=\binom{n+k-1}{n-1}.$$
\end{lemma}
\begin{proof}
    The characterizations of the spaces $V_{k}^{a}(\emptyset)$ and $V_{k+2s}^{a}(\{y=0\})$ are direct consequences of Liouville-type theorem  \cref{prop:Liouville} with \cref{prop:odd-extension}. To compute their dimension, calling $m:= \lfloor k/2\rfloor$, we notice that two general elements $v\in V_{k}^{a}(\emptyset), w\in V_{k+2s}^{a}(\{y=0\})$ take the form
    $$v(x,y)= \overunderset{m}{\ell=0}{\sum}y^{2\ell}p_{\ell}(x),\quad w(x,y)= |y|^{2s}\overunderset{m}{\ell=0}{\sum}y^{2\ell}q_{\ell}(x),$$
    where, for every $\ell \in \{0,\dots, m\}$, $p_{\ell}$ and $q_{\ell}$ are $(k-2\ell)$-homogeneous polynomials in $\R^{n}$ and
    $$\begin{cases}
        -\Delta_{x}p_{\ell}=4(\ell+1)(\ell+1-s)p_{\ell+1},\\
        -\Delta_{x}q_{\ell}=4(\ell+1)(\ell+1+s)q_{\ell+1},
    \end{cases}\qquad \text{for every } \ell \in \{0,\dots, m-1\}.$$
    In particular, $p_{1},\dots,p_{m}$ and $q_{1},\dots,q_{m}$ are uniquely determined once the $k$-homogeneous polynomials $p_{0}$ and $q_{0}$ are chosen. This proves that $V_{k}^{a}(\emptyset)$ and $V_{k+2s}^{a}(\{y=0\})$ have both the same dimension of the space of $k$-homogeneous polynomials in $\R^{n}$, as desired.
\end{proof}
 As a consequence of \cref{remark:monotonicity-eigenvalues} and \cref{lem:eigenspaces-empty-full} we deduce the following lemma.
\begin{lemma}\label{lem:homog-fractional-solutions}
   Given a closed cone $\Lambda \subseteq \{y=0\}$ and $\alpha\ge 0$, if $V_{\alpha}^{a}(\Lambda)\neq \{0\}$, then $\alpha \in \bigcup_{k\in \N}[k,k+2s].$ 
\end{lemma}
\begin{proof}
    Define the sequence $\{g(k)\}_{k\in \N}$ recursively as follows:
    $$g(0)=1,\qquad g(k+1)=g(k)+\binom{n+k-1}{n-1}\quad \text{for every } k\in \N.$$
    By \cref{lem:eigenspaces-empty-full} and the monotonicity of eigenvalues with respect to domain inclusion (see \cref{remark:monotonicity-eigenvalues}) we have, for every $k\in \N$, 
    $$\lambda(k)=\lambda_{j}^{a}(\emptyset)\le \lambda_{j}^{a}(\Lambda)\le \lambda_{j}^{a}(\{y=0\})=\lambda(k+2s)\qquad \text{for every $j\in \{g(k),\dots, g(k+1)-1\}$}.$$
    In particular, if $u\in V_{\alpha}^{a}(\Lambda)\neq \emptyset$, it means that $\lambda(\alpha)=\lambda_{j}^{a}(\Lambda)\in \bigcup_{k\in \N}[\lambda(k),\lambda(k+2s)]$ for some $j\ge 1$, that is $\alpha \in \bigcup_{k\in \N}[k,k+2s]$. 
\end{proof}
By \cref{lem:homog-fractional-solutions}, we obtain the following explicit frequency gaps.
\begin{proposition}[Explicit frequency gaps]\label{prop:explicit-frequency-gap}
    Let $\mathcal{A}_{n,s}$ be the set of admissible frequencies in dimension $n+1$ from \eqref{def:admissible-frequencies}. Then, for every $k\in \N$, the following hold:
    \begin{itemize}
        \item [i)] if $s<1/2$, then $\mathcal{A}_{n,s}\cap (k+2s, k+1) = \emptyset$;
        \smallskip
        \item[ii)] if $s>1/2$, then $\mathcal{A}_{n,s}\cap (k+1, k+2s) = \emptyset$.
    \end{itemize}   
\end{proposition}
\begin{proof}
    Let $u:\R^{n+1}\to \R$ be an $\alpha$-homogeneous non-zero solution of the thin obstacle problem \eqref{def:soluzione-estesa} with zero obstacle. Since $u\in V_{\alpha}^{a}(\Lambda(u))$, if $s<1/2$, \cref{lem:homog-fractional-solutions} immediately gives that $\alpha \not \in (k+2s,k+1)$ for any $k\ge 0$. If $s>1/2$, taking into account \cref{remark:derivative-dualproblem} we may apply \cref{lem:homog-fractional-solutions} to $v:= y^{a}\partial_{y}u \in V^{-a}_{\alpha-2s}\left(\{y=0\}\setminus \Lambda(u)\right)$ and deduce that 
    $\alpha-2s\not\in(k+2(1-s),k+1+2(1-s))$ for any $k\ge 0$, or equivalently
    $\alpha \not \in (k+2,k+1+2s)$ for any $k\ge 0$.
\end{proof}
Now we prove that for $s\neq 1/2$, the sign conditions on solutions of \eqref{def:soluzione-estesa} with zero obstacle
 are not compatible with homogeneities $2m+1$ and $2m+1+2s$.
 \begin{proposition}\label{prop:snot=1/2}
    Let $\mathcal{A}_{n,s}$ be the set of admissible frequencies in dimension $n+1$ from \eqref{def:admissible-frequencies}. Suppose that $s\neq 1/2$, then $$\mathcal{A}_{n,s} \cap \bigcup_{m\in \N}\{2m+1,2m+1+2s\}=\emptyset.$$
\end{proposition}
    \begin{proof}
        Let $u$ be a non-zero $\alpha$-homogeneous solution of \eqref{def:soluzione-estesa} with zero obstacle, and let $v$ be the even extension of $y^{a}\partial_{y}u$ to the whole $\R^{n+1}$. By the regularity theory we know that $u\in V_{\alpha}^{a}(\Lambda(u))$, $v\in V_{\alpha-2s}^{-a}(\{y=0\}\setminus \Lambda(u))$ and both $u$ and $v$ are continuous. 
    
Suppose by contradiction that $\alpha =2m+1$ for some $m\in \N$. If $s<1/2$, by \cref{lem:homog-fractional-solutions} and the strict monotonicity of eigenvalues  (see \cref{remark:monotonicity-eigenvalues}) we deduce that $\Lambda(u)$ has empty interior in the relative topology of $\{y=0\}$. The same holds true for $s>1/2$ as one may show using the very same argument with $v$ in the place of $u$. Then, since $v$ is continuous and vanishes in $\{y=0\}\setminus \Lambda(u)$, it vanishes in the whole $\{y=0\}$, which means that $u\in V_{2m+1}^{a}(\emptyset)$. 
       By \cref{lem:eigenspaces-empty-full}, $u$ can be written as
        $$u(x,y)=\overunderset{m}{\ell=0}{\sum}y^{2\ell}q_{\ell}(x),$$
        where $q_{\ell}$ are $(2m+1-2\ell)$-homogeneous polynomials. In particular $0\le u(x,0)=q_{0}(x)$. Now, since $q_{0}$ has odd homogeneity, we deduce that $q_{0}\equiv 0$, which implies $u\equiv 0$, a contradiction.
        
        Similarly, if $\alpha=2m+1+2s$ for some $m\in \N$, then $u\in V_{2m+1+2s}^{a}(\{y=0\})$, and so $u$ takes the form
        $$u(x,y)=|y|^{2s}\overunderset{m}{\ell=0}{\sum}y^{2\ell}q_{\ell}(x),$$
        for some $(2m+1-2\ell)$-homogeneous polynomials $q_{\ell}$. As a consequence, 
        $0\ge \lim_{y\downarrow 0}y^{a}\partial_{y}u(x,y)= 2s q_{0}(x).$
        Having $q_{0}$ odd homogeneity, this implies that $q_{0}\equiv 0$, which in turn forces $u\equiv 0$, a contradiction again.
       \end{proof}
    
    \begin{remark}\label{remark:closeness-admissible-freq}
        The set of admissible frequencies $\mathcal{A}_{n,s}$ is closed in $\R$. This can be proved with a simple compactness argument based on regularity estimates for solutions of the thin obstacle problem \eqref{def:soluzione-estesa}. Given a sequence $\alpha_{j}\in \mathcal{A}_{n,s}$ such that $\alpha_{j}\to \alpha$, consider $\alpha_{j}$-homogeneous solutions $u_{j}$ of \eqref{def:soluzione-estesa} with zero obstacle. Assuming that $\lVert u_{j}\rVert_{L^{\infty}(B_{1})}=1$ for every $j$, by the regularity estimates for the thin obstacle problem, the sequence $\{u_{j}\}_{j\in \N}$ is uniformly bounded in $C^{1,s}_{a}(B_{1/2}^+)$, and thus uniformly converges, up to subsequences to a non-zero $\alpha$-homogeneous solution of \eqref{def:soluzione-estesa} with zero obstacle, thus showing $\alpha\in \mathcal{A}_{n,s}$. 
    \end{remark}
\subsection{Frequency gap in $\boldsymbol{(2m,2m+2s)}$}\label{subsection:prop:gap-fs24}
Now we prove that are no admissible frequencies in the interval $(2m,2m+2s)$, for $m\in\N$, following closely the argument presented in \cite{fs24-gap}. 
\begin{proposition}\label{prop:gap-fs24}
    Let $\mathcal{A}_{n,s}$ be the set of admissible frequencies in dimension $n+1$ from \eqref{def:admissible-frequencies}. Then, for every $m\in\N$, we have $$\mathcal {A}_{n,s}\cap (2m,2m+2s)=\emptyset.$$
\end{proposition}
\begin{proof}
    We use polar coordinates $(r,\theta)\in (0,\infty)\times [0,\pi/2]$ such that $r=|X|$ and $ y=r\cos(\theta)$. For every $\alpha \ge 0$,
we consider the $\alpha$-homogeneous function $\widetilde{p}_{\alpha}(X)=r^{\alpha}p_\alpha(\theta)$, radial with respect to the $y$ axis, which solves $$L_a\widetilde p_\alpha=0\quad\text{in }\R^{n+1}_+\qquad\text{and}\qquad \widetilde p_\alpha(e_{n+1})=1.$$
Writing the operator $L_a$ in polar coordinates, we see that $p_\alpha$ is uniquely determined by solving the following ODE: 
\begin{equation}\label{eq:ODE-gap}
     p''_\alpha+\Big((n-1)\cot(\theta)-a\tan(\theta)\Big)p_\alpha'+\lambda(\alpha)p_\alpha =0 \quad \mbox{in } (0,\pi/2),\qquad  p_\alpha(0)=1, \quad p'_\alpha(0)=0.
\end{equation}
where $\lambda(\alpha)=\alpha(n+a+\alpha-1)$. 
Denoting by $\partial^{a}_{\theta} p_{\alpha}:=\cos(\theta)^{a}p'_{\alpha}$, we observe that \be\label{eq:gap0}\lim_{y\downarrow0}y^a\partial_y\widetilde p_\alpha(x,y)=-|x|^{\alpha-2s}\partial^{a}_{\theta}p_{\alpha}(\pi/2),\qquad \partial^{a}_{\theta}p_{\alpha}(\pi/2):=\lim_{\theta\uparrow \pi/2}\partial_\theta^ap_\alpha(\theta).\ee
We claim that 
\be\label{eq:claim-gap}p_\alpha(\pi/2) \partial_\theta^{a}p_\alpha(\pi/2)<0\quad\text{for every }\alpha\in(2m,2m+2s).\ee
Given the claim, we conclude as follows. Suppose $u$ is an $\alpha$-homogeneous solution of \eqref{def:soluzione-estesa} with zero obstacle and $\alpha\in(2m,2m+2s)$. Since $u$ and $\widetilde p_\alpha$ have the same homogeneity, integration by parts yields
$$0=\int_{B_1^+}uL_a\widetilde p_\alpha\,dX-\int_{B_1^+}\widetilde p_\alpha L_au\,dX=\int_{B_1'}\widetilde p_\alpha\lim_{y\downarrow 0}y^a\partial_yu\,d\HH^n-\int_{B_1'}u\lim_{y\downarrow 0}y^a\partial_{y}\widetilde p_\alpha\,d\HH^n.$$
By \eqref{eq:gap0} and \eqref{eq:claim-gap}, together with the sign conditions $u\ge0$ and $\lim_{y\downarrow 0}y^a\partial_yu\le0$ on $B_1'$, the latter forces $u\equiv\lim_{y\downarrow 0}y^a\partial_yu\equiv0$ on $B_1'$, which implies that $u\equiv0$.

Claim \eqref{eq:claim-gap} is proved arguing exactly as in \cite[Lemma 2]{fs24-gap}. First, by the explicit form of the ODE \eqref{eq:ODE-gap}, zeros of $p_\alpha$ and $\partial_{\theta}^{a}p_\alpha$ in $[0,\pi/2]$ must be isolated and intertwined. Secondly, by the strong maximum principle, between any two consecutive zeros of $p_\alpha$ there must be at least a zero of $p_{\alpha'}$, with $\alpha>\alpha'$, then the number of zeros of $p_\alpha$ is non-decreasing in $\alpha$.
Furthermore, $p_\alpha(\pi/2)=0$ if and only if $\alpha\in2\N+2s$, and $\partial_{\theta}^{a}p_\alpha(\pi/2)=0$ if and only if $\alpha\in2\N$, by the characterization of eigenfunctions on the half sphere (see \Cref{subsection:explicit-gap}).
By the continuity of the family $p_\alpha$ with respect to $\alpha$, the number of zeros
of $p_\alpha$ and $\partial_{\theta}^{a}p_{\alpha}$ in the interval $[0, \pi/2]$ remains constant as $\alpha$ ranges between $2m$ and $2m+2s$, and between $2m+2s$ and $2m$, increasing exactly by one each time $\alpha$ passes $2m$ or $2m+2s$. 
Hence $p_\alpha(\pi/2)$ is positive if and only if $\alpha\in(0,2s)\cup\bigcup_{m\in\N_{\ge1}}(4m-2+2s,4m+2s)$, while $\partial_{\theta}^ap_\alpha(\pi/2)$ is positive if and only if $\alpha\in\bigcup_{m\in\N_{\ge1}}(4m-2,4m)$, and this shows \eqref{eq:claim-gap}.
\end{proof}
\begin{proof}[Proof of \cref{teo:gaps}]
    It is enough to collect together \cref{prop:gap-odd}, \cref{prop:explicit-frequency-gap}, \cref{prop:snot=1/2}, \cref{remark:closeness-admissible-freq}, \cref{prop:gap-fs24} and \cite[Proposition 1.3]{car24}.
\end{proof}
\section{Dimension reduction arguments}\label{section-dim_reduc}
In this section we will estimate the Hausdorff dimension of the sets $\mathbf{\Gamma}_{\ge 2}\setminus \Gammab_{\ge k+\gamma}$, $\mathbf{\Gamma}_{*}$ and $\Gammab_{2}^{\rm a}$ by using dimension reduction arguments which take advantage of the classification of admissible frequencies in low dimension. In the case $s=1/2$ and $\vf \equiv 0$ we recover the dimensional bounds obtained in \cite{ft23}. Recall the definition of the map $\tau: \Gammab \to [0,1]$ in \cref{prop:continuity-tau}.
\subsection{Dimensional bound for $\boldsymbol{\Gammab_{\ge 2}\setminus \Gammab_{\ge k+\gamma}}$ and $\boldsymbol{\Gammab_{*}}$}
In this subsection we estimate the size of the sets $\mathbf{\Gamma}_{\ge 2}\setminus \Gammab_{\ge k+\gamma}$ and $\mathbf{\Gamma}_{*}$. 
The analogous result in the case $s=1/2$ and $\vf\equiv0$ was proved in \cite[Proposition 3.1]{ft23}.
\begin{proposition}\label{prop:dimension-reduction-ge2e*}
    Let $u:B_{1}\times [0,1]\to \R$ be a family of solutions to \eqref{def:soluzione-estesa}, \eqref{def:soluzione-famiglia} with obstacle $\vf$ satisfying \eqref{e:hypo-phi}. Then:
        \begin{itemize}
        \item[i)] $\text{dim}_{\mathcal{H}}(\Gammab_{\ge 2}\setminus \Gammab_{\ge k+\gamma})\le n-1$ if $n\ge 2$, and $\Gammab_{\ge 2}\setminus \Gammab_{\ge k+\gamma}$ is discrete if $n=1$;
        \smallskip
        \item[ii)] $\text{dim}_{\mathcal{H}}(\Gammab_{*})\le n-2$ if $n\ge 3$, $\Gammab_{*}$ is discrete if $n=2$, and it is empty if $n=1$. 
        \end{itemize}  
\end{proposition}
\noindent We first need some preliminary lemmas. The first one settles some useful properties of homogeneous solutions with zero obstacle (see also \cite[Lemma 2.1, Lemma 2.6]{ft23} for the case $s=1/2$).
\begin{lemma}\label{lemma:properties-homogeneous-solutions}
    Let $u:\R^{n+1}\to\R$ be a $\kappa$-homogeneous solution to \eqref{def:soluzione-estesa} with zero obstacle. Then:
    \begin{itemize}
        \item [i)] if $u\ge 0$, or $u\le 0$ and $\kappa>2s$, then $u\equiv 0$;
        \smallskip
        \item [ii)] if $\partial_{e}u\ge 0$ for some non-zero vector $e\in \R^{n}\times \left\{0\right\}$ and $\kappa \ge 2$, then $\partial_{e}u \equiv 0$;
        \smallskip
        \item [iii)] if $\kappa>2s$ and $v$ is another solution to \eqref{def:soluzione-estesa} with zero obstacle such that $v\ge u$ in $B_{1}$ and $v(0)=0$, then $v\equiv u$.
    \end{itemize}
\end{lemma}
\begin{proof} Let us prove point i). If $u\ge 0$ and $u\not \equiv 0$, then, by the strong maximum principle, $u>0$ in $\R_{+}^{n+1}$ and \cref{lemma:Hopf} gives $\lim_{y\downarrow 0}y^{a}\partial_{y}u(0,y)>0,$ which, thanks to formula \eqref{formula:y-derivative-on-thin-space}, contradicts the fact that $u$ is $L_a$-superharmonic. If instead $u\le 0$, $\kappa >2s$ and $u\not \equiv 0$, from the same reasoning as above, we have that $\lim_{y\downarrow 0}y^{a}\partial_{y}u(0,y)<0,$ which contradicts the $(a+\kappa-1)$-homogeneity of $|y|^{a}u$. 

Then we prove point ii). Let us call $v:= \partial_{e}u$. By assumption we have $v\ge 0$ and since $\kappa >1$, $v(0)=0$. We suppose that $v \not\equiv 0$, otherwise we are done. Then the strong maximum principle gives $v>0$ in $\R_{+}^{n+1}$ and \cref{lemma:Hopf} yields $\lim_{y\downarrow 0}y^{a}\partial_{y}v(0,y)>0,$ from which we deduce that 
    $2\le\kappa \le 1+2s$. Hence, by \cref{teo:gaps}, $\kappa=2$ and $v$ is a non-negative linear function, thus $v\equiv0$, a contradiction.

    Finally, for point iii), we assume by contradiction that $v\not\equiv u$. Then, by the strong maximum principle, $v>u$ in $\R_{+}^{n+1}$, and by \cref{lemma:Hopf}, combined with the fact that $\kappa >2s$,
    $$\lim_{y\downarrow 0}y^{a}\partial_{y}v>\lim_{y\downarrow 0}y^{a}\partial_{y}u =0,$$
    a contradiction, since $v$ is $L_a$-superharmonic.    
\end{proof}
Next, we prove the following accumulation lemma, which is similar to \cite[Lemma 3.2]{ft23}.
\begin{lemma}\label{lemma:fundamental-dimension-reduction}
    Let $u:B_{1}\times [0,1]\to \R$ be a family of solutions to \eqref{def:soluzione-estesa}, \eqref{def:soluzione-famiglia} with obstacle $\vf$ satisfying \eqref{e:hypo-phi}. Let $\bar{X}\in \Gammab_{\ge 2}\setminus \Gammab_{\ge k+\gamma}$. Suppose that there exist a sequence of radii $r_{j}\downarrow 0$ and a sequence of points $X_{j}\in \Gammab_{\ge 2}\setminus \Gammab_{\ge k+\gamma}$ such that $|X_{j}-\bar{X}|\le r_{j}$ and
    $$Y_{j}:=\frac{X_{j}-\bar{X}}{r_{j}}\to \bar{Y}\neq 0,\quad \lambda_{j}:=N(0^{+},\widetilde{u}(\cdot,\tau(X_{j})))\to N(0^{+},\widetilde{u}(\cdot,\tau(\bar{X})))=:\lambda.$$
    If $q$ is any non-zero $\lambda$-homogeneous solution of \eqref{def:soluzione-estesa} with zero obstacle obtained as a blow-up of  $\widetilde{u}^{\bar{X}}(\cdot,\tau(\bar{X}))$ at the point $\bar{X}$ along the sequence $r_{j}$, then $q$ is translation invariant in the direction $\bar{Y}$. 
\end{lemma}
\begin{proof}
    We define the following functions:
    \begin{gather*}
        v_{j}:= \widetilde{u}^{\bar{X}}(\bar{X}+r_{j}\cdot,\tau(\bar{X})),\\
        w_{j}:= \widetilde{u}^{X_{j}}(X_{j}+r_{j}\cdot,\tau(X_{j})),\\
        z_{j}:= \widetilde{u}^{X_{j}}(X_{j}+r_{j}\cdot,\tau(\bar{X})).
    \end{gather*}
    Observe that $\tau(X_{j})\to \tau(\bar{X})$ as $X_{j}\to \bar{X}$, thanks to \cref{prop:continuity-tau}. We may assume without loss of generality that $\tau(X_{j})\le \tau(\bar{X})$, so that, by the monotonicity assumption in the definition of a family of solutions \eqref{def:soluzione-famiglia}, we have $w_{j}\le z_{j}$.
    We suppose that 
    $$\frac{v_{j}}{\lVert v_{j}\rVert_{L^{2}(\partial B_{1},|y|^{a})}}\rightharpoonup q\quad \mbox{in } W_{\loc}^{1,2}(\R^{n+1},|y|^{a}),$$
    where $q$ is a $\lambda$-homogeneous solution of \eqref{def:soluzione-estesa} with zero obstacle such that $\lVert q\rVert_{L^{2}(\partial B_{1},|y|^{a})}=1$, and we wish to prove that $q$ is invariant in the direction $\bar{Y}$. First of all, since $\lambda_{j}\in [2,\kappa+\gamma)$ and $X_{j}\to \bar{X}$, the sequence $w_{j}/\lVert w_{j}\rVert_{L^{2}(\partial B_{1},|y|^{a})}$ is bounded in $W_{\loc}^{1,2}(\R^{n+1},|y|^{a})$ and in $C_{a,\loc}^{1,s}(\R^{n+1}_+)$. In particular, since in addition $\lambda_{j}\to \lambda$, up to subsequences
    $$\frac{w_{j}}{\lVert w_{j}\rVert_{L^{2}(\partial B_{1},|y|^{a})}}\rightharpoonup \widetilde{q}\quad \mbox{in } W_{\loc}^{1,2}(\R^{n+1},|y|^{a}),$$
    where $\widetilde{q}$ is a $\lambda$-homogeneous solution of \eqref{def:soluzione-estesa} with zero obstacle such that $\lVert \widetilde{q}\rVert_{L^{2}(\partial B_{1},|y|^{a})}=1$. Then observe that we can re-write
    $$z_{j}=v_{j}(Y_{j}+\cdot)+\left(\widetilde{\varphi}^{\bar{X}}-\widetilde{\varphi}^{X_{j}}\right)(X_{j}+r_{j}\cdot),$$ where $\widetilde \vf^{\bar X}$ and $\widetilde \vf^{ X_j}$ are as in \eqref{def:phitilde}.
    Therefore, since 
    $$\left|\left(\widetilde{\varphi}^{\bar{X}}-\widetilde{\varphi}^{X_{j}}\right)(X_{j}+r_{j}\cdot)\right|=O(r_{j}^{\kappa+\gamma})\quad\mbox{on } \partial B_{1}\qquad\mbox{and}\qquad \lVert v_{j}\rVert_{L^{2}(\partial B_{1},|y|^{a})}\ge Cr_{j}^{k+\gamma-\eps},$$
    we have
    $$\frac{z_{j}}{\lVert v_{j}\rVert_{L^{2}(\partial B_{1},|y|^{a})}}=\frac{v_{j}(Y_{j}+\cdot)}{\lVert v_{j}\rVert_{L^{2}(\partial B_{1},|y|^{a})}}+\frac{\left(\widetilde{\varphi}^{\bar{X}}-\widetilde{\varphi}^{X_{j}}\right)(X_{j}+r_{j}\cdot)}{\lVert v_{j}\rVert_{L^{2}(\partial B_{1},|y|^{a})}}\rightharpoonup q(\bar{Y}+\cdot)$$
    in $W_{\loc}^{1,2}(\R^{n+1},|y|^{a})$. From this we deduce that
    $$\frac{z_{j}}{\lVert z_{j}\rVert_{L^{2}(\partial B_{1},|y|^{a})}}\rightharpoonup \frac{q(\bar{Y}+\cdot)}{\lVert q(\bar{Y}+\cdot)\rVert_{L^{2}(\partial B_{1},|y|^{a})}}=:\hat{q}\quad \mbox{in } W_{\loc}^{1,2}(\R^{n+1},|y|^{a}).$$
    Now, calling $\eps_{j}:=\lVert w_{j}\rVert_{L^{2}(\partial B_{1},|y|^{a})}+\lVert z_{j}\rVert_{L^{2}(\partial B_{1},|y|^{a})}$, we clearly have, up to subsequences,
    $$\frac{w_{j}}{\eps_{j}}\rightharpoonup \alpha\widetilde{q},\quad \frac{z_{j}}{\eps_{j}}\rightharpoonup \beta\hat{q}\quad \mbox{in }W_{\loc}^{1,2}(\R^{n+1},|y|^{a})$$
    for some $\alpha,\beta \in [0,1]$ such that $\alpha+\beta=1$. In addition, from the ordering $w_{j}\le z_{j}$, we deduce that $\alpha\widetilde{q}\le \beta\hat{q}$. Hence, in particular, none of $\alpha$ and $\beta$ can be zero. In fact, if $\alpha=0$ or $\beta=0$, we would obtain respectively that $q\ge 0$, or $\widetilde{q}\le 0$, which are both excluded by point i) in \cref{lemma:properties-homogeneous-solutions} as $\lambda \ge 2$. Therefore $\alpha,\beta \in (0,1)$. Now, calling $\sigma:=\lVert q(\bar{Y}+\cdot)\rVert_{L^{2}(\partial B_{1},|y|^{a})}$, and using the $\lambda$-homogeneity of both $q$ and $\widetilde{q}$ we get
    $$q(\rho\bar{Y}+\cdot)\ge \frac{\alpha \sigma}{\beta}\widetilde{q}(\cdot)\quad \mbox{for every } \rho >0,$$
    from which we deduce that  $q\ge (\alpha \sigma/\beta)\widetilde{q}$. This in turn implies that $q= (\alpha \sigma/\beta)\widetilde{q}$ by point iii) in \cref{lemma:properties-homogeneous-solutions}. Hence,
    $$q(\rho\bar{Y}+\cdot)\ge q(\cdot)\quad \mbox{for every } \rho >0,$$
    which gives that $\partial_{\bar{Y}}q\ge 0$. By point ii) in \cref{lemma:properties-homogeneous-solutions} we finally obtain that $\partial_{\bar{Y}}q\equiv 0$, as desired.
\end{proof}
\begin{proof}[Proof of \cref{prop:dimension-reduction-ge2e*}]
    Let us prove point i) starting from the case $n\ge 2$. We apply \cref{lemma:characterization-of-m-dim-set} in the way explained in \cref{remark:use-of-lemma-char-m-dim-sets} to $E=\Gammab_{\ge 2}\setminus \Gammab_{\ge k+\gamma}$, $f:E\rightarrow \R$ given by $f(X):=N(0^{+},\widetilde{u}(\cdot, \tau(X)))$ and $m=n-1$. Assume by contradiction that $\dim_{\mathcal{H}}(E)>n-1$. Then there exist $\bar{X}\in E$, a sequence of radii $r_{j}\downarrow 0$ and $n$ sequences of points $X^{1}_{j},\dots,X^{n}_{j} \in E$, such that
    $$|X^{\ell}_{j}-\bar{X}|\le r_{j},\quad \frac{X^{\ell}_{j}-\bar{X}}{r_{j}}\to \bar{Y}^{\ell}\neq 0,\quad f(X^{\ell}_{j})\to f(X)=:\lambda \ge 2,$$
    with $\bar{Y}^{1},\dots, \bar{Y}^{n}$ linearly independent vectors in the thin space.
    Let $q$ be any non-zero $\lambda$-homogeneous solution to \eqref{def:soluzione-estesa} with zero obstacle obtained as a blow-up of $\widetilde{u}^{\bar{X}}(\cdot,\tau(\bar{X}))$ at the point $\bar{X}$. By \cref{lemma:fundamental-dimension-reduction}, $q$ is invariant in each of the directions $\bar{Y}^{1},\dots, \bar{Y}^{n}$. In particular we have found a non-zero solution of \eqref{def:soluzione-estesa} with zero obstacle in dimension $n+1=1$ which is $\lambda$-homogeneous, with $\lambda \ge 2$, and this contradicts point i) in \cref{prop:classification-solution-dim2}. Hence $\dim_{\mathcal{H}}(\Gammab_{\ge 2}\setminus \Gammab_{\ge k+\gamma})\le n-1$. If $n=1$, the same argument shows that points in $\Gammab_{\ge 2}\setminus \Gammab_{\ge k+\gamma}$ cannot accumulate. Concerning point ii), the case $n\ge 3$  is proved analogously, using $E:= \Gammab_{*}$ and $m=n-2$, this time contradicting point ii) of \cref{prop:classification-solution-dim2}. When $n=2$ the same argument shows that points in $\Gammab_{*}$ cannot accumulate, while if $n=1$, $\Gammab_{*}$ is clearly empty by definition.
\end{proof}
\subsection{Dimensional bound for $\boldsymbol{\Gammab_{2}^{\rm a}}$}
For the anomalous quadratic points, i.e.~quadratic points where the homogeneity of the second blow-up is in $[2,3)$, we follow the same strategy as in the previous subsection. In this case, the second blow-up takes the role that was previously played by the first blow-up. We will prove the following proposition (see \cite[Proposition 4.3]{ft23} for the case $s=1/2$ with zero obstacle).
\begin{proposition}\label{prop:gamma-anomalous}
    Let $u:B_1\times[0,1]\to\R$ be a family of solutions to \eqref{def:soluzione-estesa}, \eqref{def:soluzione-famiglia} with obstacle $\varphi$ satisfying \eqref{e:hypo-phi}, with $k\ge 4$. Then $\text{dim}_\HH
    (\Gammab_{2}^{\rm a})\le n-2$ if $n \ge 3$, $\text{dim}_\HH(\Gammab_{2}^{\rm a})$ is discrete if $n=2$, and it is empty if $n=1$.
\end{proposition}
\noindent We point out that the lower bound $k\ge4$ is necessary for the set  $\Gammab_{2}^{\rm a}$ to be well-defined (see \cref{section-quadratic}).

To prove \cref{prop:gamma-anomalous} we need some preliminary lemmas. The first one is the analogue of \cref{lemma:properties-homogeneous-solutions} of the previous subsection.
\begin{lemma}\label{lemma:properties-homogeneous-solutions-verythin}
    Suppose that $s>1/2$, and let $u:\R^{n+1}\to\R$ be a $\kappa$-homogeneous solution to the very thin obstacle problem \eqref{def:very-thin-obstacle-problem} on $L=\{x_{n}=y=0\}$. Then,
    \begin{itemize}
        \item [i)] if $u\ge 0$, or $u\le 0$ and $\kappa>2s-1$, then $u\equiv 0$;
        \smallskip
        \item [ii)] if $\partial_{e}u\ge 0$ for some non-zero vector $e\in L$ and $\kappa \ge 2$, then $\partial_{e}u \equiv 0$;
        \smallskip
        \item [iii)] if $\kappa>2s-1$ and $v$ is another solution to \eqref{def:soluzione-estesa} such that $v\ge u$ in $B_{1}$ and $v(0)=0$, then $v\equiv u$.
    \end{itemize}
\end{lemma}
\begin{proof}
    According to \cref{prop:operator-La-forverythinobstacle}, we use the following notation:
    $$f_{a}[u](0):= \underset{\eps \downarrow 0}{\lim}\int_{\partial D_{\eps}}u_{\nu}(0,x_{n},y)|y|^{a}\,d\mathcal{H}^{1}.$$
    To prove point i) we first observe that, by the maximum principle in \cref{prop:maximum-principle}, we can suppose that either $u>0$ or $u<0$ in $\R^{n+1}\setminus L$. Now,  if $u>0$ in $\R^{n+1}\setminus L$, then $f_{a}[u](0)>0$ by \cref{hopf-verythincase}. This is a contradiction with \cref{prop:operator-La-forverythinobstacle} and with the fact that $u$ is $L_{a}$-superharmonic. If instead $u<0$ in $\R^{n+1}\setminus L$ and $\kappa>2s-1$, then $f_{a}[u](0)=0$, by homogeneity, which contradicts \cref{hopf-verythincase}. 
    
    Let us now address point ii). Call $w=\partial_{e}u$ and notice that $w$ is $(\kappa-1)$-homogeneous and satisfies $L_{a}w=0$ in $\R^{n+1}\setminus L$. Suppose that $w>0$ in $\R^{n+1}\setminus L$. If $\kappa>2s$, then $f_{a}[w](0)=0$ by homogeneity, which contradicts \cref{hopf-verythincase}. 
    
    Finally, let us show point iii). Suppose that $v>u$ on $\R^{n+1}\setminus L$. Then, again by \cref{hopf-verythincase} we must have $f_{a}[v](0)>f_{a}[u](0)$. Now, if $\kappa>2s-1$, then $f_{a}[u](0)=0$, which means that $f_{a}[v](0)>0$. This contradicts \cref{prop:operator-La-forverythinobstacle}, since $v$ is $L_{a}$-superharmonic. 
\end{proof}
We proceed with the following accumulation lemma, which is the analogue of \cref{lemma:fundamental-dimension-reduction} for the set $\Gammab_{2}^{\rm a}$ (see also \cite[Lemmas 4.4-4.6]{ft23} for the case $s=1/2$ and $\vf\equiv0$).
\begin{lemma}\label{lemma:accumulation-quadratic-points}
    Let $u:B_{1}\times [0,1]\to \R$ be a family of solutions to \eqref{def:soluzione-estesa}, \eqref{def:soluzione-famiglia} with obstacle $\vf$ satisfying \eqref{e:hypo-phi}, with $k\ge4$. Let $\bar{X}\in \Gammab_{2}^{\rm a}$, call $p\in \mathcal{P}_{2}$ the first blow-up of $\widetilde{u}^{\bar{X}}(\cdot, \tau(\bar{X}))$ at $\bar{X}$ and $$v:= \widetilde{u}^{\bar{X}}(\bar{X}+\cdot, \tau(\bar{X}))-p.$$ Suppose that there exist a sequence of radii $r_{j}\downarrow 0$ and a sequence of points $X_{j}\in \Gammab_{2}^{\rm a}$ with $|X_{j}-\bar{X}|\le r_{j}$ such that, calling $p_{j}\in \mathcal{P}_{2}$ the first blow-up of $\widetilde{u}^{X_{j}}(\cdot, \tau(X_{j}))$ at $X_{j}$ and $$v_{j}:= \widetilde{u}^{X_{j}}(X_{j}+\cdot, \tau(X_{j}))-p_{j},$$ we have 
     $$Y_{j}:=\frac{X_{j}-\bar{X}}{r_{j}}\to \bar{Y}\neq 0,\quad \lambda_{j}:= N(0^{+},v_{j})\to N(0^{+},v)=:\lambda.$$
    Then, $\bar{Y}\in L= \{(x,0): p(x,0)=0\}$. Moreover, let $q\in W_{\loc}^{1,2}(\R^{n+1},|y|^{a})$ be any non-zero $\lambda$-homogeneous function obtained as a blow-up of $v$ at $0$ along the sequence $r_{j}$:
    \begin{itemize}
        \item [a)] if $\lambda=2$, then $q(\bar{Y})=0$;
        \smallskip
        \item [b)] if $\lambda\in (2,3)$, then $q$ is translation invariant in the direction $\bar{Y}$.
    \end{itemize}
\end{lemma}
\begin{proof}
    We divide the proof into several steps. 

    \smallskip
    \noindent \textbf{Step 1}: ($p_{j}$ converges to $p$). On the one hand, by \cref{prop:classification-blowup-freq-2},
    $$\lVert v\rVert_{L^{\infty}(B_{r})}+ \lVert v_{j}\rVert_{L^{\infty}(B_{r})}\le Cr^{2}\omega(r)\quad \mbox{for every } r\in (0,r_{0}),$$
    where $\omega(r):= (-\log(r))^{-c}\to 0$ as $r\downarrow 0$. On the other hand, by \cref{prop:continuity-tau}, 
    $$\widetilde{u}^{X_{j}}(X_{j}+\cdot, \tau(X_{j}))\to \widetilde{u}^{\bar{X}}(\bar{X}+\cdot, \tau(\bar{X})),$$ where the convergence is uniform.
    Thus, for every $r\in (0,r_{0})$,
    \begin{align*}
        \underset{j\to \infty}{\limsup}\lVert p_{j}-p\rVert_{L^{\infty}(B_{r})}&\le \underset{j\to \infty}{\limsup}\left(\lVert v\rVert_{L^{\infty}(B_{r})}+ \lVert v_{j}\rVert_{L^{\infty}(B_{r})}\right)\\
        &+ \underset{j\to \infty}{\limsup} \lVert \widetilde{u}^{X_{j}}(X_{j}+\cdot, \tau(X_{j}))- \widetilde{u}^{\bar{X}}(\bar{X}+\cdot, \tau(\bar{X}))\rVert_{L^{\infty}(B_{r})}\\
        &\le Cr^{2}\omega(r),
    \end{align*}
    which implies that $p_{j}\to p$ locally uniformly, as $p,p_{j}\in \mathcal{P}_{2}$ are quadratic polynomials.

    \smallskip
    \noindent \textbf{Step 2}: ($\bar{Y}\in L$). As $L$ coincides with the subspace of $\R^{n}\times \{0\}$ of invariance directions of $p$, it is enough to prove that $p(\bar{Y}+\cdot)\equiv p$. 
    We may assume without loss of generality that $\tau(X_{j})\ge \tau(\bar{X})$. Then, thanks to the monotonicity assumption in \eqref{def:soluzione-famiglia} and recalling $\widetilde \vf^{\bar X}$ and $\widetilde \vf^{ X_j}$ as in \eqref{def:phitilde}, in $B_{2}$:
    \begin{align*}
        r_{j}^{-2}\widetilde{u}^{X_{j}}(X_{j}+r_{j}\cdot, \tau(X_{j}))&\ge r_{j}^{-2}\widetilde{u}^{\bar{X}}(\bar{X}+r_{j}(Y_{j}+\cdot), \tau(\bar{X}))\\
        &+r_{j}^{-2}\left(\widetilde{\varphi}^{\bar{X}}(\bar{X}+r_{j}(Y_{j}+\cdot))-\widetilde{\varphi}^{X_{j}}(X_{j}+r_{j}\cdot)\right)\\
        &\ge r_{j}^{-2}\widetilde{u}^{\bar{X}}(\bar{X}+r_{j}(Y_{j}+\cdot), \tau(\bar{X}))-Cr_{j}^{k+\gamma-2},
    \end{align*} where in the last inequality we used the regularity assumption on $\vf$.
    Hence, by \cref{prop:classification-blowup-freq-2} we get, in $B_{2}$:
    \begin{align*}
        C\omega(r_{j})+p_{j}&\ge r_{j}^{-2}\widetilde{u}^{X_{j}}(X_{j}+r_{j}\cdot, \tau(X_{j}))\\
        &\ge r_{j}^{-2}\widetilde{u}^{\bar{X}}(\bar{X}+r_{j}(Y_{j}+\cdot), \tau(\bar{X}))-Cr_{j}^{k+\gamma-2}\\
        &\ge p(Y_{j}+\cdot)-C\omega(r_{j})-Cr_{j}^{k+\gamma-2}.
    \end{align*}
    From the inequality above we see that the $L_{a}$-harmonic functions $$P_{j}:= p_{j}-p(Y_{j}+\cdot)+2C\left(\omega(r_{j})+r_{j}^{k+\gamma-2}\right)$$ are non-negative in $B_{2}$, and since $p\ge 0$ on the thin space, we also have $P_{j}(0)\le 2C\left(\omega(r_{j})+r_{j}^{k+\gamma-2}\right)$. Hence, from the Harnack inequality for $L_{a}$ (see \cref{lemma:harnack}) we deduce
    $$\lVert p_{j}-p(Y_{j}+\cdot)\rVert_{L^{\infty}(B_{1})}\le C\left(\omega(r_{j})+r_{j}^{k+\gamma-2}\right).$$
    Finally, passing to the limit in $j\to \infty$ and using Step 1 we get $\lVert p-p(\bar{Y}+\cdot)\rVert_{L^{\infty}(B_{1})}=0$, as desired.

    \smallskip
    \noindent \textbf{Step 3}: (Blow-up analysis). Let us suppose that
    $$\frac{v(r_{j}\cdot)}{\lVert v(r_{j}\cdot)\rVert_{L^{2}(\partial B_{1},|y|^{a})}}\rightharpoonup q\quad \mbox{in }W_{\loc}^{1,2}(\R^{n+1},|y|^{a}),$$
    where $q$ is non-zero and $\lambda$-homogeneous. Moreover $q$ is an $L_{a}$-harmonic polynomial if $\lambda=2$, or a solution of the very thin obstacle problem \eqref{def:very-thin-obstacle-problem} on $L$ if $\lambda\in (2,3)$.
    We consider the functions $$w_{j}:= \widetilde{u}^{X_{j}}(X_{j}+\cdot, \tau(X_{j}))-p.$$
    Then $w_{j}$ can be written as the sum of three terms: $w_{j}=w_{j}^{(1)}+w_{j}^{(2)}+w_{j}^{(3)}$, where 
    \begin{align*}
        w_{j}^{(1)}&:= \widetilde{u}^{X_{j}}(X_{j}+\cdot, \tau(X_{j}))-\widetilde{u}^{\bar{X}}(X_{j}+\cdot, \tau(\bar{X})),\\
        w_{j}^{(2)}&:= \widetilde{u}^{\bar{X}}(X_{j}+\cdot, \tau(\bar{X}))-p(X_{j}-\bar{X}+\cdot),\\
        w_{j}^{(2)}&:= p(X_{j}-\bar{X}+\cdot)-p.
    \end{align*}
    First of all, we observe that by the uniform bounds on the family of solutions $u$, the compactness of the sequence $X_{j}$, the convergence of $\lambda_{j}$ to $\lambda$, and that of $p_{j}$ to $p$, one can repeat the same arguments used to prove \cref{second-blowup-quadratic-points} to deduce that, up to subsequences:
    $$\frac{w_{j}(r_{j}\cdot)}{\lVert w_{j}(r_{j}\cdot)\rVert_{L^{2}(\partial B_{1},|y|^{a})}}\rightharpoonup \widetilde{q}\quad \mbox{in }W_{\loc}^{1,2}(\R^{n+1},|y|^{a}),$$
    where $\widetilde{q}$ is non-zero and $\lambda$-homogeneous. Moreover $\widetilde{q}$ is an $L_{a}$-harmonic polynomial if $\lambda=2$, or a solution of the very thin obstacle problem \eqref{def:very-thin-obstacle-problem} on $L$ if $\lambda\in (2,3)$.
    
    Secondly, assuming without loss of generality that $\tau(X_{j})\ge \tau(\bar{X})$, by the monotonicity property \eqref{def:soluzione-famiglia}, we have
    $$w_{j}^{(1)}(r_{j}\cdot)\ge -Cr_{j}^{k+\gamma-2}\quad \mbox{in }B_{1}.$$
    Then observe that $w_{j}^{(2)}(r_{j}\cdot)= v(r_{j}(Y_{j}+\cdot))$. Hence, up to subsequences
    $$\frac{w_{j}^{(2)}(r_{j}\cdot)}{\lVert w_{j}^{2}(r_{j}\cdot)\rVert_{L^{2}(\partial B_{1},|y|^{a})}}\rightharpoonup \frac{q(\bar{Y}+\cdot)}{\lVert q(\bar{Y}+\cdot)\rVert_{L^{2}(\partial B_{1},|y|^{a})}}\quad \mbox{in }W_{\loc}^{1,2}(\R^{n+1},|y|^{a}).$$
    Finally, let us call $Z_{j}$ the projection of $Y_{j}$ on $L$. Since $p$ is invariant in direction $Z_{j}$, we have $$w_{j}^{(3)}(r_{j}\cdot)= r_{j}^{-2}\left(p((Y_{j}+\cdot))-p\right)= r_{j}^{-2}\left(p((Y_{j}-Z_{j}+\cdot))-p\right).$$
    Hence, up to subsequences, we have
     $$\frac{w_{j}^{(3)}(r_{j}\cdot)}{\lVert w_{j}^{(3)}(r_{j}\cdot)\rVert_{L^{2}(\partial B_{1},|y|^{a})}}\rightharpoonup \nabla p \cdot e\quad \mbox{in }W_{\loc}^{1,2}(\R^{n+1},|y|^{a}),$$
    where $e\in \R^{n}\times \{0\}$ is some vector such that $e\perp L$. Let us call $$\eps_{j}:= \lVert w_{j}^{(1)}(r_{j}\cdot)\rVert_{L^{2}(\partial B_{1},|y|^{a})}+\lVert w_{j}^{(2)}(r_{j}\cdot)\rVert_{L^{2}(\partial B_{1},|y|^{a})}+\lVert w_{j}^{(3)}(r_{j}\cdot)\rVert_{L^{2}(\partial B_{1},|y|^{a})}.$$
    Up to extracting a further subsequence, we get
    $$
    \left\{
    \begin{array}{rclll}
        w_{j}(r_{j}\cdot)/\eps_{j}&\rightharpoonup& c\widetilde{q}\\
        w_{j}^{(1)}(r_{j}\cdot)/\eps_{j}&\rightharpoonup& \hat{q}\ge 0\\
        w_{j}^{(2)}(r_{j}\cdot)/\eps_{j}&\rightharpoonup& c_{2}q(\bar{Y}+\cdot)\\
        w_{j}^{(3)}(r_{j}\cdot)/\eps_{j}&\rightharpoonup& c_{3}\nabla p \cdot e
    \end{array}
    \right.\qquad \mbox{in } W_{\loc}^{1,2}(\R^{n+1},|y|^{a}),
    $$
    where $c, c_{2}, c_{3}\in [0,1]$ are not all zero, and $\hat{q}$ is $L_{a}$-harmonic away from $L$ (to get $\hat{q}\ge 0$ we have used that $\eps_{j}\ge Cr_{j}^{3}$ for some $C>0$, as $\lambda<3$).

    \smallskip
    \noindent \textbf{Step 4}: (Conclusion). From Step 3 we obtain the balance equation
    \begin{equation}\label{eq:balance-conclusion-dimreduction-quadratic}
        c\widetilde{q}=\hat{q}+c_{2}q(\bar{Y}+\cdot)+c_{3}\nabla p \cdot e.
    \end{equation}
    We consider the two cases $\lambda=2$ and $\lambda\in (2,3)$ separately. 

    \textit{Case a)}.
    If $\lambda=2$, then $q$ and $\widetilde{q}$ are both $L_{a}$-harmonic, thus $\hat{q}$ is $L_{a}$-harmonic too, and since $\hat{q}\ge 0$, by Liouville's theorem (see \cref{prop:Liouville}), $\hat{q}\equiv 0$. Then
    $$c\widetilde{q}=c_{2}q(\bar{Y}+\cdot)+c_{3}\nabla p \cdot e.$$
    Taking the partial derivative of this expression at $0$ in direction $\bar{Y}$, using that $q$ and $\widetilde{q}$ are $2$-homogeneous and that $\partial_{\bar{Y}}p=0$, we get
    $$2c_{2}q(\bar{Y})=0.$$
    Finally observe that $c_{2}$ cannot be zero, otherwise by comparison of homogeneities we would also have $c=c_{3}=0$. Hence $q(\bar{Y})=0$ as desired.

    \textit{Case b)}.
    Suppose instead that $\lambda\in (2,3)$. Then $s>1/2$, $L$ is $(n-1)$-dimensional, and both $q$ and $\widetilde{q}$ are $\lambda$-homogeneous solutions of the very thin obstacle problem \eqref{def:very-thin-obstacle-problem} on $L$. Notice that $\widetilde{q}$ and $q(\bar{Y}+\cdot)$, when restricted to the thin space, must be even with respect to $L$. Indeed, their odd parts are $L_{a}$-harmonic in the whole $\R^{n+1}$. Hence if not zero, they would be polynomials with integer homogeneity, which contradicts $\lambda\in (2,3)$. Using the balance \eqref{eq:balance-conclusion-dimreduction-quadratic}, we deduce that the odd part of $\hat{q}$ must cancel exactly $c_{2}\nabla p \cdot e$. Therefore, as $\hat{q}\ge 0$, we get
    $$c\widetilde{q}\ge c_{2}q(\bar{Y}+\cdot).$$
    We observe that none of $c$ and $c_{2}$ can be zero, otherwise, by \cref{lemma:properties-homogeneous-solutions-verythin} point i) we would get either $q\equiv 0$ or $\widetilde{q}\equiv 0$, a contradiction. Thus, using the $\lambda$-homogeneity of both $q$ and $\widetilde{q}$, we deduce that indeed
    $$c\widetilde{q}\ge c_{2}q.$$
    Now \cref{lemma:properties-homogeneous-solutions-verythin} point iii) gives $c\widetilde{q}= c_{2}q$, hence 
    $$q\ge q(\bar{Y}+\cdot),$$
    which in turn implies, again by homogeneity, that $\partial_{-\bar{Y}}q\ge 0$. Finally, \cref{lemma:properties-homogeneous-solutions-verythin} point ii) ensures that $\partial_{\bar{Y}}q\equiv 0$, thus $q$ is translation invariant in direction $\bar{Y}$, as desired.
\end{proof}
\begin{proof}[Proof of \cref{prop:gamma-anomalous}]
    We split $\Gammab_{2}^{\rm a}$ into the disjoint sets $\Gammab_{2,2}$ and $\Gammab_{2,(2,3)}:=\Gammab_{2}^{\rm a}\setminus \Gammab_{2,2}$. We will use two slightly different dimension reduction arguments to get the dimensional bound on the two parts of the splitting.
    
    \smallskip
    \noindent \textbf{Step 1:} (Bound on $\Gammab_{2,2}$). We first assume that $n\ge3$. Remember that $\Gammab_{2,2}$ is made of those quadratic points at which second blow-ups are $2$-homogeneous $L_{a}$-harmonic polynomials (recall \cref{second-blowup-quadratic-points}). Suppose by contradiction that 
    $$\mathcal{H}^{\beta}(\Gammab_{2,2})>0$$
    for some $\beta>n-2$. Then, by classical properties of Haussdorf measures, there must exist a point $\bar{X}\in \Gammab_{2,2}$, a set $A\subset B_{1}'$ with $\mathcal{H}^{\beta}(A)>0$, and a sequence of radii $r_{j}\downarrow 0$ such that, for each $\bar{Y} \in A$, there exists a sequence $X_{j}\in \Gammab_{2,2}$ with $|X_{j}-\bar{X}|\le r_{j}$ for which $(X_{j}-\bar{X})/r_{j}\to \bar{Y}$. Let $q$ be a non-zero $2$-homogeneous $L_{a}$-harmonic polynomial obtained as a second blow-up of $\widetilde{u}^{\bar{X}}(\cdot, \tau(\bar{X}))$ at $\bar{X}$. Then, thanks to \cref{lemma:accumulation-quadratic-points}, we have 
    $$A\subset \{(x,0): p(x,0)=q(x,0)=0\}.$$
    Now, since $\mathcal{H}^{\beta}(A)>0$ and $\beta>n-2$, it must be the case that $L=\{(x,0):p(x,0)=0\}$ is an $(n-1)$-dimensional subspace of $\R^{n}\times \{0\}$ and $q\equiv 0$ on $L$. This is because both $p$ and $q$ are $2$-homogeneous polynomials and $p$ is non-negative on the thin space, therefore $L$ coincides with the subspace of invariant directions of $p(\cdot,0)$. We may then assume without loss of generality that 
    $$L=\{x_{n}=y=0\}.$$
    As $p$ and $q$ are both $L_{a}$-harmonic even in $y$, and vanish on $L$, using also the fact that $p\ge 0$ on the thin space, we deduce that $p$ and $q$ must have this form:
    $$
    \begin{cases}
        p(x,y)=(1+a)bx_{n}^{2}-by^{2};\\
        q(x,y)= (1+a)(c\cdot x)x_{n}-c_{n}y^{2}.
    \end{cases}\quad \mbox{for some } b\in (0,\infty) \mbox{ and } c\in \R^{n}.$$
    Now we exploit the orthogonality conditions \eqref{eq:ortogonality-first-second-bu} using $\bar{p}_{i}(x,y)=(1+a)(c_{i}^{2}x_{i}^{2}+x_{n}^{2}+2c_{i}x_{i}x_{n})-(1+c_{i}^{2})y^{2}$ for $i=1,\dots,n-1$ to get
    \begin{gather*}
        0=\int_{\partial B_{1}}pq\,d\HH^n= bc_{n}\int_{\partial B_{1}}\big((1+a)x_{n}^{2}-y^{2}\big)^{2}\,d\HH^n \implies c_{n}=0;\\
        0\ge \int_{\partial B_{1}}\bar{p}_{i}q\,d\HH^n = 2(1+a)^{2}c_{i}^{2}\int_{\partial B_{1}}x_{i}^{2}x_{n}^{2}\,d\HH^n \implies c_{i}=0 \quad \mbox{for every } i\in \{1,\dots,n-1\}.
    \end{gather*}
    Therefore $q\equiv 0$, which is a contradiction. 

    In the case $n=2$, the same argument gives that $\Gammab_{2,2}$ cannot accumulate. If $n=1$, the orthogonality conditions \eqref{eq:ortogonality-first-second-bu} used as above give that the second blow-up cannot be $2$-homogeneous, and so $\Gammab_{2,2}$ is empty.
    
    \smallskip
    \noindent \textbf{Step 2:} (Bound on $\Gammab_{2,(2,3)}$). We assume $n\ge3$. Recall from \cref{second-blowup-quadratic-points} that $\Gammab_{2,(2,3)}$ is empty if $s\le 1/2$, while if $s>1/2$, it is made of all those quadratic points at which the first blow-up has an $(n-1)$-dimensional spine and each second blow-up solves the very thin obstacle problem \eqref{def:very-thin-obstacle-problem} on such spine. This time the dimension reduction argument will be based on \cref{lemma:characterization-of-m-dim-set}. Consider the function $f:\Gammab_{2,(2,3)}\to (2,3)$ defined as
    $$f(X):= N(0^{+},\widetilde{u}^{X}(X+\cdot, \tau(X))-p^{X})$$
    where $p^{X}\in \mathcal{P}_{2}$ is the first blow-up of $\widetilde{u}^{X}(\cdot, \tau(X))$ at the point $X$. Suppose by contradiction that $\dim_{\mathcal{H}}(\Gammab_{2,(2,3)})>n-2$. Then, by \cref{lemma:characterization-of-m-dim-set} there must exist a point $\bar{X}\in \Gamma_{2,(2,3)}(u)$, a sequence of radii $r_{j}\downarrow 0$ and $n-1$ sequences of points $X^{1}_{j},\dots,X^{n-1}_{j}\in \Gammab_{2,(2,3)}(u)$ with $|X^{i}_{j}-\bar{X}|\le r_{j}$ such that
    $$Y^{i}_{j}:=\frac{X^{i}_{j}-\bar{X}}{r_{j}}\to \bar{Y}^{i}\neq 0,\quad \lambda^{i}_{j}:= f(X^{i}_{j})\to f(\bar{X})=:\lambda,$$
    where $\bar{Y}^{1},\dots,\bar{Y}^{n-1}$ are linearly independent. In particular, if $q$ is any non-zero $\lambda$-homogeneous solution of the very thin obstacle problem \eqref{def:very-thin-obstacle-problem} obtained as a second blow-up of $\widetilde{u}^{\bar{X}}(\cdot, \tau(\bar{X}))$ at $\bar{X}$ along the sequence $r_{j}$, by \cref{lemma:accumulation-quadratic-points}, $q$ has an $(n-1)$-dimensional space of invariant directions in $\R^{n}\times \{0\}$ and so it may be regarded as a solution of the very thin obstacle problem \eqref{def:very-thin-obstacle-problem} in $\R^{2}$. This however contradicts \cref{prop:classification-verythinobstacle-dim2} since $\lambda \in (2,3)$. 

    If $n=2$, the same argument gives that $\Gammab_{2,(2,3)}$ cannot accumulate. Finally, if $n=1$, $\Gammab_{2,(2,3)}$ is empty by \cref{prop:classification-verythinobstacle-dim2}. 
\end{proof}
\section{Cleaning results}\label{section-clean}
In this section we will prove cleaning results for the sets $\Gammab_2^{o}$ and $\Gammab_{2+2s}$, using the rate of convergence to the blow-up and comparison arguments. We will also show how to deduce some cleaning results for high frequencies and for $\Gammab_{*}$ from \cref{prop:cleaning} and the frequency gaps from \cref{teo:gaps}.
Recall the notation \lq\lq Clean'' for the cleaning rate of a family of sets from \cref{def:cleaned}.
\subsection{Cleaning of the set $\boldsymbol{\Gammab_2^{\rm o}}$} For the ordinary harmonic quadratic points $\Gammab_2^{\rm o}$ (see definition in \eqref{def:gammao-and-gammaa}) we prove the following proposition, which is the generalization of \cite[Proposition 4.1]{ft23} to the case $s\in(0,1)$ and $\vf\not\equiv0$.

\begin{proposition}\label{prop:gamma-ordinary}
    Let $u:B_{1}\times [0,1]\to \R$ be a family of solution of \eqref{def:soluzione-estesa}, \eqref{def:soluzione-famiglia}, with obstacle $\vf $ satisfying \eqref{e:hypo-phi}, with $k\ge 4$. Then:
    \begin{itemize}
        \item [a)] if $s\le1/2$, we have $$\{\Gamma_2^{\rm o}(u(\cdot,t))\}_{t\in[0,1]}\in\text{Clean}(3);$$
        \item [b)] if $s>1/2$, we have $$\{\Gamma_2^{\rm o}(u(\cdot, t))\}_{t\in[0,1]}\in\text{Clean}(4-2s).$$ 
    \end{itemize}
\end{proposition}
We first prove a preliminary lemma, which is inspired by \cite[Lemma 4.2]{ft23} and \cite[Lemma 2.8]{fr21}.
\begin{lemma}\label{lemma:detachement}
    Let $u:B_{1}\times [0,1]\to \R$ be a family of solutions to \eqref{def:soluzione-estesa}, \eqref{def:soluzione-famiglia}, with obstacle $\vf $ satisfying \eqref{e:hypo-phi}, with $k\ge 4$. Suppose that $0\in \Gamma_{2}(u(\cdot,0))$. Let $D_r:=\partial B_r\cap \{|y|\ge r/2\}$ and call $h_{t}:= u(\cdot, t)-u(\cdot, 0)$.
    Then, for every $\eps>0$ there are constants $\rho_\eps>0$ and $c_\eps>0$ such that the following holds.
    \begin{itemize}
        \item [a)] If $s\le 1/2$, we have $$\min_{D_r}h_t\ge c_\eps r^\eps t\quad\mbox{for every } r\in(0,\rho_\eps),\ t\in[0,1].$$ 
         \item [b)] If $s>1/2$, we have $$\min_{D_r}h_t\ge c_\eps r^{2s-1+\eps} t\quad\mbox{for every } r\in(0,\rho_\eps),\ t\in[0,1].$$ 
    \end{itemize}
\end{lemma}
\begin{proof} 
    Let $v(\cdot, t)= u(\cdot, t)-\widetilde{\vf}^{0}$ where $\widetilde \vf ^0$ is as in \eqref{def:phitilde}. Let $p\in \mathcal{P}_{2}$ be the blow-up of $v(\cdot,0)$ at $0$. We set $$C_\delta:=\left\{X\in\R^{n+1}:\text{dist}\left(\frac {X}{|X|},L(p)\right)<\delta\right\},$$ where $L(p)\subset \R^{n+1}_{0}$ is given by \eqref{eq:L(p)}. Then $$B_1\cap \{v(\rho\cdot,0)=0\}'\subset C_\delta\quad\mbox{for every } \rho\in(0,2\rho_{\delta})$$ for some constant $\rho_\delta>0$ which depends on $\delta$, since the minimum of $p$ in $B_1'\setminus C_\delta$ is strictly positive. 

    Let $\psi_\delta(r,\theta):=r^{\mu(\delta)}\phi_{\delta}(\theta)$, where $\phi_{\delta}\ge0$ is the first eigenfunction of the spherical operator $L_a^{\mathbb{S}^n}$ (see \Cref{subsection:epiperimetric}) on $\partial B_1\setminus C_\delta$ and $\mu(\delta)$ is chosen so that $\psi_\delta$ is $L_a$-harmonic when positive.  
We have the following cases:
    
   \textit{Case a).} if $s\le 1/2$, then $L(p)$ is a linear space of dimension at most $n-1$ which has zero $L_a$-harmonic capacity. In particular $\mu(\delta)\downarrow 0$ as $\delta\downarrow 0$;
   
    \textit{Case b).} if $s>1/2$, in the worst case scenario $L(p)$ is $(n-1)$-dimensional, thus, since for $\delta=0$ the first eigenfunction up to rotation is $(x_1^2+x_{n+1}^2)^{-\frac a2}$, then $\mu(\delta)\downarrow -a=2s-1$ as $\delta\downarrow 0$.
     
\noindent We call
    $$
    \sigma:= \begin{cases}
        0\quad &\mbox{if } s\le 1/2,\\
        2s-1\quad &\mbox{if } s> 1/2.
    \end{cases}
    $$
    We can choose $\delta\in(0,1/4)$ such that $\mu(\delta)<\sigma+\eps$ and $D_{\rho_\delta}\cap C_{2\delta}=\emptyset$.
Moreover the function $h_t \ge 0$ is $L_a$-harmonic in $$(B_1\setminus \{y=0\})\cup( B_{2\rho_\delta}'\setminus C_\delta).$$ Then, by the Harnack inequality in \cref{lemma:harnack}, there is $c_\delta>0$ such that $h_t\ge c_\delta t$ on $\partial B_{\rho_\delta}\setminus C_{2\delta}$. Using $$w_t=c_\delta t\frac{\psi_{2\delta}}{\|\psi_{2\delta}\|_{L^\infty(\partial B_{\rho_\delta})}}$$ as a lower barrier for $h_t$ in $B_{\rho_\delta}\setminus C_{2\delta}$, we conclude.
\end{proof}
\begin{proof}[Proof of \cref{prop:gamma-ordinary}] We set   $$
\beta:=\begin{cases}
    3\quad &\mbox{if } s\le 1/2,\\
    4-2s\quad &\mbox{if } s> 1/2,
\end{cases}\qquad\mbox{and}\qquad \sigma:= \begin{cases}
        0\quad &\mbox{if } s\le 1/2,\\
        2s-1\quad &\mbox{if } s> 1/2.
    \end{cases}
$$ We need to prove that for every $\epsilon>0$, and every $X_{0}\in \Gamma_{2}^{\rm o}(u(\cdot,t_0))$, there is $\rho>0$ such that
$$\{X\in B_{\rho}(X_{0}):t>t_{0}+|X-X_{0}|^{\beta-2\eps}\}\cap \Gamma_{2}^{\rm o}(u(\cdot, t))=\emptyset\quad\mbox{for every } t\in[0,1].$$
Without loss of generality, we may assume that $X_{0}=0$ and $t_0=0$. We suppose that for some $X_{1}\in B_{\rho}$ with $r:=|X_{1}|<\rho$, there is $t_{1}>r^{\beta-2\eps}$ such that $X_{1}\in \Gamma_{2}^{\rm o}(u(\cdot, t_{1}))$ and we search for a contradiction for $\rho>0$ sufficiently small. We define $$w(X):= r^{-2}\left(u(rX, t_{1})-\widetilde{\vf}^{0}(rX)\right).$$ 
Let $p\in \mathcal{P}_{2}$ be the first blow-up of $\widetilde{u}^{0}$ at $0$. Using the definition of $\Gamma_2^{o}(u(\cdot,0))$ and \cref{lemma:detachement} we get, for $X\in \partial B_{2}$,
\begin{align*}
    w(X)&= r^{-2}\left(u(rX,t_{1})-u(rX,0)\right)+r^{-2}\widetilde{u}^{0}(rX,0)\\
    &\ge c_{\eps}r^{\sigma-2+\eps}t_{1}\chi_{\{|y|\ge 1\}}(X)+p(X)-Cr\\
    &\ge c_{\eps}r^{\sigma+\beta-2-\eps}\chi_{\{|y|\ge 1\}}(X)+p(X)-Cr.
\end{align*}
Next observe that by \eqref{def:soluzione-estesa-tilde} and \eqref{eq:stima-g},
$$-L_{a}w\ge -Cr^{k+\gamma-1-2s}\quad \mbox{on } B_{2}.$$
Then, calling $\psi:B_{2}\to \R$ the $L_{a}$-harmonic function with boundary datum $\chi_{\{|y|\ge1/2\}}$ on $\partial B_2$, since $p$ is $L_{a}$-harmonic, by the maximum principle we get
$$w\ge c_{\eps}r^{\sigma+\beta-2-\eps}\psi+p-Cr-Cr^{k+\gamma-1-2s}\quad \mbox{in } B_{2}.$$ 
Now, by the Harnack inequality, $\psi\ge c>0$ in $B_{3/2}$. Moreover, recall that $p\ge 0$ on $\R^{n+1}_{0}$. Finally, since $\sigma+\beta-2-\eps<1$, $k+\gamma-1-2s>1$, and $r<\rho$, we can find $\rho>0$ sufficiently small so that
$$w>0\quad \mbox{in } B_{3/2}'.$$
From this we deduce that $B_{3r/2}'$ does not contain free boundary points for $u(\cdot,t_{1})$. This, however, contradicts the choice of the point $X_{1}$, and the proof is concluded.
\end{proof}

\subsection{Cleaning of the set $\boldsymbol{\mathbf{\Gammab}_{2+2s}}$}
In this subsection we prove a cleaning result for the set $\Gammab_{2+2s}$. This is obtained similarly to  \cite[Proposition 5.1]{ft23}, using
the polynomial rate of convergence to the blow-up limit given by \cref{prop:rate}.
\begin{proposition}\label{prop:cleaning-(2+2s)}
         Let $u:B_{1}\times [0,1]\to \R$ be a family of solution of \eqref{def:soluzione-estesa}, \eqref{def:soluzione-famiglia} with obstacle $\vf $ satisfying \eqref{e:hypo-phi}, with $k\ge 4$. Then $$\{\Gamma_{2+2s}(u(\cdot,t))\}_{t\in[0,1]}\in\text{Clean}(2+\eta),$$ for some $\eta>0$, which depends on $n$ and $s$.
\end{proposition}
First we prove the following lemma (see  \cite[Lemma 2.2]{ft23} for the case $s=1/2$).
\begin{lemma}\label{lemma:lemma2.2}
    Let $u:B_{1}\times [0,1]\to \R$ be a family of solutions to \eqref{def:soluzione-estesa}, \eqref{def:soluzione-famiglia}, with obstacle $\vf $ satisfying \eqref{e:hypo-phi}. Set $h_{t}:= u(\cdot, t)-u(\cdot, 0)$. Then there is a constant $c>0$ such that, for every $t\in[0,1]$,  $$ h_t\ge ct|y|^{2s} \quad\mbox{in } B_{1/2}.$$
\end{lemma}
\begin{proof}     
    By definition of $h_t$ and the monotonicity assumption in \eqref{def:soluzione-famiglia}, we have that
        $h_t\ge0$ in $ B_1$ and $ h_t\ge t$ in $\partial B_1\cap \{|y|>\frac12\}. $
        Take a function $\psi$ such that 
        $$
    \left\{
    \begin{array}{rclll}
         L_a\psi&=&0 &\quad \mbox{in } B_1\setminus\{y=0\},\\
        \psi&=&1&\quad \mbox{on } \partial B_1\cap \{|y|\ge\frac12\},\\
        \psi&=&0&\quad \mbox{on } \partial B_1\cap \{|y|<\frac12\},\\
        \psi&=&0&\quad \mbox{on} \{y=0\},
    \end{array}
    \right.
    $$
    By the maximum principle in \cref{prop:maximum-principle},  $\psi\le h_{t}/t$ in $B_1$. Moreover, by the Hopf lemma  (\cref{lemma:Hopf}), we know that $|y|^{a}\psi\ge c|y|$ in $B_{1/2}$ for some $c>0$, which concludes the proof.
\end{proof}

Now we are ready to prove \cref{prop:cleaning-(2+2s)}.
\begin{proof}[Proof of \cref{prop:cleaning-(2+2s)}] 
    We need to show that there is $\eta>0$ such that, given $X_{0}\in \Gamma_{2+2s}(u(\cdot, t_{0}))$, we can find $r_{0}>0$ sufficiently small so that
    $$\{X\in B_{r_{0}}(X_{0}):t>t_{0}+|X-X_{0}|^{2+\eta}\}\cap \Gamma_{2+2s}(u(\cdot, t))=\emptyset \quad \mbox{for every } t\in [0,1].$$
    We suppose that for some $X_{0}\in \Gamma_{2+2s}(u(\cdot, t_{0}))$ there is $X_{1}\in \Gamma_{2+2s}(u(\cdot, t_{1}))$ for which $r:=|X_{1}-X_{0}|<r_{0}$ and $t_{1}>t_{0}+r^{2+\eta}$, and we search for a contradiction up to a suitable choice of $\eta, r_{0}>0$. 

    Without loss of generality we may assume that $X_{1}=0$ and $t_{1}=1$. We define $$w(X):= r^{-2-2s}\left(u(rX, t_{0})-\widetilde{\vf}^{0}(rX)\right)$$
    and we observe that, by the scaling law of $L_{a}$, and the definition of the extended obstacle $\widetilde{\vf}^{0}$ in \eqref{def:phitilde},
    $$|L_{a}w|\le Cr^{k+\gamma-2-2s}|y|^{a}\quad \mbox{in } B_{2}\setminus \{w=0\}'.$$
    Using \cref{lemma:lemma2.2}, \cref{prop:rate} and \cref{corollary:classification-(2+2s)}, we obtain the following upper bound for $w$ on $B_{2}$:
    \begin{align*}
        w(X)&= r^{-2-2s}\left(u(rX,t_{0})-u(rX,1)\right)+r^{-2-2s}\left(u(rX,1)-\widetilde{\vf}^{0}\right)\\
        &\le -r^{-2-2s}C(1-t_{0})|ry|^{2s}+\left(by^{2}-x\cdot Ax\right)|y|^{2s}+Cr^{\alpha}\\
        &\le -Cr^{\eta}|y|^{2s}+by^{2+2s}+Cr^{\alpha}.
    \end{align*}
    Here $A$ is a symmetric non-negative definite $n\times n$ matrix and $b>0$ (see \cref{corollary:classification-(2+2s)}), $\alpha\in (0,1)$ is given by \cref{prop:rate}, and we have used that $(1-t_{0})>r^{2+\eta}$.  

    We claim that $w\equiv 0$ on $B'_{3/2}$, provided that $\eta$ and $r_{0}$ are chosen sufficiently small. This would imply that $X_{0}$ is in the interior of $\Lambda(u(\cdot, t_{0}))$, contradicting the fact that $X_{0}\in \Gamma (u(\cdot, t_{0}))$. We will prove the claim with a barrier argument. Let us define, for every $Z\in B_{3/2}'$ and $\delta \ge 0$,
    $$\psi_{Z,\delta}(X):= |X-Z|^{2}-\left(\frac{n}{1+a}+2\right)y^{2}+\delta.$$
    A simple computation shows that
    $$L_{a}\psi = -2(1+a)|y|^{a}\quad \mbox{in } \R^{n+1}_{0}.$$
    Moreover, a suitable choice of $\eta, r_{0}, \rho>0$, independent of $Z$, ensures that, for every $\delta \ge 0$, on $\partial B_{\rho}(Z)$,
    \begin{align*}
        w(X)&\le -Cr^{\eta}|y|^{2s}+by^{2+2s}+Cr^{\alpha}\\
        &\le \rho^{2}- \left(\frac{n}{1+a}+2\right)y^{2}+\delta \le \psi_{Z,\delta}(X),
    \end{align*}
    choosing $\rho=(Cr^{\alpha})^{1/2}$, $\eta<\alpha (1-s)$ and $r_{0}$ sufficiently small. This implies that $w$ and $\psi_{Z,\delta}$ cannot touch each other on $\partial B_{\rho}(Z)$ if $\delta>0$. We now show that they cannot even touch each other in the interior of $B_{\rho}(Z)$, which implies that $w(Z)\le \psi_{Z,\delta}(Z)=\delta$ for every $\delta>0$, thus $w(Z)=0$. In fact, since $\psi_{Z, \delta}\ge\delta>0$ on $\R^{n+1}_{0}$, then $\psi_{Z,\delta}$ and $w$ cannot touch at a point in $\{w=0\}'$. On the other hand, if $X\in B_{\rho}(Z)\setminus \{w=0\}'$ and $r_{0}$ is chosen sufficiently small, then
    $$L_{a}\psi_{Z,\delta}(X)=-2(1+a)|y|^{a}< -Cr^{k+\gamma-2-2s}|y|^{a}\le L_{a}w(X),$$
    which is impossible since $w$ touches $\psi_{Z,\delta}$ from below at $X$.
\end{proof}
\subsection{Other cleaning results} Here we collect some cleaning results for other subsets of the free boundary following from  \cref{prop:cleaning}. 
\begin{proposition}\label{prop:other-cleaning-results} Let $u:B_{1}\times [0,1]\to \R$ be a family of solution to \eqref{def:soluzione-estesa}, \eqref{def:soluzione-famiglia} with obstacle $\vf $ satisfying \eqref{e:hypo-phi}, with $k\ge4$. Then:
\begin{itemize}
    \item [i)] we have 
\be\label{eq:stima1}\{\Gamma_{\ge 3+s}(u(\cdot, t))\setminus \Gamma_{\ge k+\gamma}(u(\cdot, t))\}_{t\in [0,1]}\in\text{Clean}(3-s);\ee
\item [ii)] there exists a constant $\sigma>0$, which depends on $n$ and $s$, such that
\begin{itemize}
\smallskip
    \item [a)] if $s\le1/2$, then \be\label{eq:stima2}\{\Gamma_{*}(u(\cdot,t))\}_{t\in[0,1]}\in\text{Clean}(3-2s+\sigma);\ee
\item [b)] if $s>1/2$, then \be\label{eq:stima2.5}\{\Gamma_{*}(u(\cdot,t))\}_{t\in[0,1]}\in\text{Clean}(2+\sigma);\ee
\end{itemize}
\item [iii)]
finally:
\begin{itemize}
\smallskip
    \item [a)] if $s\le1/2$, then \be\label{eq:stimafinale1}\{\Gamma_{\ge k+\gamma}(u(\cdot,t))\}_{t\in[0,1]}\in\text{Clean}(k+\gamma-1);\ee
\item [b)] if $s>1/2$, then \be\label{eq:stimafinale2}\{\Gamma_{\ge k+\gamma}(u(\cdot,t))\}_{t\in[0,1]}\in\text{Clean}(k+\gamma-2s).\ee
\end{itemize} 
\end{itemize}
\end{proposition}
\begin{proof} The cleaning results \eqref{eq:stima1}, \eqref{eq:stimafinale1} and \eqref{eq:stimafinale2} follow directly from \cref{prop:cleaning}.
To get \eqref{eq:stima2} and \eqref{eq:stima2.5}, we first observe that there is a constant $\sigma>0$ such that $\Gammab_{*}\subset \Gammab_{\ge 3+\sigma}$, for $s\le 1/2$ and $\Gammab_{*}\subset \Gammab_{\ge 2+2s+\sigma}$, for $s> 1/2$, as a consequence of the frequency gaps in \cref{teo:gaps}. Then we conclude by \cref{prop:cleaning} again.
\end{proof}

\section{Proof of generic regularity}\label{section-proofs_main}
In this final section we combine the dimensional bounds of \cref{section-dim_reduc} with the cleaning results of \cref{section-clean} to prove our main results on generic regularity, \cref{thm:mainthm} and \cref{thm:main_esteso}. First,  we prove the following lemma, which gives the implication from \cref{thm:main_esteso} to \cref{thm:mainthm}.
\begin{lemma}\label{lemma:link_thm_main}
     Let $v:\mathbb{R}^{n}\times[0,1]\to \R$ be a solution of \eqref{def:soluzione-frac-famiglia} with obstacle $\varphi$ satisfying \eqref{e:hypo-phi}, and let us assume that $\{\varphi>0\}\ssubset \R^{n}$. Let $u:\R^{n+1}\times[0,1]\to \R$ be such that, for every $t\in [0,1]$, $u(\cdot,t)$ is the standard even $L_a$-harmonic extension of $v(\cdot,t)+t$ in $\R^{n+1}$. Then, up to a positive multiplicative constant, $u$ is a family of solutions to \eqref{def:soluzione-estesa}, \eqref{def:soluzione-famiglia}.
\end{lemma}
\begin{proof}
    First of all, up to re-scaling, we may assume that $\{\vf>0\}\ssubset \{x\in \R^{n}:|x|<1/2\}$. since $v(\cdot, t)$ is a solution of \eqref{def:soluzione-frac-famiglia} with obstacle $\varphi-t$ and vanishing at $\infty$, then $w(\cdot, t):= v(\cdot, t)+t$ is a solution with obstacle $\varphi$ and limit $t$ at $\infty$. In particular $w(\cdot, t)\ge w(\cdot, t')$ if $t\ge t'$, by the maximum principle. Now, $u$ is given by
    $$u(x,y,t)=c_{n,s}\int_{\R^{n}}\frac{|y|^{2s}w(z,t)}{(|x-z|^{2}+|y|^{2})^{\frac{n+2s}{2}}}\,dz.$$
    Hence $u(\cdot, t)\ge u(\cdot, t')$ if $t\ge t'$. In addition, for $(x,y)\in \partial B_{1}\cap \left\{|y|\ge 1/2\right\}$, we have
    \begin{align*}
        u(x,y,t)-u(x,y,t')&=c_{n,s}\int_{\R^{n}}\frac{|y|^{2s}(w(z,t)-w(z,t'))}{(|x-z|^{2}+|y|^{2})^{\frac{n+2s}{2}}}\,dz\\
                &\ge c\int_{\R^{n}\setminus B'_{1}}\frac{(t-t')+(v(z,t)-v(z,t'))}{(1+|z|)^{n+2s}}\,dz.
    \end{align*}
    Observe that $v(\cdot, t)>0$ in $\R^{n}$, provided that $\{\vf>t\}\neq \emptyset$, since $(-\Delta)^{s}v(\cdot, t)\ge 0$,  $v(\cdot, t)\ge \vf-t$ and $v(\cdot, t)$ decays at $\infty$. In particular, in this case  $$\Lambda(v(\cdot, t))\subset \{\varphi>t\} \subset \{\vf>0\}\ssubset \{x\in \R^{n}:|x|<1/2\}.$$
    If instead $\{\vf>t\}=\emptyset$, then clearly $v(\cdot, t)\equiv 0$. 
    As a consequence, the function $h:=v(\cdot,t')-v(\cdot,t)$ satisfies
    $$(-\Delta)^{s}h=0,\quad 0\le h\le t-t'\qquad \mbox{in } \{x\in \R^{n}:|x|>1\},$$
    and also $h(x)\to 0$ as $|x|\to \infty$.
    In particular, we can use the fundamental solution of the $s$-Laplacian as a barrier to obtain a decay rate for $h$, namely $h(x)\le (t-t')|x|^{2s-n}$. Thus $u(x,y,t)-u(x,y,t')\ge c(t-t')$ on $\partial B_{1}\cap \{|y|\ge 1/2\}$, as desired. The uniform bound in $C^{2s}(B_{1})$ for the family $u:B_{1}\times [0,1]\to \R$ comes from the $C^{1,s}$-interior regularity of solutions to \eqref{def:soluzione-estesa}, together with the uniform bounds in $L^{\infty}(B_{2})$ coming from the Poisson formula for $u$. 
\end{proof}
 \begin{proposition}\label{prop:final}
     Let $u$ be a solution to \eqref{def:soluzione-estesa}, \eqref{def:soluzione-famiglia} with obstacle $\vf$ satisfying \eqref{e:hypo-phi}, with $k\ge4$ and $\gamma \in (a^{-},1)$. Let $\pi_2:(X,t)\mapsto t$ be the standard projection. Then there are constants $\eta>0$ and $\sigma>0$, which depend on $n$ and $s$, such that the following holds.
     \begin{enumerate}
         \item If $n=1$, then: 
         \begin{enumerate}[label=(\roman*)]
            \smallskip
            \item $\Gammab_{2}^{\rm o}$ is discrete;
            \smallskip
            \item $\Gammab_{2}^{\rm a}=\emptyset$;
            \smallskip
            \item $\Gammab_{2+2s}=\emptyset$;
            \smallskip
            \item $\Gammab_{\ge 3+s}\setminus \Gammab_{\ge k+\gamma}$ is discrete;
            \smallskip
            \item $\left\{\begin{array}{rclll}
                 \dim_\HH(\pi_2(\Gammab_{\ge k+\gamma}))&\le& 1/(k+\gamma-1)&\quad \mbox{if } s\le1/2;\\
                 \dim_\HH(\pi_2(\Gammab_{\ge k+\gamma}))&\le&1/(k+\gamma-2s) &\quad \mbox{if } s>1/2; 
            \end{array}\right.$
            \smallskip
            \item $\Gammab_{*}=\emptyset$.
        \end{enumerate}
        
        \medskip
        \item If $n=2$, then: 
         \begin{enumerate}[label=(\roman*)]
            \smallskip
            \item $\left\{\begin{array}{rclll}
            \text{dim}_\HH(\pi_2(\Gammab_{2}^{\rm o}))&\le& 1/3&\quad \mbox{if } s\le1/2;\\
            \text{dim}_\HH(\pi_2(\Gammab_{2}^{\rm o}))&\le&1/(4-2s) &\quad \mbox{if } s>1/2; 
            \end{array}\right.$

            \smallskip
            \item $\Gammab_{2}^{\rm a}$ is discrete;
            \smallskip
            \item $\text{dim}_\HH(\pi_2(\Gammab_{2+2s}))\le 1/(2+\eta)$;
           \smallskip
            \item $\text{dim}_\HH(\pi_2(\Gammab_{\ge 3+s}\setminus \Gammab_{\ge k+\gamma}))\le 1/(3-s)$;
            \smallskip
            \item $\left\{\begin{array}{rclll}
            \text{dim}_\HH(\pi_2(\Gammab_{\ge k+\gamma}))&\le& 2/(k+\gamma-1)&\quad \mbox{if } s\le1/2;\\
            \text{dim}_\HH(\pi_2(\Gammab_{\ge k+\gamma}))&\le&2/(k+\gamma-2s) &\quad \mbox{if } s>1/2; 
            \end{array}\right.$
            \smallskip
            \item $\Gammab_{*}$ is discrete.
        \end{enumerate}
        \medskip
         \item If $n=3$, then:
         \begin{enumerate}[label=(\roman*)]
            \smallskip
            \item $\left\{\begin{array}{rclll}
            \text{dim}_\HH(\pi_2(\Gammab_{2}^{\rm o}))&\le& 2/3&\quad \mbox{if } s\le1/2;\\
            \text{dim}_\HH(\pi_2(\Gammab_{2}^{\rm o}))&\le&2/(4-2s) &\quad \mbox{if } s>1/2; 
            \end{array}\right.$
            \smallskip
            \item $ \left\{\begin{array}{rclll}\text{dim}_\HH(\pi_2(\Gammab_{2}^{\rm a}))&\le&1/2 &\quad\mbox{if } s\le 1/2;\\
        \text{dim}_\HH(\pi_2(\Gammab_{2}^{\rm a}))&\le&(1+s)/(4-2s) &\quad \mbox{if } s> 1/2;
            \end{array}\right.$
            \smallskip
            \item $\text{dim}_\HH(\pi_2(\Gammab_{2+2s}))\le 2/(2+\eta)$;
            \smallskip
            \item $\text{dim}_\HH(\pi_2(\Gammab_{\ge 3+s}\setminus \Gammab_{\ge k+\gamma}))\le 2/(3-s)$;
            \smallskip
            \item $ \left\{\begin{array}{rclll} \text{dim}_\HH(\pi_2(\Gammab_{\ge k+\gamma}))&\le& 3/(k+\gamma-1) &\quad\mbox{if } s\le 1/2;\\
            \text{dim}_\HH(\pi_2(\Gammab_{\ge k+\gamma}))&\le& 3/(k+\gamma-2s) & \quad\mbox{if } s> 1/2;
            \end{array}\right.$
            \smallskip
            \item $ \left\{\begin{array}{rclll} \text{dim}_\HH(\pi_2(\Gammab_{*}))&\le& 1/(3-2s+\sigma)&\quad\mbox{if } s\le 1/2;\\
            \text{dim}_\HH(\pi_2(\Gammab_{*}))&\le& 1/(2+\sigma)&\quad\mbox{if } s> 1/2;\\
            \end{array}\right.$
        \end{enumerate}

        \medskip
        \item If $n\ge4$, then, for almost every $t\in[0,1]$:
         \begin{enumerate}[label=(\roman*)]
            \smallskip
            \item $ \left\{\begin{array}{rclll} \text{dim}_\HH(\Gamma_{2}^{\rm o}(u(\cdot,t)))&\le& n-4&\quad\mbox{if } s\le 1/2;\\
            \text{dim}_\HH(\Gamma_{2}^{\rm o}(u(\cdot,t)))&\le& n-5+2s &\quad \mbox{if } s> 1/2;
            \end{array}\right.$
            \smallskip
            \item $ \left\{\begin{array}{rclll} \text{dim}_\HH(\Gamma_{2}^{\rm a}(u(\cdot,t)))&\le& n-4&\quad\mbox{if } s\le 1/2;\\
        \text{dim}_\HH(\Gamma_{2}^{\rm a}(u(\cdot,t)))&\le& n-6/(1+s) &\quad \mbox{if } s> 1/2;
            \end{array}\right.$
            \smallskip
            \item $\text{dim}_\HH(\Gamma_{2+2s}(u(\cdot,t)))\le n-3-\eta$;
            \smallskip
            \item $\text{dim}_\HH(\Gamma_{\ge 3+s}(u(\cdot,t))\setminus\Gamma_{\ge k+\gamma}(u(\cdot,t)))\le n-4+s$;
            \smallskip
            \item $ \left\{\begin{array}{rclll} \text{dim}_\HH(\Gamma_{\ge k+\gamma}(u(\cdot,t)))&\le& n-k-\gamma+1&\quad\mbox{if } s\le 1/2,\ k+\gamma-1\le n;\\
            \text{dim}_\HH(\pi_2(\Gammab_{\ge k+\gamma}))&\le& n/(k+\gamma-1)&\quad\mbox{if } s\le 1/2,\ k+\gamma-1> n;\\
            \text{dim}_\HH(\Gamma_{\ge k+\gamma}(u(\cdot,t)))&\le& n-k-\gamma+2s &\quad \mbox{if } s> 1/2, \ k+\gamma-2s\le n;\\
            \text{dim}_\HH(\pi_2(\Gammab_{\ge k+\gamma}))&\le& n/(k+\gamma-2s) &\quad \mbox{if } s> 1/2, \ k+\gamma-2s>n;
            \end{array}\right.$
            \smallskip
            \item $ \left\{\begin{array}{rclll} 
        \text{dim}_\HH(\pi_2(\Gammab_*))&\le& 2/(3-2s+\sigma)&\quad\mbox{if } s\le 1/2, \ n=4,\\
            \text{dim}_\HH(\Gamma_{*}(u(\cdot,t)))&\le& n-5+2s-\sigma&\quad\mbox{if } s\le 1/2, \ n\ge5,\\
        \text{dim}_\HH(\pi_2(\Gammab_*))&\le& 2/(2+\sigma)&\quad\mbox{if } s>1/2, \ n=4,\\
            \text{dim}_\HH(\Gamma_{*}(u(\cdot,t)))&\le& n-4-\sigma&\quad\mbox{if } s>1/2, \ n\ge5,
            \end{array}\right.$
        \end{enumerate}
     \end{enumerate}
 \end{proposition}
 \begin{proof}
     It is enough to apply \cref{lemma:gmt} on each subset of the free boundary combined with the results in \cref{prop:dimension-reduction-ge2e*}, \cref{prop:gamma-anomalous}, \cref{prop:gamma-ordinary} \cref{prop:cleaning-(2+2s)} and \cref{prop:other-cleaning-results}.
 \end{proof}
 \begin{proof}[Proof of \cref{thm:main_esteso}]
     It is a direct consequence of \cref{prop:final}, after splitting the degenerate set as in \eqref{splitting-degenerate-set}.
 \end{proof}
 \begin{proof}[Proof of \cref{thm:mainthm}]
     It descends immediately from \cref{thm:main_esteso} via \cref{lemma:link_thm_main}.
 \end{proof}

\addappendix
\subsection{Some useful facts about $\boldsymbol{L_a}$}
Here we collect some useful properties of the operator $L_{a}$.
We start from the following interior regularity estimates which were obtained in \cite{fks82} (see also \cite{css08}).

\begin{proposition}[Harnack Inequality, \protect{\cite[Proposition 2.2]{css08}}]\label{lemma:harnack}
    Let $u:B_{r}\to \R$ be a solution of $L_{a}u=0$ in $B_{r}$. If $u\ge 0$ in $B_{r}$, then
    $$\underset{B_{r/2}}{\sup}u \le C\underset{B_{r/2}}{\inf}u.$$
\end{proposition}
\begin{proposition}[Interior estimates, \protect{\cite[Section 2]{css08}}]
\label{lemma:internal-estimates}
Let $u:B_{r}\to \R$ be a solution of $L_{a}u=0$ in $B_{r}$. Then, the following estimates hold:
\begin{gather*}
    \underset{B_{r/2}}{\sup} \big|\nabla_{x}^{k}u\big|\le \frac{C}{r^{k}}\underset{B_{r}}{\osc}u\quad \mbox{for every } k\in \mathbb{N};\\
    [ \nabla_{x}^{k}u]_{C^{\alpha}(B_{r/2})}\le \frac{C}{r^{k+\alpha}}\underset{B_{r}}{\osc}u\quad \mbox{for some } \alpha \in (0,1) \mbox{ and every } k\in \mathbb{N};\\
    \underset{B_{r/2}}{\sup} \big|\partial^{2}_{yy}u+\frac{a}{y}\partial_{y}u\big|\le \frac{C}{r^{2}}\underset{B_{r}}{\osc}u. 
\end{gather*}
\end{proposition}
Next, we recall the following property of odd extensions, which can be easily obtained integrating by parts.
\begin{proposition}\label{prop:odd-extension}
    Let $u\in W^{1,2}(B_{1}^+,|y|^{a})\cap C\left(\overline{B_{1}^{+}}\right)$ be a solution of 
    \begin{equation*}
    \left\{\begin{array}{rclll}
        L_{a}u&=&0&\quad \mbox{in } B_{1}^+,\\
        u&=&0&\quad \mbox{on } B_{1}'.\\
    \end{array}\right.
    \end{equation*}
    Then, the odd extension $\widetilde{u}$ in $B_{1}$, defined as $\widetilde{u}(x,y):= (\sgn{y})u(x,(\sgn{y})y)$, belongs to $W^{1,2}(B_{1},|y|^{a})$ and solves $L_{a}\widetilde{u}=0$ in $B_{1}$.
\end{proposition}
\begin{proposition}[Liouville-type results]\label{prop:Liouville}
    Let $u:\R^{n+1}\to \R$ be a solution of $L_{a}u=0$ in $\R^{n+1}$, and suppose that
    \begin{equation*}
        |u(X)|\le C(1+|X|^{k})\quad \mbox{for some } k\in \mathbb{N}.
    \end{equation*}
    Then the following holds.
    \begin{itemize}
        \item [i)] If $u$ is even in $y$, then $u$ is a polynomial.
        \smallskip
        \item [ii)] If $u$ is odd in $y$, then there exists a polynomial $p$, even in $y$, for which $u(x,y)=(\sgn y) |y|^{2s}p(x,y)$.
    \end{itemize}
\end{proposition}
\begin{proof}
    The first point is \cite[Lemma 2.7]{css08}. Let us prove the second point by induction on $k$. If $k=1$, then by \cref{lemma:internal-estimates}, $\partial^{2}_{yy}u+(a/y)\partial_{y}u \equiv 0$ in $\R^{n+1}$. Integrating this ODE in $y$ at each fixed $x\in \R^{n}$, we deduce that there exist two functions $b, c: \R^{n} \to \R$ for which $u(x,y)= b(x)y|y|^{-a}+c(x)$. Since $u$ is odd in $y$, we must have $c\equiv 0$. Moreover, again by \cref{lemma:internal-estimates}, $\nabla^{2}b\equiv 0$, which means that $b$ is a linear function, and we are done. Let us now assume that $k\ge 2$ and that the thesis is true for every natural number less than $k$. Since, 
    $$\partial^{2}_{yy}u+\frac{a}{y}\partial_{y}u=-\Delta_{x}u,$$
    and $\Delta_{x}$ commutes with $L_{a}$, the function $v:=\partial^{2}_{yy}u+(a/y)\partial_{y}u$ solves $L_{a}v=0$ in $\R^{n+1}$ and its growth is controlled by  
    $$|v(X)|\le C(1+|X|^{k-2}).$$
    from the inductive hypothesis, we deduce that $v(x,y)=y|y|^{-a}q(x,y)$, for some polynomial $q$, even in $y$. Multiplying both sides with $|y|^{a}$ we get
    $$\partial_{y}(|y|^{a}\partial_{y}u)=yq(x,y),$$
    and integrating in $y$ we obtain $\partial_{y}u=|y|^{-a}\bar{q}(x,y)$ for some polynomial $\bar{q}$, even in $y$. Integrating once more we get the thesis.
\end{proof}
\begin{lemma}[Extension of polynomials, \protect{\cite[Lemma 5.2]{gr19}}]
\label{lemma:unique-extension-polynomial-aHarmonic}
    Let $q:\R^{n} \to \R$ be an homogeneous polynomial of degree $k$. Then there is a unique homogeneous polynomial $\widetilde{q}:\R^{n+1}\to \R$ of degree $k$ such that
    \begin{equation*}
    \left\{
        \begin{array}{rclll}
            L_{a}\widetilde{q}&=&0&\quad \mbox{in } \R^{n+1},\\
            \widetilde{q}(\cdot,0)&=&q&\quad \mbox{in } \R^{n},\\
            \widetilde{q}(x,-y)&=&\widetilde{q}(x,y)&\quad \mbox{for } (x,y)\in \R^{n+1}.
        \end{array}\right.
    \end{equation*}
    In particular, an homogeneous polynomial of degree $k$, which is $L_a$-harmonic, even in $y$, and vanishes on the thin space, must vanish identically.
\end{lemma}

\begin{proposition}[Maximum principle, \protect{\cite[Remark 4.2]{cs14}}] \label{prop:maximum-principle}
    Let $u\in W^{1,2}(B_{r},|y|^{a})\cap C\left(\overline{B_{r}}\right)$ be such that
    $$\left\{\begin{array}{rclll}
        -L_au&\ge& 0 &\quad\mbox{in } B_{r}\setminus\{y=0\},\\
        -\lim_{y\downarrow 0}\y\partial_y u(\cdot,y)&\ge&0 &\quad\mbox{in } \{x\in \R^{n}: |x|<r\},\\
        u&\ge&0 &\quad\mbox{on } \partial B_{r}.
    \end{array}\right.$$ Then $u\ge 0$ in $B_{r}$. Moreover, either $u\equiv 0$ or $u>0$ in $B_{r}\setminus\{y=0\}$.
\end{proposition}

\begin{lemma}[Hopf Lemma, \protect{\cite[Proposition 4.11]{cs14}}]\label{lemma:Hopf}
    Let $u\in W^{1,2}(B_{r}^{+},|y|^{a})\cap C\left(\overline{B_{r}^{+}}\right)$ be such that
    \begin{equation*}
    \left\{
        \begin{array}{rclll}
            -L_{a}u&\ge& 0&\quad \mbox{in } B_{r}^{+},\\
            u&>&0&\quad \mbox{in } B_{r}^{+},\\
            u(0)&=&0.
        \end{array}\right.
    \end{equation*}
    Then $\liminf_{y\downarrow 0}y^{a}\frac{u(0,y)}{y}>0$. In addition, if $y^{a}\partial_{y}u \in C\left(\overline{B_{r}^{+}}\right)$, then $\lim_{y\downarrow 0}y^{a}\partial_{y}u(0,y)>0.$  
\end{lemma}
\subsection{The very thin obstacle problem}
Next, we recall some results related to the very thin obstacle problem \eqref{def:very-thin-obstacle-problem} (see \cite[Section 8]{fj21} for more details).
\begin{proposition}[\protect{\cite[Lemma 8.1]{fj21}}]\label{prop:operator-La-forverythinobstacle}
    Let $L=\{x_{n}=y=0\}$ and $u\in W^{1,2}(B_{1},|y|^{a})\cap C(B_{1})$ be such that $L_{a}u=0$ in $\R^{n+1}\setminus L$. Then 
    \begin{equation*}
        L_{a}u=f_{a}[u](x')\mathcal{H}^{n-1}\res L,
    \end{equation*}
    where $f_{a}[u]:\R^{n-1}\to \R$ is defined as
    \begin{equation}\label{def:fa}
        f_{a}[u](x'):=\underset{\eps \downarrow 0}{\lim}\int_{\partial D_{\eps}}u_{\nu}(x',x_{n},y)|y|^{a}\,d\mathcal{H}^{1}.
    \end{equation}
    Here, $$D_{\eps}:= \{(x_{n},y):x_{n}^{2}+y^{2}<\eps^{2}\},\qquad u_{\nu}(x',x_{n},y):=\nabla_{x_n,y}u(x',x_{n},y)\cdot \nu\quad\mbox{on } \partial D_{\eps},$$ and $\nu=(x_n,y)/\eps$ is the unit normal to $\partial D_{\eps}$.
\end{proposition}
\begin{proposition}[Hopf lemma in the very thin case]\label{hopf-verythincase}
    Let $L=\{x_{n}=y=0\}$ and $u\in W^{1,2}(B_{1},|y|^{a})\cap C(B_{1})$ be such that 
    \begin{equation*}
    \left\{
        \begin{array}{rclll}
            -L_{a}u&\ge& 0&\quad \mbox{in } B_{1}\setminus L,\\
            u&>&0&\quad \mbox{in } B_{1}\setminus L,\\
            u(0)&=&0.
        \end{array}\right.
    \end{equation*}
    Then
    $\liminf_{\eps \downarrow 0}\int_{\partial D_{\eps}}\frac{u(0,x_{n},y)}{\eps}|y|^{a}\,d\mathcal{H}^{1}>0.$
    In addition, if $\eps\mapsto \int_{\partial D_{\eps}}u_{\nu}(0,x_{n},y)|y|^{a}\,d\mathcal{H}^{1}$ is well-defined and has a limit as $\eps \downarrow 0$, then $f_{a}[u](0)>0$, where $f_{a}[u]$ is defined in \eqref{def:fa}.
\end{proposition}
\begin{proof}
    We will use a barrier argument, similar to the one in \cite[Proposition 4.11]{cs14}. Consider the open set $$\mathcal{C}_{r}:= \{(x',x_{n},y)\in \R^{n+1}:\  |x'|<1/2,\  0<x_{n}^{2}+y^{2}<r^{2}\}.$$
    Given a constant $A>0$ to be chosen later, we define the barrier
    $$w(x',x_{n},y):= |(x_{n},y)|^{-a}(1+Ay^{2})\phi(x'),$$
    where $\phi:\{x'\in \R^{n-1}: |x'|<1/2\}\to \R$ is the first eigenfunction of the Laplacian with zero Dirichlet boundary conditions in $\{x'\in \R^{n-1}: |x'|<1/2\}$, that is to say
    $$
    \left\{
    \begin{array}{rclll}
          -\Delta_{x'}\phi&=&\lambda_{1}\phi\quad &\mbox{in } \{x'\in \R^{n-1}: |x'|<1/2\},\\
        \phi&=&0\quad &\mbox{on } \{x'\in \R^{n-1}: |x'|=1/2\},\\
        \phi(0)&>&0.
    \end{array}
    \right.
    $$
    Here $\lambda_{1}>0$ is the first eigenvalue, and $\phi>0$ on $\{x'\in \R^{n-1}: |x'|<1/2\}$. We may compute, in $\mathcal{C}_{r}$,
    \begin{align*}
        L_{a}w&= |y|^{a}|(x_{n}, y)|^{-a}\left\{A(2+2a)-\lambda_{1}(1+Ay^{2})-4aA\frac{y^{2}}{x_{n}^{2}+y^{2}}\right\}\phi(x')\\
        &\ge |y|^{a}|(x_{n}, y)|^{-a}\left\{A(2-2a)-\lambda_{1}(1+Ay^{2})\right\}\phi(x').
    \end{align*}
    Hence, up to choosing $A$ large and $r$ small, we get that $w$ is $L_{a}$-subharmonic in $\mathcal{C}_{r}$. Notice also that $$w=0 \quad \mbox{on }\{|x'|=1/2, \ x_{n}^{2}+y^{2}<r^{2}\}\cup\{x_{n}=y=0\}.$$
    In particular, by the positivity and continuity of $u$ out of $\{x_n=y=0\}$, there exists $\delta>0$ small enough so that $$u\ge \delta w \quad \mbox{on } \partial \mathcal{C}_{r}.$$
    Then, by the comparison principle, $u\ge \delta w$ in $\mathcal{C}_{r}$, thus
    \begin{align*}
        \underset{\eps \downarrow 0}{\liminf}\int_{\partial D_\eps}\frac{u(0,x_{n},y)}{\eps}|y|^{a}\,d\mathcal{H}^{1}&\ge \underset{\eps \downarrow 0}{\liminf}\int_{\partial D_\eps}\delta\frac{w(0,x_{n},y)}{\eps}|y|^{a}\,d\mathcal{H}^{1}\\
    &\ge\underset{\eps \downarrow 0}{\liminf}\fint_{\partial D_\eps}c\delta\phi(0)\frac{|y|^{a}}{|(x_{n},y)|^{a}}\,d\mathcal{H}^{1}>0,
    \end{align*} which concludes the proof.
\end{proof}
\begin{proposition}[\protect{\cite[Lemma 8.18]{fj21}}]\label{prop:classification-verythinobstacle-dim2} Let $u:\R^{2}\to \R$ be a $\kappa$-homogeneous non-zero solution of the very thin obstacle problem \eqref{def:very-thin-obstacle-problem} in $\R^2$. Then $\kappa \in \{-a\}\cup \mathbb{N}_{\ge 1}$.
\end{proposition}

\bigskip
\subsection*{Acknowledgements} 
  We warmly thank Xavier Fern\'andez-Real for proposing this problem and for all the useful subsequent discussions. We also thank him, together with Bozhidar Velichkov, for the many valuable comments on a preliminary version of the paper. 
  Finally, we would like to thank Giorgio Tortone for the useful discussions on the existence literature about the degenerate/singular elliptic operators.

 M.C. is supported by the European Research Council (ERC), through the European Union's Horizon 2020 project ERC VAREG - \it Variational approach to the regularity of the free boundaries \rm (grant agreement No. 853404). M.C. also acknowledge the MIUR Excellence Department Project awarded to the Department of Mathematics, University of Pisa, CUP I57G22000700001. 

R.C. is supported by the Swiss National Science Foundation (SNF grant
PZ00P2\_208930), by the Swiss State Secretariat for Education, Research and Innovation (SERI) under
contract number MB22.00034. 

\printbibliography
\end{document}